\documentclass[11pt,reqno]{amsart}
\usepackage{todonotes}

\title[the one-phase Muskat problem]{Global  well-posedness for the one-phase Muskat problem}

\author{Hongjie Dong}
\address{Division of Applied Mathematics, Brown University, Providence, RI 02912}
\email{hongjie$\_$dong@brown.edu }

\author{Francisco Gancedo}
\address{Departamento de An\'alisis Matem\'atico $\&$ IMUS, Universidad de Sevilla, Sevilla, Spain}
\email{fgancedo@us.es}

\author{Huy Q. Nguyen}
\address{Department of Mathematics, Brown University, Providence, RI 02912}
\email{hnguyen@math.brown.edu}

\usepackage[margin=1in]{geometry}
\usepackage{amsmath, amsthm, amssymb, mathrsfs, stmaryrd}
\usepackage{times}
\usepackage{color}
\usepackage{hyperref}

\newcommand{\bq}{\begin{equation}}
\newcommand{\eq}{\end{equation}}
\newcommand{\bqa}{\begin{eqnarray*}}
\newcommand{\eqa}{\end{eqnarray*}}

\theoremstyle{plain}
\newtheorem{theo}{Theorem}[section]
\newtheorem{prop}[theo]{Proposition}
\newtheorem{lemm}[theo]{Lemma}
\newtheorem{coro}[theo]{Corollary}
\newtheorem{assu}[theo]{Assumption}
\newtheorem{defi}[theo]{Definition}
\theoremstyle{definition}
\newtheorem{rema}[theo]{Remark}
\newtheorem{nota}[theo]{Notation}

\DeclareMathOperator{\Lip}{Lip}
\DeclareMathOperator{\Tr}{Tr}
\DeclareSymbolFont{pletters}{OT1}{cmr}{m}{sl}
\DeclareMathSymbol{s}{\mathalpha}{pletters}{`s}

\def\tt{\theta}
\def\eps{\varepsilon}
\def\na{\nabla}

\def\mez{\frac{1}{2}}
\def\tdm{\frac{3}{2}}

\def\Rr{\mathbb{R}}
\def\T{\mathbb{T}}
\def\Nn{\mathbb{N}}
\def\Zz{\mathbb{Z}}

\def\cF{\mathcal{F}}
\def\cK{\mathcal{K}}
\def\cN{\mathcal{N}}
\def\cO{\mathcal{O}}

\def\p{\partial}
\def\na{\nabla}
\def\wc{\rightharpoonup}
\def\ka{\kappa}

\def\wt{\widetilde}
\def\wh{\widehat}

\def\ol{\overline}

\def\vp{\varphi}

\numberwithin{equation}{section}

\date{today}
\begin{document}
\begin{abstract}
The free boundary problem for a two-dimensional fluid filtered in porous media is studied. This is known as the one-phase Muskat problem and is mathematically equivalent to the vertical Hele-Shaw problem driven by gravity force. We prove that if the initial free boundary is the graph of a periodic Lipschitz function, then there exists a global-in-time Lipschitz solution in the strong $L^\infty_t L^2_x$ sense and it is the unique viscosity solution. The proof requires quantitative estimates for layer potentials and pointwise elliptic regularity in Lipschitz domains. This is the first construction of  unique global strong solutions for the Muskat problem with initial data of arbitrary size.
\end{abstract}


\noindent\thanks{\em{ MSC Classification:  .}}

\maketitle
\section{Introduction}

In this paper we study the dynamics of a two dimensional incompressible fluid permeating a homogeneous porous medium. This physical phenomena is modeled by the classical Darcy law \cite{Darcy1856}
\bq\label{Darcy}
\mu u(x, y, t)=-\na_{x, y}p(x, y, t)-\rho(0, 1),\quad \na_{x, y}\cdot u(x, y, t)=0,\quad (x, y)\in \Omega_t\subset\Rr^2,~t\in \Rr_+.
\eq
Here $u$ is the fluid velocity,  $p$ is the fluid pressure, and the positive constants $\mu$ and $\rho$ are respectively the dynamic viscosity and  fluid density.  The permeability of the media and the gravitational acceleration have been normalized for the sake of notational simplicity. This problem is known as the one-phase Muskat problem \cite{Muskat1937book} and, interestingly,  is mathematically equivalent to the vertical Hele-Shaw problem driven by gravity. It is a fundamental parabolic free boundary problem in fluid mechanics. We should point out that the horizontal Hele-Shaw problem driven by fluid injection or suction has a different nature and has been studied extensively.

 The fluid occupies a time-dependent domain $\Omega_t\subset \Rr^2$ whose boundary $\Sigma_t$ moves with the fluid
 \bq\label{V:boundary}
 \mathcal{V}(\Sigma_t)=u\cdot n \quad\text{on } \Sigma_t=\partial \Omega_t,
 \eq
where $n$ is the outward pointing unit normal to $\Sigma_t$. Neglecting surface tension, the pressure is continuous across the free boundary
\bq\label{bc:p}
p\vert_{\Sigma_t}=0.
\eq
One is interested in the {\it geometry} and {\it regularity} of the free boundary $\Sigma_t$ as time evolves. Regarding the geometry, there are two distinguished cases: graph and non-graph boundary. For both kinds of geometry, the existence and uniqueness of strong solution for a {\it short time} have been established  even for much more general settings, including multi-phase, with rigid boundaries, with surface tension, nonconstant permeability.  See e.g. \cite{DuchonRobert1984, Chen1993, EscherSimonett1997} and more recent developments in \cite{CordobaGancedo2007, Ambrose2007, CCG2011, CCG2013, Ambrose2014, CG-BS2016, Matioc2019, AlazardLazard2019, AlazardOneFluid2019, NguyenPausader2019, HQNguyen2019, FlynnNguyen2020, AlazardHung22020, AlazardHung32020}.

In this paper we are concerned with long-term dynamics of the Muskat problem. Global existence and uniqueness of solutions have been obtained either when the initial free boundary is the  graph of a small function (in certain function spaces) \cite{SCH2004, CordobaGancedo2007, EscherMatioc2011, CG-BS2016, CGSV2017} or is close to a circle \cite{Chen1993, ConstantinPugh1993,GG-JPS2019Bubble}. There are many recent developments for the former in the direction of low regularity solutions. These include ``medium data''  in the Wiener algebra \cite{CCGS2013, CCGR-PS2016, GG-JPS2019, GG-BS2020}, ``medium data'' in the Lipschitz norm \cite{Cameron2019, Cameron2020}, small data in critical Sobolev spaces allowing for arbitrarily large slopes \cite{CordobaLazar2018, GancedoLazar2020, AlazardHung22020} and even infinite slope \cite{AlazardHung32020}. We note that these norms are scaling invariant for the setting considered.

The geometry has an interesting interplay with the regularity. It is known that for the two-phase problem (two fluids separated by a free interface), some initial graph interfaces with large slopes can turn over, passing from a stable regime to an unstable one \cite{CCFGL-F2012}.  Subsequently the solutions lose regularity in finite time \cite{CCFG2013}. On the other hand, there exist solutions that shift stability regime from stable to unstable and then return back to stable \cite{CG-SZ2017}. Particles on the free boundary cannot collide along a smooth curve (splat singularity) for both the two-phase \cite{CordobaGancedo2010} and one-phase problems \cite{CordobaPernas-Castano2017}. However, the one-phase problem is more singular in the sense that it can develop splash singularity \cite{CCFG2016} from some {\it non-graph} initial boundary, while the two-phase problem cannot \cite{GancedoStrain2014}. Here, splash singularity occurs when two particles collide at a single point on the free boundary while the boundary remains regular. The  remaining less well-understood scenario is the one-phase problem starting from {\it graph initial boundaries}.  Then, it is known that in stark contrast to its two-phase counterpart, the free boundary of the one-phase problem  cannot turn over. Two fundamental questions for this scenario are: 1) Does there exist a unique global solution? 2) If yes, what is its long-term regularity? 

In this paper we affirmatively answer the first question. More precisely, we prove that if the initial free boundary is the graph of a periodic Lipschitz function, then there exists a global Lipschitz solution in the strong $L^\infty_t L^2_x$ sense, and hence almost everywhere. Moreover, it is the unique viscosity solution \cite{CrandallLions}. In order to establish this result, we first show that the graph free boundary of the one-phase Muskat problem obeys a nonlinear integro-differential equation.

\begin{prop}\label{prop:reform}
Assume that the fluid domain is given by
\[
\Omega_t=\{(x, y)\in \Rr^2,\quad y<f(x, t)\}
\]
for some  function $f(x, t): \Rr\times [0, T]\to \Rr$ that is $2\pi$ periodic in $x$. Then, $f$ obeys the equation
\bq\label{reform:f}
\p_t f=-\ka G(f)f,\quad \ka=\frac{\rho}{\mu}.
\eq
Here, for $f, g: \T\equiv \Rr/2\pi\Zz\to \Rr$ the Dirichlet-Neumann operator $G(f)g$ is given by
\bq
\label{reform:Gfg}
G(f)g=\frac{1}{4\pi}p.v.\int_{\T}\frac{\sin(x\!-\!x')\!+\!\sinh(f(x)\!-\!f(x'))\p_xf(x)}{\cosh(f(x)\!-\!f(x'))\!-\!\cos(x\!-\!x')}\tt(x')dx',
\eq
where $\tt:\T\to \Rr$ satisfies
\bq\label{reform:tt}
\mez \tt(x)+\frac{1}{2\pi} p.v.\int_{\T}\frac{\sinh(f(x)\!-\!f(x'))\!-\!\sin(x\!-\!x')\p_xf(x)}{\cosh(f(x)\!-\!f(x'))\!-\!\cos(x\!-\!x')}\tt(x')dx'=\partial_x g(x).
\eq
\end{prop}
We shall prove using deep results from  layer potential theory for $C^1$ and Lipschitz domains \cite{Fabes1978, Verchota1984} that $G(f)g$ is well-defined in $L^2(\T)$ with a quantitative bound, provided that $f\in W^{1, \infty}(\T)$ and $g\in H^1(\T)$. Our main result is stated as follows.
\begin{theo}\label{theo:main}
For all  $f_0\in W^{1, \infty}(\T)$, there exists
\[
f\in C(\T\times [0, \infty))\cap L^\infty([0, \infty); W^{1, \infty}(\T)),\quad \p_t f\in L^\infty([0, \infty); L^2(\T))
\]
such that  $f\vert_{t=0}=f_0$, $f$ satisfies \eqref{reform:f} in $L^\infty_tL^2_x$, and
\[
\| f(t)\|_{W^{1, \infty}(\T)}\le \| f_0\|_{W^{1, \infty}(\T)}\quad\text{a.e. } t>0.
\]
Moreover, $f$ is the unique viscosity solution of \eqref{reform:f} and is stable in $L^\infty(\T)$.
\end{theo}
\begin{rema}
We shall prove in Theorem \ref{theo:comparison:viscosity} that viscosity solutions of \eqref{reform:f} obey the comparison principle. Consequently, every modulus of continuity of $f_0$ is preserved by $f(t)$ for all $t> 0$.
\end{rema}
It turns out that  sufficiently smooth solutions of the one-phase Muskat problem obey the comparison principle. See  Proposition \ref{comparison:Muskatr} and  \cite{AlazardOneFluid2019}. As a consequence, so long as the free boundary remains to be a graph, its slope is bounded by the initial slope. Thus starting from a Lipschitz graph, any sufficiently smooth solution must be a graph.

 The comparison principle has been discovered for other free boundary problems including the  Stefan problem \cite{ACS1996, Kim2003, Kim2011} and the horizontal Hele-Shaw problem \cite{Caffarelli3, Kim2003,  Guillen2019}. Motivated by the fact that these models may develop  singularities in finite time,  the notion of viscosity solutions, introduced by Crandall and Lions \cite{CrandallLions}, has been employed to construct global  solutions past singularities. In the aforementioned works,  viscosity solutions for the unknown scalars  are merely continuous and so are their zero level sets--the free boundary. Consequently, the dynamics of the free boundary is not satisfied in the classical sense (pointwise). This is also the case for variational weak solutions constructed in \cite{Elliott1981, Kamenomostskaja1961}.  Higher regularity of weak (viscosity or variational) solutions  to these problems has been studied in \cite{ACS1996, ChoiJerisonKim2007, ChoiJerisonKim2009}. In particular, for the horizontal one-phase Hele-Shaw problem, the authors in \cite{ChoiJerisonKim2007} proved that if the initial Lipschitz norm of the free boundary is small, then for a short time the unique viscosity solution satisfies the equations pointwise. We also mention that by using the methodology of  convex integration, non-unique weak solutions have been constructed in the purely unstable \cite{Forster2018, CFM2019, Mengual2020} and partially unstable scenarios \cite{CFM2021}.  

As far as the Muskat problem is concerned, to the best of our knowledge there has not been any  global well-posedness result for initial data of arbitrary size, either for weak or strong solutions. We should point out that  for the two-phase problem, the authors in \cite{Lin2017}  proved the existence (without uniqueness) of global weak solutions that are monotone in $\Rr$. 
 In order to obtain viscosity solutions that satisfy the equation almost everywhere, one cannot appeal to the Perron method as in the aforementioned works. Instead, we construct solutions by the vanishing viscosity approach. More precisely, for small $\eps>0$ we consider the approximate equation
\bq\label{intro:eqapp}
\p_tf^\eps=-\ka G(f^\eps)f^\eps+\eps \p_x^2 f^\eps.
\eq
The added viscous term $\eps \p_x^2 f^\eps$ retains the comparison principle for smooth solutions. The first difficulty is to establish  global regularity for \eqref{intro:eqapp} with large data.  This is done in Proposition \ref{GlobalEpsilonProblem} via the layer potential representation \eqref{reform:Gfg}-\eqref{reform:tt} of the Dirichlet-Neumann operator $G(f)g$ and a
careful decomposition of its singular and non-singular parts. An important ingredient  in its proof is the $L^2$ bound
\bq\label{intro:L2DN}
\| G(f)g\|_{L^2(\T)}\le C(1+\| f\|_{\Lip(\T)})^2\|\partial_xg\|_{L^2(\T)},
\eq
where the slope $\| f\|_{\Lip(\T)}$ is controlled for all time. The proof of \eqref{intro:L2DN} relies on a quantitative bound for the inverse $(\mez I-K^*[f])^{-1}$, where $K[f]$ is the boundary double layer potential for the fluid domain  $\{y<f(x)\}$. This is established in Section \ref{Section:potential}. Under the $C^{2, \alpha}$ regularity condition for the boundary, a similar bound was obtained \cite{CCG2011}. The quantitative bound is also crucial to the uniform $L^2$ bound for the solution $\tt^\eps$ of \eqref{reform:tt} with $f=f^\eps$. Finally, from the uniform Lipschitz bound for $f^\eps$ and the uniform $L^2$ bound for $\tt^\eps$, we are able to pass to the limit $\eps \to 0$ in  equations \eqref{reform:f}-\eqref{reform:Gfg} and \eqref{reform:tt} using a decomposition into small and large scales inspired by \cite{CCGS2013, CCGR-PS2016}. We thus obtain a global Lipschitz  solution in the strong $L^\infty_t L^2_x$ sense.

An advantage of the viscosity regularization \eqref{intro:eqapp} is that one can show, in a direct manner, that the constructed solution $f$ is also a viscosity solution (see Definition \ref{def:viscosity}). The proof of this makes use of the contraction estimate for the Dirichlet-Neumann operator associated to two different domains \cite{NguyenPausader2019}. In order to obtain the uniqueness of viscosity solutions we prove that they obey the comparison principle. By the sup and inf-convolution technique, this reduces to proving the consistency of viscosity solutions. That is, if a viscosity solution is smooth ($C^{1,1}$) at a point $(x_0, t_0)$ then it satisfies equation \eqref{reform:f} classically at the same point.  It is known that the harmonic extension $\phi$ of $f$ to the domain $\{y<f(x)\}$ is $C^1$ at the boundary point $(x_0, f(x_0))$ in all nontangential directions (see, for instance, \cite[Lemma 11.17]{CaffarelliSalsa2005}), so that $G(f)f$ is classically well-defined at $x_0$. However, without min-max formulas available at hand as in \cite{CaffarelliSilvestre2009, Silvestre2011, Guillen2019}, the $C^1$ regularity seems insufficient. Here, we prove that in two dimensions, the harmonic extension $\phi$ is in fact $C^{1, \alpha}$ at $(x_0, f(x_0))$ with quantitative estimates (see Theorem \ref{thm1} and Corollary \ref{coro:pointwiseelliptic}). This pointwise $C^{1, \alpha}$ regularity for harmonic functions in Lipschitz domains is of independent interest. It allows us to reduce to the Dirichlet-Neumann operator for an interior disk tangent to $\{y=f(x)\}$ at $(x_0, f(x_0))$, so that an explicit integral formula can be used (see Proposition \ref{prop:disk}) to conclude the consistency of viscosity solutions.

 \section{The Dirichlet-Neumann operator}
\begin{nota}
For any  set $\cO\in \Rr^d$, we denote
\[
\Lip(\cO)=\Big\{g: \cO\to \Rr: \| g\|_{\Lip(\cO)}:=\sup_{x, y\in \cO, x\ne y}\frac{|g(x)-g(y)|}{|x-y|}<\infty\Big\}.
\]
For $\alpha\in (0, 1]$ we say that $g$ is $C^{1, \alpha}$ at $x_0\in \overline{\cO}$ if there exists $v\in \Rr^3$ and positive numbers $M$ and $\gamma$ such that
\bq\label{def:C1apha}
|g(x)-g(x_0)-v\cdot (x-x_0)|\le M|x-x_0|^{1+\alpha}\quad\forall x\in \Omega,~|x-x_0|< \gamma.
\eq
If $x_0$ is an interior point of  $\cO$ then \eqref{def:C1apha} implies that $\na g(x)$ exists and equals $v$.
\end{nota}
\begin{nota}
For $f:\T^d\to \Rr$ we denote
\begin{align}
&\Omega_f=\{(x, f(x)): x\in \T^d,~y<f(x)\},\quad \Sigma=\{(x, f(x)): x\in \T^d\},\\
& N(x)=(-\p_xf(x), 1),\quad n(x)=\frac{N(x)}{|N(x)|}.
\end{align}
For $\phi:\Omega_f\to \Rr$ and $z_0\in \Sigma$ we denote by
\bq\label{Nlimit}
\lim_ {z\to_N z_0}\phi(z)
\eq
the limit of $\phi$ when $\Omega_f\ni z\to z_0$ in the direction of $N$.
\end{nota}
\begin{nota}
For $x_0\in \Rr^m$, $m\ge 1$  we denote the ball of radius $r$ centered at $x_0$ by $B_r(x_0)$. When $x_0=0$, we shall write $B_r=B_r(0)$. 
\end{nota}
\subsection{Definition and global properties}
\begin{defi}
For $f:\T^d\to \Rr$, the Dirichlet-Neumann operator $G(f)$ is defined by
\bq\label{def:Gh}
\big(G(f)g\big)(x)=\p_N\phi(x, f(x)):=\lim_{h\to 0^-}\frac{1}{h}\big[\phi\big((x, f(x))+hN(x)\big)-\phi(x, f(x))\big],
\eq
where $\phi(x, y)$ solves the elliptic problem
\bq\label{elliptic:G}
\begin{cases}
\Delta_{x, y}\phi=0\quad\text{in}~\Omega_f,\\
\phi(x, f(x))=g(x),\quad \nabla_{x, y}\phi\in L^2(\Omega_f).
\end{cases}
\eq
\end{defi}
When $f$ and $g$ are time-dependent, we  write
\bq\label{nota:DNt}
\big(G(f)g\big)(x, t)\equiv \big(G\big(f(t)\big)g(t)\big)(x).
\eq
The one-phase Muskat problem has a compact reformulation in terms of the Dirichlet-Neumann operator.
\begin{prop}\label{prop:reformDN}
If the fluid domain is given by $\Omega_f$ for some $f(x, t):\Rr\times (0, T)\to \Rr$, then $f$ satisfies
\bq\label{Muskat:DN}
\p_tf=-\ka G(f)f\quad\text{on } (0, T),
\eq
where $\ka={\rho}/{\mu}$.
\end{prop}
This reformulation has been exploited in \cite{AlazardOneFluid2019, NguyenPausader2019, HQNguyen2019, FlynnNguyen2020}. We recall its proof for completeness.
\begin{proof}
Let $q=p+\rho y$ denote the `hydraulic head'. From Darcy's law \eqref{Darcy} we have $\mu u=-\na_{x, y}q$ and $\Delta_{x, y}q=0$ in $\Omega_f$. Moreover, \eqref{bc:p} implies $q(x, f(x))=\rho f(x)$. By the definition of the Dirichlet-Neumann operator, at any fixed time we have
\[
\rho \big(G(f)f\big)(x)=\big(G(f)(\rho f)\big)(x)=N(x)\cdot \nabla_{x, y} q(x, f(x))=-\mu N(x)\cdot u(x, f(x)).
\]
For graph boundary, \eqref{V:boundary} yields
\[
\p_t f(x, t)= N(x, t)\cdot u(f(x, t), t)=-\frac{\rho}{\mu}\big(G(f)f\big)(x, t),
\]
which finishes the proof of \eqref{Muskat:DN}.
\end{proof}
In order to define $G(f)g$ we first study the well-posedness of the elliptic problem \eqref{elliptic:G}. It was proved in Proposition 3.6  \cite{NguyenPausader2019} that \eqref{elliptic:G} has a unique variational solution when $f\in \Lip(\Rr^d)$ and $g\in \dot H^\mez(\Rr^d)$. Of course, this is  only nonstandard when the domain is unbounded.  Adaption to horizontally periodic domains is straightforward. To fix notation for later purposes, we prove
\begin{prop}\label{prop:elliptic}
Let $f\in \Lip(\T^d)$ and $g\in \dot H^\mez(\T^d)$. Then there exists a unique variational solution $\phi\in \dot H^1(\Omega_f)$ to \eqref{elliptic:G},
where
\bq\label{def:dotH1}
\dot H^1(\Omega_f)=\{u\in L^1_{\text{loc}}(\Omega_f): \na_{x, y}u\in L^2(\Omega_f)\}/\Rr
\eq
endowed with the norm $\| u\|_{\dot H^1(\Omega_f)}=\| \na u\|_{L^2(\Omega_f)}$. Moreover, $\phi$ satisfies
\bq\label{est:vari}
\| \phi\|_{\dot H^1(\Omega_f)}\le C(1+\| f\|_{\Lip(\T^d)})\| g\|_{\dot H^\mez(\T^d)}.
\eq
We shall refer to the solution $\phi$ of \eqref{elliptic:G} as {\it the} harmonic extension of $g$ to $\Omega_f$.
\end{prop}
\begin{proof}
According to Theorem \ref{theo:lifting}, there exists $\underline{g}\in \dot H^1(\Omega_f)$ such that $\Tr(\underline{g})=g$ and
\bq\label{est:lifting}
\| \underline{g}\|_{\dot H^1(\Omega_f)}\le C(1+\| f\|_{\Lip(\T^d)})\| g\|_{\dot H^\mez(\T^d)}.
\eq
Denote
\bq
\dot H^1_0(\Omega_f)=\{ u\in \dot H^1(\Omega_f): \Tr(u)=0\},
\eq
where the trace operator $\Tr: \dot H^1(\Omega_f)\to \dot H^\mez(\T)$ is given in Theorem \ref{theo:trace}. We then define $\phi$ solution to \eqref{elliptic:G} to be
\bq\label{def:vari}
\phi=u+\underline{g},
\eq
where $u$ is the unique variational solution in $\dot H^1_0(\Omega_f)$ of the equation $-\Delta_{x, y} u=\Delta_{x, y} \underline{g}$. That is, $u\in\dot H^1_0(\Omega_f)$ satisfies
\bq\label{form:vari}
\int_{\Omega_f}\na_{x, y}u\cdot \nabla_{x, y}\varphi dxdy=-\int_{\Omega_f}\na_{x, y}\underline{g}\cdot \nabla_{x, y}\varphi dxdy\quad\forall \varphi\in \dot H^1_0(\Omega_f).
\eq
Since $\dot H^1_0(\Omega_f)$ is a Hilbert space, the existence and uniqueness of $u$ is guaranteed by the Lax-Milgram theorem. From \eqref{form:vari} we have
\bq\label{vari:phi}
\int_{\Omega_f}\na_{x, y}\phi\cdot \nabla_{x, y}\varphi dxdy=0\quad\forall \varphi\in \dot H^1_0(\Omega_f).
\eq
Moreover, inserting $\varphi=u$ in \eqref{form:vari} we obtain  \eqref{est:vari} from \eqref{est:lifting}. We note that the solution constructed by \eqref{def:vari} and \eqref{form:vari} is independent of the choice of $\underline{g}$.
\end{proof}
With Proposition \ref{prop:elliptic} at hand,  a straightforward adaptation of the results in \cite{Alazard2014, PoyferreNguyen2017, NguyenPausader2019}, which hold for the non-periodic setting, yields
\begin{prop}[\protect{\cite[Theorem 3.8]{Alazard2014}  and \cite[Proposition 3.7]{NguyenPausader2019}}]\label{prop:DNlow}
If  $f\in \Lip(\T^d)$ and $g\in \dot H^\mez(\T^d)$, then $G(f)g$ is well-defined in $H^{-\mez}(\T^d)$ and there exists a universal constant $C$ such that
\bq\label{cont:DN:low}
\| G(f)g\|_{H^{-\mez}(\T^d)}\le C(1+\| f\|_{\Lip(\T^d)})^2\| g\|_{\dot H^\mez(\T^d)}.
\eq
\end{prop}
  In higher Sobolev spaces, the Dirichlet-Neumann operator obeys the following tame estimate.
\begin{prop}[\protect{\cite[Proposition 2.13]{PoyferreNguyen2017}}]\label{prop:estDN}
Let $s_0>1+\frac{d}{2}$ and  $\sigma\ge \mez$. Then there exists a nondecreasing function $\cF:\Rr^+\to \Rr^+$  such that
\bq\label{est:DN}
\| G(f)g\|_{H^{\sigma-1}(\T^d)}\le \cF(\| f\|_{H^{s_0}(\T^d)})\left(\| g\|_{H^\sigma(\T^d)}+\| f\|_{H^\sigma(\T^d)}\| g\|_{H^{s_0}(\T^d)}\right)
\eq	
for all $f,\, g\in H^{\max\{s_0, \sigma\}}(\T^d)$.
\end{prop}
Despite its nonlocality with respect to the boundary, the Dirichlet-Neumann operator has the contraction property given in the next result.
\begin{prop}[\protect{\cite[Corollary 3.25]{NguyenPausader2019}}]\label{theo:contraDN}
Let $s_0>1+\frac{d}{2}$ and $\sigma\in [\mez, s_0]$. Then there exists a nondecreasing function $\cF:\Rr^+\to \Rr^+$  such that
\[
\| G(f_1)g-G(f_2)g\|_{H^{\sigma-1}(\T^d)}\le \cF(\| (f_1, f_2)\|_{H^{s_0}(\T^d)})\| f_1-f_2\|_{H^\sigma(\T^d)}\| g\|_{H^{s_0}(\T^d)}
\]
for all $f_j,\, g\in H^{s_0}(\T^d)$.
\end{prop}
\subsection{Pointwise properties in two dimensions}
\subsubsection{Pointwise $C^{1, \alpha}$ estimate for harmonic functions}\label{Section:elliptic}
Suppose that $ U$ is a Lipschitz domain in $\Rr^2$. For $(x_0,y_0)\in \Rr^2$ and $r>0$, we denote $ U_r(x_0,y_0)=B_r(x_0,y_0)\cap  U$ and $ U_r= U_r(0)$. We also define the half ball as
$$
B^+_r(x_0,y_0)=\{(x,y)\in B_r(x_0,y_0):y>y_{0}\}.
$$
We assume that $0\in \partial U$. Suppose that there exists some $r_0>0$ such that in a coordinate system, $\partial U\cap B_{2r_0}$ can be represented by a Lipschitz graph with Lipschitz constant $L>0$.

Let $u$ be a harmonic function in $ U$, which vanishes on $\partial U$.
\begin{lemm}
Under the conditions above, there exist $\varepsilon_0=\varepsilon_0(L)>0$ and $M_1=M_1(L)>0$ such that $u\in C^{\mez+\varepsilon_0}( U_{r_0})$ and
\begin{equation}
                        \label{eq8.07}
\|u\|_{C^{\mez+\varepsilon_0}( U_{r_0})}\le M_1r_0^{-\tdm-\varepsilon_0}\|u\|_{L^2( U_{2r_0})}.
\end{equation}
\end{lemm}
\begin{proof}
By scaling, we may assume that $r_0=1$ and $\|u\|_{L^2( U_2)}=1$. By the boundary De Giorgi-Nash-Moser estimate, we know that
\begin{equation}
                                \label{eq6.53}
\|u\|_{L^\infty( U_{3/2})}\le M(L).
\end{equation}
Now we fix a point $(x_0, y_0)\in \partial U\cap B_1$. Since $\partial U\cap B_{2}$ can be represented by a Lipschitz graph with Lipschitz constant $L>0$, we may assume that
\begin{equation}
                                \label{eq6.47}
 U_{1/2}(x_0, y_0)\subset \big\{(x, y)\in B_{1/2}(x_0, y_0): y-y_0>-L|x-x_0|\big\}.
\end{equation}
Let $\beta=\tan^{-1}(1/L)\in (0,\frac{\pi}{2})$, $\gamma=\pi/(2\pi-\beta)\in (\mez,1)$, and define
$$
v(x, y)=\|u\|_{L^\infty( U_{3/2})}\text{Re} \big(-2i[x-x_{0}+i(y-y_{0})]\big)^\gamma (\cos((\pi-\beta)\gamma))^{-1}.
$$
In view of \eqref{eq6.47}, it is also easily seen that $v$ is a harmonic function in $ U_{1/2}(x_0, y_0)$. Moreover, $v\ge 0$ on $\partial U\cap B_{1/2}(x_0, y_0)$ and $v\ge \|u\|_{L^\infty( U_{3/2})}$ on $ U\cap \partial B_{1/2}(x_0, y_0)$. By the comparison principle, we have $|u|\le v$ in $ U_{1/2}(x_0, y_0)$, which implies that
$$
|u(x, y)|\le M|(x-x_0)^2+(y-y_0)^2|^\frac{\gamma}{2} \|u\|_{L^\infty( U_{3/2})}
$$
in $ U_{1/2}(x_0, y_0)$. This together with \eqref{eq6.53} and the interior regularity of harmonic functions complete the proof with $\varepsilon_0=\gamma-\mez$.
\end{proof}
We remark that alternatively the above H\"older estimate can be obtained by using the $W^{1, 4+\varepsilon}$ estimate by Jerison-Kenig \cite{JerisonKenig1995} and the Morrey embedding.
\begin{assu}\label{assump1}
There exist constants $M_0,r_0>0$ and function $\psi$ in $(-r_0,r_0)$ such that in a coordinate system
$$
\psi(0)=\psi'(0)=0,\quad  U_{r_0}=\{(x,y)\in B_{r_0}:\,y>\psi(x)\},
$$
and for any $r\in (0,r_0)$,
\begin{equation}
                        \label{eq7.33}
\{(x,y)\in B_{r}:\,y>M_0r^2\}\subset  U_r.
\end{equation}
Moreover, $ U$ satisfies the exterior ball condition at $0$ with radius $1/M_0$.
\end{assu}
Note that if $\partial U$ is $C^{1,1}$ at $0\in \partial U$, then Assumption \ref{assump1} is satisfied.

\begin{theo}
                        \label{thm1}
Under the conditions above, $u$ is $C^{1,\alpha}$ at $0$, i.e., for any $(x,y)\in U$ such that $\sqrt{x^2+y^2}<r_0$, we have
\begin{equation}
                            \label{eq12.24}
|u(x,y)-(x,y)\cdot\nabla_{x, y} u(0)|\le M|x^2+y^2|^{\frac{1+\alpha}{2}}r_0^{-2-\alpha}\|u\|_{L^2( U_{2r_0})},
\end{equation}
where $M>0$ is a constant depending only on $M_0r_0$, and $L$, and $\alpha\in (0,1)$ is a small constant depending only on $L$.
\end{theo}
First, by scaling we may assume that $r_0=1$, so that the new $M_0$ becomes $M_0r_0$, and $\varepsilon_0$ and $L$ remain the same. Moreover, by dividing $u$ by a constant we may also assume that $\|u\|_{L^2( U_{2})}=1$.

Since $\partial U$ satisfies an exterior ball condition, by using a barrier argument, we know that $u$ is Lipschitz at $0$ and
\begin{equation}
                                \label{eq12.24b}
|u(x,y)|\le M(x^2+y^2)^\mez \quad\text{in}\,\, U_{2}.
\end{equation}
By the Caccioppoli inequality, it is easily seen that for $r\in (0,1)$,
\begin{equation}
                            \label{eq11.18}
\|\nabla_{x, y} u\|_{L^2( U_{3r/2})}\le Mr^{-1}\|u\|_{L^2( U_{2r})}\le Mr,
\end{equation}
where we used \eqref{eq12.24b} in the second inequality.
It follows from the reverse H\"older's inequality (see, for instance, \cite[Ch. V]{Giaquinta1983}) and \eqref{eq11.18} that there exists $p_0=p_0(L)>2$ such that
\begin{equation}
                            \label{eq11.24}
\|\na_{x, y} u\|_{L^{p_0}( U_r)}\le Mr^{\frac{2}{p_0}}.
\end{equation}

Now we take a smooth domain $E$ such that $B_{2/3}^+\subset E\subset B_{3/4}^+$. For any $(x_0,y_0)\in \Rr^2$ and $r>0$, denote
\begin{align*}
E_r(x_0,y_0)&=\{(x,y)\in \Rr^2:r^{-1}(x-x_0,y-y_0)\in E\},\\
\Gamma_r(x_0,y_0)&=\{(x,y)\in \partial E_r(x_0,y_0): y=y_{0}\}.
\end{align*}
Clearly, for $r$ sufficiently small, by \eqref{eq7.33} we have $E_r(0,M_0r^2)\subset  U_r$.

Take a smooth function $\eta=\eta(s)$ on $\Rr$ such that $\eta(s)=0$ in $(-\infty,1)$ and $\eta(s)=1$ in $(2,\infty)$. Denote $\eta_r(s)=\eta(s/(M_0r^2))$. A simple calculation reveals that $u(x,y)\eta_r(y)$ satisfies
\begin{equation*}
\Delta_{x, y} (u(x,y)\eta_r(y))=\partial_y(u\eta'_r)+\partial_yu\eta'_r\quad \text{in}\,\,E_r(0,M_0r^2)
\end{equation*}
and $u(x,y)\eta_r(y)=0$ on $\Gamma_r(0,M_0r^2)$. Note that the right-hand side is supported in a narrow strip $\{(x,y)\in  U_r:M_0r^2<y<2M_0r^2\}$.

We decompose $u\eta_r$ in $E_r(0,M_0r^2)$ as follows. Let $w=w_r$ be a weak solution to
\begin{equation*}
\Delta_{x, y} w=\partial_y(u\eta'_r)+\partial_yu\eta'_r\quad \text{in}\,\,E_r(0,M_0r^2)
\end{equation*}
with the zero Dirichlet boundary condition on $\partial E_r(0,M_0r^2)$.
Then $v=v_r=u\eta_r-w$ satisfies
\begin{equation*}
\Delta_{x, y} v=0\quad \text{in}\,\,E_r(0,M_0r^2)
\end{equation*}
and $v=0$ on $\Gamma_r(0,M_0r^2)$.

{\em Estimates of $w$.} Since $E_r(0,M_0r^2)$ is smooth, by the $W^{1, p}$ estimate, we know that for any $p<\infty$, $w\in W^{1, p}(E_r(0,M_0r^2))$. Notice that $\eta_r'(y)\le M|y-\psi(x)|^{-1}$. By using Hardy's inequality and a duality argument (see the proof of \cite[Theorem 3.5]{DongXiong2015}), we have
\begin{align}
                    \label{eq7.52}
\|\nabla_{x, y} w\|_{L^p(E_r(0,M_0r^2))}
&\le M\|\nabla_{x, y} u\|_{L^p( U_r\cap\{y<2M_0r^2\})}.
\end{align}
Now we fix $p=\frac{(2+p_0)}{2}$ and let $q>1$ be such that $\frac{1}{q}=\frac{1}{p}-\frac{1}{p_0}$. Using \eqref{eq7.52}, H\"older's inequality, and \eqref{eq11.24}, we get
\begin{align}
                    \label{eq11.33}
\|\nabla_{x, y} w\|_{L^p(E_r(0,M_0r^2))}
&\le M\|\nabla_{x, y} u\|_{L^{p_0}( U_r\cap\{y<2M_0r^2\})}r^{\frac{3}{q}}\le Mr^{\frac{2}{p_0}+\frac{3}{q}}.
\end{align}
By the zero boundary condition and the Morrey embedding, from \eqref{eq11.33} we obtain
\begin{align}
                    \label{eq7.55}
\|w\|_{L^\infty(E_r(0,M_0r^2))}
&\le Mr^{1-\frac{2}{p}}\|w\|_{C^{1-\frac{2}{p}}(E_r(0,M_0r^2))}\notag\\
&\le Mr^{1-\frac{2}{p}}\|\nabla w\|_{L^p(E_r(0,M_0r^2))}\notag\\
&\le Mr^{1-\frac{2}{p}+\frac{2}{p_0}+\frac{3}{q}}=Mr^{1+\frac{1}{q}}.
\end{align}

{\em Estimates of $v$.} Since $B^+_{r/2}(0,M_0r^2)\subset E_r(0,M_0r^2)$, we know that $v$ is harmonic in $B^+_{r/2}(0,M_0r^2)$ and vanishes on the flat boundary. By the boundary estimate for harmonic functions,
\begin{align*}
\|\nabla_{x, y} v\|_{L^\infty(B^+_{r/4}(0,M_0r^2))}\le Mr^{-1}\|v\|_{L^\infty(B^+_{r/2}(0,M_0r^2))},
\end{align*}
which together with \eqref{eq7.55} and the Lipschitz regularity of $u$ at $0$ implies that
\begin{equation}
                        \label{eq11.51}
\|\nabla_{x, y} v\|_{L^\infty(B^+_{r/4}(0,M_0r^2))}\le M.
\end{equation}
Moreover, by \cite[Lemma 2.5]{DongEscauriazaKim2018}, for any linear function $\ell$ of $y$,
\begin{equation*}
\|\nabla^2_{x, y} v\|_{L^\infty(B^+_{r/4}(0,M_0r^2))}\le Mr^{-2}\|v-\ell\|_{L^\infty(B^+_{r/2}(0,M_0r^2))}.
\end{equation*}
Thus by the mean value theorem and because $v(0,M_0r^2)=\partial_xv(0,M_0r^2)=0$, for any $\kappa\in (0,1/4)$,
\begin{align}
                \label{eq12.25}
&\|v-(y-M_0r^2)\partial_yv(0,M_0r^2)\|_{L^\infty(B^+_{\kappa r}(0,M_0r^2))}\notag\\
&\le M\kappa^{2}\|v-\ell\|_{L^\infty(B^+_{r/2}(0,M_0r^2))}.
\end{align}

{\em Estimates of $u$.}
Recall that $u\eta_r=w+v$ in $E_r(0,M_0r^2)$.
Combining \eqref{eq7.55} and \eqref{eq12.25} gives
\begin{align}
                        \label{eq11.46}
&\|u\eta_r-(y-M_0r^2)\partial_yv(0,M_0r^2)\|_{L^\infty(B^+_{\kappa r}(0,M_0r^2))}\notag\\
&\le M\kappa^{2}\inf_{a,b\in \Rr}\|u\eta_r-(a+by)\|_{L^\infty(B^+_{r/2}(0,M_0r^2))}
+Mr^{1+\frac{1}{q}}.
\end{align}
By the $C^{\mez +\varepsilon_0}$ estimate \eqref{eq8.07}, we also have
\begin{align*}
&\|u(1-\eta_r)\|_{L^\infty( U_r)}
\le \sup_{ U_r\cap \{y<2M_0r^2\}}|u(x,y)|\\
&=\sup_{ U_r\cap \{y<2M_0r^2\}}|u(x,y)-u(x,\psi(x))|
\le Mr^{1+2\varepsilon_0}.
\end{align*}
Combining the above inequality, \eqref{eq11.51}, and \eqref{eq11.46} yields that for any $\kappa\in (0,\frac18)$,
$$
\inf_{a,b\in \Rr}\|u-(a+by)\|_{L^\infty( U_{\kappa r})}
\le M\kappa^{2}\inf_{a,b\in \Rr}\|u-(a+by)\|_{L^\infty( U_{r})}
+Mr^{1+\alpha},
$$
where $\alpha=\min\{2\varepsilon_0, \frac{1}{q}\}$.
By a standard iteration argument, we then get
$$
\inf_{a,b\in \Rr}\|u-(a+by)\|_{L^\infty( U_{r})}\le Mr^{1+\alpha},
$$
i.e., for any $r\in (0,1)$, there exist constants $a_r$ and $b_r$ such that
\begin{equation}
                                \label{eq8.10}
\|u-(a_r+b_ry)\|_{L^\infty( U_{r})}\le Mr^{1+\alpha}.
\end{equation}
\begin{lemm}
                        \label{lem1.1}
We have
\begin{equation}
                        \label{eq8.11}
|u(0)-a_r|\le Mr^{1+\alpha}
\end{equation}
and for any $\frac{r}{2}\le s\le r<1$,
\begin{equation}
                        \label{eq8.20}
|b_{s}-b_r|\le Mr^\alpha.
\end{equation}
\end{lemm}
\begin{proof}
The inequality \eqref{eq8.11} follows directly from \eqref{eq8.10}. For \eqref{eq8.20}, by the triangle inequality, for any $(x,y)\in  U_{s}$,
\begin{align*}
&|(a_{s}+b_{s}y)-(a_{r}+b_{r}y)|\\
&\le |u(x,y)-(a_{s}+b_{s}y)|+|u(x,y)-(a_{r}+b_{r}y)|\le Mr^{1+\alpha}.
\end{align*}
This together with \eqref{eq8.11} implies \eqref{eq8.20}. The lemma is proved.
\end{proof}

\begin{proof}[Proof of Theorem \ref{thm1}]
From \eqref{eq8.20} we see that $b_r$ is convergent as $r\to 0$. Let $b^*$ be the limit. Then again from \eqref{eq8.20}, we have
\begin{equation}
                        \label{eq8.21}
|b^*-b_r|\le Mr^\alpha.
\end{equation}
Combining \eqref{eq8.10}, \eqref{eq8.11}, and \eqref{eq8.21}, we reach
$$
|u(x,y)-b^* y|\le M|x^2+y^2|^{\frac{1+\alpha}{2}},
$$
which gives \eqref{eq12.24} with $\partial_x u(0)=0$ and $\partial_y u(0)=b^*$.
The theorem is proved.
\end{proof}
\begin{coro}\label{coro:pointwiseelliptic}
Assume that $f\in W^{1, \infty}(\T)$ is $C^{1, 1}$ at $x_0\in \T$. Then the harmonic extension $\phi$ of $f$ to $\Omega_f$ is $C^{1, \alpha}$ at $(x_0, f(x_0))$ in the following quantitative sense. Let $\gamma\in (0, 1)$ and $M_0>0$ be such that
\[
|f(x)- f(x_0)-\p_xf(x_0)(x-x_0)|\le M_0|x-x_0|^2\quad \forall |x-x_0|<\gamma.
\]
Then there exists $\alpha\in (0, 1)$  depending only on $\| f\|_{W^{1, \infty}(\T)}$, and there exists $M>0$ depending only on $\| f\|_{W^{1, \infty}(\T)}$ and $M_0\gamma$, such that
\[
|\phi(x, y)-f(x_0)-(x-x_0, y-f(x_0))\cdot \na\phi(x_0, f(x_0))|\le M\gamma^{-1-\alpha}\big|(x-x_0)^2+(y-f(x_0))^2\big|^{\frac{1+\alpha}{2}}
\]
for all $(x, y)\in \Omega_f$ satisfying $(|x-x_0|^2+|y-f(x_0)|^2)^\mez<\gamma$.
\end{coro}
\begin{proof}
Assume without loss of generality that $(x_0, f(x_0))=(0, 0)$. Let $p=\phi-y$ so that $p$ is harmonic in $\Omega_f$ and vanishes on $\p\Omega_f$. Since $\p\Omega$ is a Lipschitz graph globally, Theorem \ref{thm1} implies that
\bq\label{Holder:p}
|p(x,y)-(x,y)\cdot\nabla_{x, y} p(0)|\le M|x^2+y^2|^{\frac{1+\alpha}{2}}\gamma^{-2-\alpha}\|p\|_{L^2(\Omega_f\cap B_{2\gamma})},\quad \sqrt{x^2+y^2}<\gamma,
\eq
where $\alpha$ and $M$ are as in the statement. Since $f\in W^{1, \infty}(\T)$, in \eqref{def:vari} we can take $\underline{g}(x, y)=f(x)\chi(y-f(x))$, where the cutoff $\chi$ is supported on $[-1, 0]$ and identically $1$ on $[-\mez, 0]$.  Then $\phi=u+\underline{g}$,  where $u\in \dot H^1_0(\Omega_f)$ and $\| \underline{g}\|_{H^1(\Omega_f)}\le C(\| f\|_{W^{1, \infty}})$. In the bounded strip $\Omega_{f, \gamma}=\{f(x)-2\gamma<y<f(x)\}$, Poincar\'e's inequality yields
\[
\| u\|_{L^2(\Omega_{f, \gamma})}\le 2\gamma\|\na_{x, y} u\|_{L^2(\Omega_{f, \gamma})}\le 2\gamma\| \na_{x, y}\underline{g}\|_{L^2(\Omega_f)}\le \gamma C(\| f\|_{W^{1, \infty}})
\]
upon recalling \eqref{form:vari}. In addition, $\| \underline{g}\|_{L^2(\Omega_f\cap B_{2\gamma})}\le C_0\gamma \| f\|_{L^\infty(\T)}$, where $C_0$ is an absolute constant. It follows that $\| \phi\|_{L^2(\Omega_f\cap B_{2\gamma})}\le  \gamma  C(\| f\|_{W^{1, \infty}})$. Finally, inserting $p=\phi-y$ into \eqref{Holder:p} we conclude the proof.
\end{proof}
\subsubsection{Pointwise comparison principle for Dirichlet-Neumann}
\begin{prop}\label{prop:comparisonDN}
Let $f_1,\, f_2\in W^{1, \infty}(\T)$ be $C^{1, 1}$ at $x_0$. If $f_1(x)\le f_2(x)$ for all $x\in \T$ and $f_1(x_0)=f_2(x_0)$, then  $G(f_j)f_j$ are classically well-defined at $x_0$ and
\bq\label{compare:DN}
\big(G(f_1)f_1\big)(x_0)\ge \big(G(f_2)f_2\big)(x_0).
\eq
\end{prop}
\begin{proof}
Let $\phi_j$ be the harmonic extension of $f_j$ to $\Omega_j\equiv\Omega_{f_j}$ as given by Proposition \ref{prop:elliptic}. Then $p_j:=\phi_j-y$ are harmonic in $\Omega_j$ and identically vanish on $\p\Omega_j=\{(x, y): y=f_j(x)\}$. Since $x_0$ is a $C^{1, 1}$ boundary point, Theorem \ref{thm1} implies that the normal derivatives  $\p_{N_j}p_j$, where $N_j(x)=(-\p_x f_j(x), 1)$, are classically well-defined at $x_0$. 
Consequently, in view of the definition \eqref{def:Gh}, we have
\bq\label{Gff:p}
\begin{aligned}
\big(G(f_j)f_j\big)(x_0)=\p_{N_j(x_0)}(p_j+y)(x_0, f_j(x_0))=\p_{N_j(x_0)}p_j(x_0, f_j(x_0))+1
\end{aligned}
\eq
are classically well-defined at $(x_0, f(x_0))$. We first claim that
\bq\label{Pj:positive}
p_j(x, y)\ge 0\quad a.e.~ \Omega_j.
\eq
If $\phi_j$ are smooth then so do $p_j$ and \eqref{Pj:positive} is a consequence of the maximum principle for harmonic functions. Here, following Proposition 4.3  \cite{NguyenPausader2019} we prove \eqref{Pj:positive} for variational solutions. Indeed, recall from \eqref{vari:phi} that
\bq\label{vari:phij}
\int_{\Omega_j}\na_{x, y}\phi\cdot \nabla_{x, y}\varphi dxdy=0\quad\forall \varphi\in \dot H^1_0(\Omega_j).
\eq
Inserting $\varphi=\min\{\phi-\inf_{\T^d}f_j, 0\}\in \dot H^1_0(\Omega_j)$ into \eqref{vari:phij} gives
\bq\label{inf:phij}
\inf_{\Omega_j}\phi_j\ge \inf_{\T^d}f_j.
\eq
On the other hand, by choosing $\varphi=\max\{\phi-\sup_{\T^d}f_j, 0\}$ we obtain
\bq\label{sup:phij}
\sup_{\Omega_j}\phi_j\le \sup_{\T^d}f_j.
\eq
It follows from \eqref{inf:phij} that $p_j\ge 0$ for $y\le \inf_{\T^d} f_j$.
Hence, $p_j^-=\min\{p_j, 0\}$ identically vanishes in $\{ (x, y): y\le \inf_{\T^d}f_j\}$ and $p_j^-\in \dot H^1_0(\Omega_j)$. Thus inserting $\varphi=p_j^-$ into \eqref{vari:phij} and replacing $\phi_j=p_j+y$ we obtain
\begin{align*}
\int_{\Omega_j}\na_{x, y}p_j\cdot \nabla_{x, y}p_j^-dxdy&=-\int_{\Omega_j}\p_yp_j^- dxdy\\
&=-\int_{ \inf_{\T^d} f_j<y<f_j(x)}\p_yp_j^- dxdy=p^-_j(x, \inf_{\T^d} f_j)=0.
\end{align*}
This implies $\int_{\Omega_j}| \na_{x, y}p_j^-|^2=0$ and thus $p_j\ge 0$ a.e. $\Omega_j$. We obtain the claim \eqref{Pj:positive}.

Since $f_1\le f_2$, we have $\Omega_1\subset \Omega_2$ and  $p:=p_1-p_2$ satisfies $p(x, f_1(x))=-p_2(x, f_1(x))\le 0$ by virtue of \eqref{Pj:positive}. In addition, noticing   $p=\phi_1-\phi_2$ we find from \eqref{vari:phij} that
\[
\int_{\Omega_1}\na_{x, y}p\cdot \nabla_{x, y}\varphi dxdy=0\quad\forall \varphi\in \dot H^1_0(\Omega_1).
\]
Then as for \eqref{sup:phij}, upon choosing $\varphi=\max\big\{p-\sup_{x\in\T^d}(p(x, f_1(x))),  0\big\}$ we obtain
\[
\sup_{\Omega_1} p\le \sup_{x\in\T^d}(p(x, f_1(x)))\le 0.
\]
Therefore, $p$ attains the zero maximum at the boundary point $(x_0, f_1(x_0))$. In addition, we have $N_1(x_0)=N_2(x_0)$ since $\p_xf_1(x_0)=\p_xf_2(x_0)$ at the maximum point $x_0$ of $f_1-f_2$. We deduce that
\[
0\le  \p_{N_1(x_0)}p(x_0, f_1(x_0))=\p_{N_1(x_0)}p_1(x_0, f_1(x_0))- \p_{N_2(x_0)}p_2(x_0, f_2(x_0)).
 \]
Then \eqref{compare:DN} follows from this and \eqref{Gff:p}.
\end{proof}
\begin{rema}
Proposition \ref{prop:comparisonDN} is valid in all dimensions. Indeed, Theorem \ref{thm1} was only used to conclude that the normal derivatives $\p_{N_j}p_j$ are classically well-defined at $(x_0, f_j(x_0))$. However, the same conclusion can be deduced from  \cite[Lemma 11.17]{CaffarelliSalsa2005} which is valid in all dimensions. On the other hand, the stronger $C^{1, \alpha}$ regularity  obtained in Theorem \ref{thm1} will be crucial in the proof of the comparison principle for viscosity solutions in Section \ref{section:viscosity}.
\end{rema}
\section{Contour dynamics}
In this section, we use  potential theory to express the Dirichlet-Neumann operator, thereby obtaining a contour dynamics reformulation of the Muskat problem.

The Newtonian kernel for $\T\times \Rr$ is
\bq\label{Newton}
\cN(z)=\frac{1}{4\pi}\ln\big(\cosh y-\cos x\big),\quad z=(x, y)\in \T\times \Rr.
\eq
We shall identify functions defined on $\T$ to functions defined on the graph $\Sigma=\{(x, f(x)): x\in \T\}$. Precisely,  $h:\T\to \Rr$ is identified to  $\wt h:\Sigma\to \Rr$  defined by $\wt h(x, f(x))=h(x)$. The double layer potential associated to  $\Omega_f$ for a function $h:\T\to \Rr$ is
\bq\label{doubleL}
\begin{aligned}
\cK[f] h(z)&:=-\int_{\Sigma}(\p_{n(x')}\cN)(z- z')\wt h(z')dz'\\
&= \frac{1}{4\pi}\int_{\T} \frac{\sin(x-x')\p_xf(x')-\sinh(y-f(x'))}{\cosh(y-f(x'))-\cos(x-x')}h(x')dx',
\end{aligned}
\eq
where $z=(x, y)\in (\T\times\Rr)\setminus \Sigma$ and $z'=(x', f(x'))\in \Sigma$. Clearly, $\cK[f] h(z)$ is harmonic on $(\T\times \Rr)\setminus \Sigma$. A direct calculation using (hyperbolic) trigonometric identities gives
\[
\p_{x'}\Big(\arctan\frac{\tanh(\frac{y-f(x')}{2})}{\tan(\frac{x-x'}{2})}\Big)=-\mez \frac{\sin(x-x')\p_xf(x')-\sinh(y-f(x'))}{\cosh(y-f(x'))-\cos(x-x')}.
\]
Hence,
\bq\label{doubleL:2}
\cK[f] h(z)=-\frac{1}{2\pi}\int_{\T}\p_{x'}\Big(\arctan\frac{\tanh(\frac{y-f(x')}{2})}{\tan(\frac{x-x'}{2})}\Big)h(x')dx'.
\eq
The single layer potential is
\bq
\begin{aligned}
S[f]h(x, y)=\frac{1}{4\pi}\int_\T \ln\big(\cosh (y-f(x'))-\cos (x-x')\big)h(x')dx',\quad (x, y)\in \T\times \Rr.
\end{aligned}
\eq
Note that the preceding single layer potential differs from the one in \cite{Fabes1978, Verchota1984} by the arc length element $\sqrt{1+|\p_xf(x')|^2}$.

The boundary double layer potential is
\bq\label{K:arctan}
\begin{aligned}
K[f] h(x)
&=:p.v.\frac{1}{4\pi}\int_{\T} \frac{\sin(x-x')\p_xf(x')-\sinh(f(x)-f(x'))}{\cosh(f(x)-f(x'))-\cos(x-x')}h(x')dx'\\
&=-p.v.\frac{1}{2\pi}\int_{\T} \p_{x'}\Big(\arctan\frac{\tanh(\frac{f(x)-f(x')}{2})}{\tan(\frac{x-x'}{2})}\Big)h(x')dx',
\end{aligned}
\eq
where $z=(x, f(x))$ and $z'=(x', f(x'))$.
The adjoint of $K[f]$ is
\bq\label{def:K*}
\begin{aligned}
K^*[f] h(x)&=-p.v.\frac{1}{2\pi}\int_{\T}\p_{x}\Big(\arctan\frac{\tanh(\frac{f(x)-f(x')}{2})}{\tan(\frac{x-x'}{2})}\Big)h(x')dx'\\
&=p.v.\frac{1}{4\pi}\int_{\T} \frac{\sinh(f(x)-f(x'))-\sin(x-x')\p_xf(x)}{\cosh(f(x)-f(x'))-\cos(x-x')}h(x')dx'.
\end{aligned}
\eq
If $f\in \Lip(\T)$ and $h\in L^p(\T)$ for some $p\in (1, \infty)$ then the nontangential limit
\bq
\lim_{\Omega\ni z'\to z}\mathcal{K}h(z')=(\mez I+K)h(z)
\eq
holds a.e. $z\in \T$ (see Theorem 1.10 \cite{Verchota1984}). Consequently, the unique solution $\phi$ of the Dirichlet problem \eqref{elliptic:G}
is given by
\bq\label{elliptic:DL}
\phi=\mathcal{K}(\mez I+K)^{-1}g.
\eq
It was proved in Section 4 \cite{Verchota1984} that for Lipschitz domains, the mappings
\begin{align}\label{invert:1}
&\mez I+K: L^p(\T)\to L^p(\T), \quad p\in [2, \infty),\\ \label{invert:2}
&\mez I+K: W^{1, p}(\T)\to W^{1, p}(\T), \quad p\in (1, 2],\\ \label{invert:3}
& \mez I-K^*: L^p_0(\T)\to L^p_0(\T), \quad p\in (1,  2]
\end{align}
are invertible. Here, $L^p_0(\T)$ denotes the space of zero mean functions in $L^p(\T)$.
\begin{prop}\label{prop:DNcontour}
For  $f\in\Lip(\T)$ and $g\in H^1(\T)$, we have for a.e. $x\in\T$ that
\begin{align}
\label{DN:contour12}
(G(f)g)(x)&=\frac{1}{4\pi}p.v.\int_\T \frac{\sin(x-x')+\sinh(f(x)-f(x'))\p_xf(x)}{\cosh(f(x)-f(x'))-\cos(x-x')}\tt(x')dx'\\ \label{DN:contour1}
&=\frac{1}{4\pi}p.v.\int_{\T}\p_x\ln\Big(\cosh(f(x)-f(x'))-\cos(x-x')\Big)\tt(x')dx',
\end{align}
where
\bq\label{tt:dg}
\tt=(\mez I-K^*)^{-1}(\p_xg).
\eq
\end{prop}
\begin{proof}
Applying \eqref{elliptic:DL} we have that the unique solution $\phi$ to \eqref{elliptic:G} is $\phi=\mathcal{K}\Theta$ with $\Theta=(\mez I+K)^{-1}g$. Note that $\Theta\in H^1(\T)$ in view of \eqref{invert:2}. Setting $\tt=\p_x\Theta\in L^2(\T)$ we deduce from \eqref{doubleL:2} that
\[
\phi(x, y)=\frac{1}{2\pi}\int_{\T}\arctan\frac{\tanh(\frac{y-f(x')}{2})}{\tan(\frac{x-x'}{2})}\tt(x')dx',\quad (x, y)\in \T\times \Rr.
\]
It follows from \eqref{K:arctan} and \eqref{def:K*}  that
\begin{align*}
\p_xK[f]\Theta(x)&=p.v.\frac{1}{2\pi}\int_{\T} \p_{x}\Big(\arctan\frac{\tanh(\frac{f(x)-f(x')}{2})}{\tan(\frac{x-x'}{2})}\Big)\p_x\Theta(x')dx'=-K^*[f]\theta(x),
\end{align*}
whence $\tt=\p_x(\mez I+K)^{-1}g=(\mez I-K^*)^{-1}\p_xg$.

For $(x, y)\in (\T\times \Rr)\setminus\Sigma$,  we compute
\begin{align*}
&\p_x\phi(x, y)=-\frac{1}{4\pi}\int_\T \frac{\sinh(y-f(x'))}{\cosh(y-f(x'))-\cos(x-x')}\tt(x')dx',\\
&\p_y\phi(x, y)=\frac{1}{4\pi}\int_\T \frac{\sin(x-x')}{\cosh(y-f(x'))-\cos(x-x')}\tt(x')dx'.
\end{align*}
Set  $T(x)=(1, \p_xf(x))$ to be the tangent to $\Sigma$. For $\Omega_f\ni (\wt x, y)\to_N (x, f(x))$, by the  definition \eqref{def:Gh} and the mean-value theorem, we have
\begin{align*}
(G(f)g)(x)&=\lim_{(\wt x, y)\to_N (x, f(x))} N(x)\cdot(\na_{x, y}\phi)(\wt x, y)\\
&=\lim_{(\wt x, y)\to_N (x, f(x))}\frac{1}{4\pi}\int_\T \frac{\sinh(y-f(x'))\p_xf(x)+\sin(\wt x-x')}{\cosh(y-f(x'))-\cos(\wt x-x')}\tt(x')dx'\\
&=\lim_{(\wt x, y)\to_N (x, f(x))}T(x)\cdot \na S[f]\tt(\wt x, y).
\end{align*}
Since $f\in \Lip(\T)$ and $\tt\in L^2(\T)$,  Theorem 1.6 \cite{Verchota1984} asserts that the tangential derivative of $S[f]\tt$ is continuous up to the boundary. That is,
\[
\lim_{(\wt x, y)\to_N (x, f(x))}T(x)\cdot \na S[f]\tt(\wt x, y)=T(x)\cdot \na S[f]\tt(x, f(x))
\]
 for a.e. $x\in \T$. Therefore, for a.e. $x\in \T$ we obtain
\bq\label{DN:S}
\begin{aligned}
(G(f)g)(x)&=T(x)\cdot \na S[f]\tt(x, f(x))\\
&=\frac{1}{4\pi}p.v.\int_\T \frac{\sinh(f(x)-f(x'))\p_xf(x)+\sin(x-x')}{\cosh(f(x)-f(x'))-\cos(x-x')}\tt(x')dx',
\end{aligned}
\eq
proving \eqref{DN:contour12}. Finally, since
\bq\label{SL-DL}
\frac{\sin(x-x')+\sinh(f(x)-f(x'))\p_xf(x)}{\cosh(f(x)-f(x'))-\cos(x-x')}
=\p_x\ln\Big(\cosh(f(x)-f(x'))-\cos(x-x')\Big),
\eq
we obtain \eqref{DN:contour1}.
\end{proof}
%
Proposition \ref{prop:reform} follows at once from Propositions \ref{prop:reformDN} and \ref{prop:DNcontour}.
\section{Quantitative bound for inverse of double layer potential}\label{Section:potential}
According to \eqref{invert:3}, $\mez I-K^*: L^2_0(\T)\to L^2_0(\T)$ in invertible provided that the boundary $f$ of $\Omega_f$ is Lipschitz. This is sufficient to deduce the unique solvability of the Neumann problem in Lipschitz domain with $L^2$ data  \cite{Verchota1984}. However, in order to obtain uniform bounds for approximate solutions of the Muskat problem we shall need a {\it quantitative bound} for the inverse of $\mez I-K^*$. This is the content of the next proposition. 
\begin{prop}\label{prop:potential}
There exists a universal constant $C>0$ such that
\bq\label{boundDLP}
\| (\mez I\pm K^*)^{-1}\|_{L^2_0(\T)\to L^2_0(\T)}\le C(1+\| f\|_{\Lip(\T)})^{\frac{5}{2}}.
\eq
\end{prop}

\begin{proof}
	In order to get this result, we will show that there exists a constant $M=\bar C(1+\|\partial_x f\|_{L^\infty(\T)})^{5/2}$ with $\bar C>1$ such that
\begin{equation}\label{comparable}
\frac{1}{M}\leq \frac{\|(\frac12I-K^*[f])(\eta)\|_{L_0^2(\T)}}{\|(\frac12I+K^*[f])(\eta)\|_{L_0^2(\T)}}\leq M \quad \forall \eta\in L_0^2(\T).
\end{equation}
We claim that \eqref{comparable} yields
\begin{equation}\label{invertiblebound}
\|\eta\|_{L^2_0(\T)}\leq 2M\|(\frac12I-K^*[f])(\eta)\|_{L^2_0(\T)}\quad\forall \eta\in L_0^2(\T).
\end{equation}
This in turn implies the desired bound \eqref{boundDLP}. Indeed, if \eqref{invertiblebound} is false then  there exists $\eta_1\in L^2_0(\T)$ with $\|\eta_1\|_{L^2_0(\T)}=1$ such that $$\|(\frac12I-K^*[f])(\eta_1)\|_{L^2_0(\T)}<\frac{1}{2M}.$$ Then $$\|(\frac12I+K^*[f])(\eta_1)\|_{L^2_0(\T)}\geq \|\eta_1\|_{L^2_0(\T)}-\|(\frac12I-K^*[f])(\eta_1)\|_{L^2_0(\T)}>1-\frac1{2M}>\frac12,$$
but using \eqref{comparable} again we find the opposite:
$$\|(\frac12I+K^*[f])(\eta_1)\|_{L^2_0(\T)}\leq M\|(\frac12I-K^*[f])(\eta_1)\|_{L^2_0(\T)}<\frac{1}{2}.$$
This finishes the proof of \eqref{invertiblebound}.
Now we prove that
\begin{equation}\label{cpcomp}
\|(\frac12I-K^*[f])(\eta)\|_{L_0^2(\T)}\leq M\|(\frac12I+K^*[f])(\eta)\|_{L_0^2(\T)}\quad\forall \eta\in L^2_0(\T).
\end{equation}
Recall that the single layer potential given by
\begin{equation}\label{SLP}
S[f](\eta)(x,y)=\frac{1}{4\pi}\int_{-\pi}^\pi
\ln\Big(\cosh(y-f(x'))-\cos(x-x')\Big)\eta(x')dx'=H(x,y).
\end{equation}
With $\Omega=\{(x, y): x\in \T, y<f(x) \}$, $H$ satisfies (see \cite{Verchota1984})
\bq\label{ellipticH}
\begin{cases}
	\Delta_{x, y}H(x,y)=0,\quad (x, y)\in \Omega\cup(\Rr^2\setminus\overline{\Omega}),\\
	\partial_N H(x,f(x))=(-\frac12I+K^*[f])(\eta)(x)\quad \text{a.e. }x\in \T,\\
	\partial^c_N H(x,f(x))=(\frac12I+K^*[f])(\eta)(x)\quad \text{a.e. } x\in \T,
\end{cases}
\eq
where $\p_N$ is the normal derivative from inside the domain $\Omega$ and $\p_N^c$ is to the normal derivative from inside the complementary set $\Omega^c=\Rr^2\setminus\overline{\Omega}$. By Theorem 1.6 \cite{Verchota1984}, the tangential derivative $\p_\tau H$ is continuous across the $\p\Omega$. Thus we shall not distinguish the tangential derivative of $H$ when approaching $\p\Omega$ from either side. Since $\int_\T \eta dx=0$ we can write
\[
\begin{aligned}
H(x, y)&=\int_\T \ln \Big(\frac{\cosh(y-f(x'))-\cos(x-x')}{\cosh y-\cos x}\Big)\eta(x')dx'\\
&=\int_\T \ln \Big(\frac{\sinh^2(\frac{y-f(x')}{2})+\sin^2(\frac{x-x'}{2})}{\sinh^2\frac{y}{2}+\sin^2\frac{x}{2}}\Big)\eta(x')dx'
\end{aligned}
\]
and deduce that as $|y|\to \infty$, $H(x, y)$ is bounded and $\na H$ decay uniformly in $x\in \T$.

 We consider a compactly supported vector field $V(x,y)$ given by
\begin{equation}\label{defV}
V(x,y)=(\Gamma_{\delta}*G)(x,y),
\end{equation}
where the function $G(x,y)$ is given by
\bq
G(x,y)=\begin{cases}
	n(x),\quad x\in\T,\quad |y|\leq 2\|f\|_{L^\infty(\T)}+2,\\
	0\quad \mbox{elsewhere},
\end{cases}
\eq
and $\Gamma_{\delta}(x,y)$ is a smooth approximation of the identity with compact support. The parameter $\delta>0$ is taken small enough so that
\begin{equation}\label{conditiondelta}
V(x,f(x))\cdot n(x)\geq \mez \quad \forall x\in\T.
\end{equation}
One could check that
\begin{equation}\label{boundsforV}
\|V\|_{L^\infty}\leq 1\quad \mbox{and}\quad  \|\nabla V\|_{L^\infty}\leq C=O\big(\delta^{-1}\big),
\end{equation}
where $C$ is independent of $f$.

Let us recall the Rellich identity for $V$ and $H$:
\begin{equation}\label{Rellich}
\int_{\partial\Omega}V\cdot n|\nabla H|^2d\sigma=\int_{\partial\Omega}2\p_nH (V\cdot \nabla H)d\sigma+\int_{\Omega}(\nabla\cdot V) |\nabla H|^2dX-\int_{\Omega}2(\nabla V\nabla H)\cdot \nabla H dX,
\end{equation}
where $\na V$ is a $2\times 2$ matrix acting on $\na H$. Using this together with the normal and tangential decomposition for $\nabla H=(\partial_n H) n+(\p_{\tau} H )\tau$ gives
\begin{align*}
\int_{\partial\Omega}V\cdot n|\p_n H|^2d\sigma=&-\int_{\partial\Omega}2\p_nH\p_\tau H V\cdot\tau d\sigma+\int_{\partial\Omega}V\cdot n|\p_\tau H|^2d\sigma\\
&-\int_{\Omega}(\nabla\cdot V) |\nabla H|^2dX+\int_{\Omega}2(\nabla V\nabla H)\cdot \nabla H dX.
\end{align*}
Condition \eqref{conditiondelta}, Cauchy-Schwarz and Young's inequalities yield
\begin{align*}
\frac14\int_{\partial\Omega}|\p_n H|^2d\sigma\leq&(1+8\|V\|_{L^\infty})\|V\|_{L^\infty}\int_{\partial\Omega}|\p_\tau H|^2d\sigma+3\|\nabla V\|_{L^\infty}\int_{\Omega}|\nabla H|^2dX.
\end{align*}
Plugging in the bounds for $V$ in \eqref{boundsforV} we find that
\begin{align}\label{yacasi}
\int_{\partial\Omega}|\p_n H|^2d\sigma\leq&36\int_{\partial\Omega}|\p_\tau H|^2d\sigma+12C\int_{\Omega}|\nabla H|^2dX.
\end{align}
Using the boundedness of $H$ and the decay of $\na H$ as $|y|\to \infty$, we can integrate by parts to obtain
\begin{equation}\label{intparts}
\int_{\Omega}|\nabla H|^2dX=\int_{\Omega}\nabla H\cdot \na [H-H(-\pi,f(-\pi))]dX=\int_{\partial\Omega}[H-H(-\pi,f(-\pi))]\p_n H d\sigma.
\end{equation}
By the fundamental theorem of calculus and Cauchy-Schwarz's inequality,
\[
\begin{aligned}
&\left|H(x, f(x))-H(-\pi,f(-\pi))\right|=\left|\int_{-\pi}^x\p_\tau H(x', f(x'))\sqrt{1+|\p_xf(x')|^2}dx'\right|\\
&\quad\le \Big(\int_{-\pi}^x|\p_\tau H(x', f(x'))|^2\sqrt{1+|\p_xf(x')|^2}dx'\Big)^{\frac12}\Big(\int_{-\pi}^x \sqrt{1+|\p_xf(x')|^2}dx'\Big)^{\frac12}\\
&\quad\le \sqrt{2\pi \| N\|_{L^\infty(\T)}} \Big(\int_{\partial\Omega}|\p_\tau H|^2d\sigma\Big)^{\frac12}.
\end{aligned}
\]
It follows from this and Cauchy-Schwarz's inequality for \eqref{intparts} that
\begin{equation}\label{Poincare}
\int_{\Omega}|\nabla H|^2dX\leq 2\pi\|N\|_{L^\infty(\T)}\Big(\int_{\partial\Omega}|\p_\tau H|^2d\sigma\Big)^{\frac12}\Big(\int_{\partial\Omega}|\p_n H|^2d\sigma\Big)^{\frac12}.
\end{equation}
Plugging above inequality into \eqref{yacasi} and using Young's inequality we get
\begin{equation}\label{dnbdt}
\int_{\partial\Omega}|\p_n H|^2d\sigma\leq 2(12C\pi)^2\|N\|^2_{L^\infty(\T)}\int_{\partial\Omega}|\p_\tau H|^2d\sigma.
\end{equation}
Next, we employ the Rellich identity \eqref{Rellich} for $V$ and $H$ in $\Omega^c=\Rr^2\setminus\overline{\Omega}$ to get
\begin{align*}
\frac12\int_{\partial\Omega}|\nabla H|^2d\sigma&\leq \int_{\partial\Omega}2\p^c_n H (V\cdot \nabla H)d\sigma+\int_{\Omega^c}(\nabla\cdot V) |\nabla H|^2dX-\int_{\Omega^c}2(\nabla V \nabla H)\cdot\nabla H dX\\
&\leq 2\Big(\int_{\partial\Omega}|\p_n^c H|^2d\sigma\Big)^{\frac12}\Big(\int_{\partial\Omega}|\nabla H|^2d\sigma\Big)^{\frac12}+3C\int_{\Omega^c}|\nabla H|^2dX.
\end{align*}
As before, using Young's inequality and  \eqref{Poincare} for $\Omega^c$ we obtain
\begin{align*}
\frac14\int_{\partial\Omega}|\nabla H|^2d\sigma&\leq 4\int_{\partial\Omega}|\p_n^c H|^2d\sigma+6C\pi\|N\|_{L^\infty(\T)}\Big(\int_{\partial\Omega}|\p_\tau H|^2d\sigma\Big)^{\frac12}\Big(\int_{\partial\Omega}|\p^c_n H|^2d\sigma\Big)^{\frac12}.
\end{align*}
The obvious  bound $|\p_\tau H|^2\leq |\nabla H|^2$ combined with Young's inequality yields
\begin{align}\label{dtadnOc}
\int_{\partial\Omega}|\p_\tau H|^2d\sigma&\leq 2(12C\pi)^2\|N\|^2_{L^\infty(\T)}\int_{\partial\Omega}|\p^c_n H|^2d\sigma.
\end{align}
The above argument when applied to $\Omega$ gives
\begin{align}\label{dtadn}
\int_{\partial\Omega}|\p_\tau H|^2d\sigma&\leq 2(12C\pi)^2\|N\|^2_{L^\infty(\T)}\int_{\partial\Omega}|\p_n H|^2d\sigma.
\end{align}
Combining  \eqref{dnbdt} and \eqref{dtadnOc} allows us to relate both normal derivatives due to the continuity of the tangential derivative:
\begin{align*}
\int_{\partial\Omega}|\p_n H|^2d\sigma&\leq 4(12C\pi)^4\|N\|^4_{L^\infty(\T)}\int_{\partial\Omega}|\p^c_n H|^2d\sigma,\quad C=O(\delta^{-1}).
\end{align*}
Thus we  obtain  \eqref{cpcomp}  with $M=2(12C\pi)^2(1+\| f\|_{\Lip(\T)})^{5/2}$. The opposite inequality
$$
\|(\frac12I+K^*[f])(\eta)\|_{L_0^2}\leq M\|(\frac12I-K^*[f])(\eta)\|_{L_0^2}
$$
follows by the same argument, completing the proof.
\end{proof}
\begin{rema}
The invertibility of $(\mez I-K^*)$ on $L^2_0$ was proved in  \cite{Verchota1984} using the inequalities
\begin{align}\label{est:Ver}
&\| (\mez I\pm K^*)(\eta)\|_{L^2(\T)}\le C\| (\mez I\mp K^*)(\eta)\|_{L^2(\T)} +C\Big|\int_\T S[f](\eta) dx\Big|.
\end{align}
The inequalities in \eqref{comparable} dispense with the second term on the right-hand of \eqref{est:Ver} and hence allowed us to obtain the quantitative bound \eqref{boundDLP}.
\end{rema}
As an interesting corollary of  the proof of Proposition \ref{prop:potential}, we deduce a quantitative bound for the tangential derivative of the single layer potential in terms of the double layer potential. This in turn yields the continuity of $G(g):\dot H^1\to L^2$ assuming only that the boundary is Lipschitz.
\begin{coro}\label{coro:potential}
Let $f\in \Lip(\T)$. There exists an absolute constant $C>0$ such that for all $\eta\in L^2_0(\T)$, we have
\bq\label{bound:SLP}
\| \na S[f](\eta)\cdot(1, \p_xf)\|_{L^2(\T)}
\le C(1+\| f\|_{\Lip(\T)})^2\| (\frac12I-K^*[f])(\eta)\|_{L^2}.
\eq
Consequently,
\bq\label{velocitybound}
\| G(f)g\|_{L^2(\T)}\le C(1+\| f\|_{\Lip(\T)})^2\| \p_xg\|_{L^2(\T)}
\eq
for all $g\in  \dot H^1(\T)$.
\end{coro}
 \begin{proof}
Denote  $H(x, y)=S[f](\eta)(x, y)$  the single lay potential as given by \eqref{SLP}. Now \eqref{dtadn} allows us to bound $\p_\tau H$ by $\p_n H$ as
\begin{align*}
\| \na S[f](\eta)\cdot(1, \p_xf)\|_{L^2(\T)}&\le (1+\| f\|_{\Lip(\T)})\| \p_\tau H\|_{L^2(\T)}\\
&\le C(1+\| f\|_{\Lip(\T)})^2\| \p_n H\|_{L^2(\T)},
\end{align*}
where we recall that $\tau$ and $n$ are respectively the unit tangent and  normal to the surface $\{y=f(x)\}$. On the other hand, it follows from the second equation in \eqref{ellipticH}  that
\[
 \| \p_n H\|_{L^2(\T)}\le  \| \p_N H\|_{L^2(\T)}=\| (-\frac12I+K^*[f])(\eta)\|_{L^2},
 \]
 from which  \eqref{bound:SLP} follows.

 By the density of $\dot H^1(\T)$ in $H^1(\T)$, it suffices to prove  \eqref{velocitybound} for $g\in H^1(\T)$. Set $\tt=(\mez I-K^*[f])^{-1}(\p_xg)$. Recall from \eqref{DN:S}  that $G(f)g$  is a tangential derivative of $S[f]\tt$:
\[
(G(f)g)(x)=\nabla S[f](\theta)(x,f(x))\cdot (1,\p_x f(x)).
\]
Thus \eqref{velocitybound} follows at once from \eqref{bound:SLP} and the fact that $(\mez I-K^*[f])\tt=\p_xg$.
 \end{proof}
 \begin{rema}
 The inequality \eqref{velocitybound} upgrades the continuity \eqref{cont:DN:low} by half a derivative without  any additional regularity assumption on the surface $f$. By interpolating \eqref{cont:DN:low} and \eqref{velocitybound}, we obtain
 \bq
 \| G(f)g\|_{H^{\sigma-1}(\T)}\le C(1+\| f\|_{\Lip(\T)})^2\|g\|_{\dot H^\sigma(\T)}
 \eq
 for all $\sigma \in [\mez, 1]$. For the same range of $\sigma$, this is an improvement over \eqref{est:DN} which requires in addition that $f\in H^{s_0}(\T)$ for some $s_0>\tdm$.
 \end{rema}
\begin{rema}
Proposition \ref{prop:potential} and Corollary \ref{coro:potential} are valid for general bounded domains of $\Rr^n$,~$n\ge 2$ with Lipschitz boundary. 
 \end{rema}
\section{Viscosity regularization}
We consider the viscosity regularization of the Muskat problem:
\bq\label{Muskat:r}
\p_tf^\eps=-\ka G(f^\eps)f^\eps+\eps \p_x^2 f^\eps,\quad \eps\ge 0.
\eq
\subsection{Comparison principle  for smooth solutions}
\begin{prop}[{\bf Comparison principle}]\label{comparison:Muskatr}
Assume that $f_j\in C([0, T]; C^2(\T))$, $j=1, 2$ are smooth solutions of \eqref{Muskat:r} with $\eps\ge 0$. If $f_1(\cdot, 0)\le f_2(\cdot, 0)$ then
\bq
f_1(x, t)\le f_2(x, t)\quad\forall (x, t)\in \T\times [0, T].
\eq
\end{prop}
\begin{proof}
Assume by contradiction that $M_0=\max_{\T\times [0, T]}(f_1-f_2)=(f_1-f_2)(x_0, t_0)>0$. We have $t_0>0$ since $f_1(\cdot, 0)\le f_2(\cdot, 0)$. Choose $\eta>0$ sufficiently small so that $(f_1-f_2-\eta t)(x_0, t_0)=M_0-\eta t_0>0$, and thus $M_*=\max_{\T\times [0, T]}(f_1-f_2-\eta t)>0$. Moreover, $M_*$ is attained at some point $(x_*, t_*)$ with $t_*>0$ because $(f_1-f_2-\eta t)\vert_{t=0}\le 0$. Then
\[
\max_{\T\times [0, T]}(f_1-f_2-\eta t-M_*)=0
\]
and is attained at $(x_*, t_*)$. Applying Proposition \ref{prop:comparisonDN} gives
 \[
\big(G(f_1)f_1\big)(x_*, t_*)\ge \big(G(f_2+\eta t+M_*)(f_2+\eta t+M_*)\big)(x_*, t_*).
 \]
In addition, at the maximum point $(x_*, t_*)$ with  $t_*>0$ we have
\[
\p_t f_1(x_*, t_*)\ge \p_t(f_2+\eta t+M_*)(x_*, t_*), \quad \p_x^2f_1(x_*, t_*)\le \p_x^2(f_2+\eta t+M_*)(x_*, t_*).
\]
 It follows that
\begin{align*}
\p_t(f_2+\eta t+M_*)\le \p_t f_1&=  -\ka G(f_1)f_1+\eps \p_x^2f_1\\
&\le  -\ka G(f_2+\eta t+M_*)(f_2+\eta t+M_*)+\eps\p_x^2f_2\\
&= -\ka G(f_2)(f_2)+\eps\p_x^2f_2= \p_t f_2
\end{align*}
at $(x_*, t_*)$. But this implies $\eta\le 0$ which is contradictory.
\end{proof}
Owing to the translation invariance of \eqref{Muskat:r},  comparison principle implies maximum principles for the amplitude and the slope of solutions.
\begin{coro}[{\bf Maximum principles}]\label{prop:maxslop}
Let $f\in C([0, T]; C^2(\T))$ be a  smooth solution of \eqref{Muskat:r} with $\eps\ge 0$. If  $f(\cdot, 0)$ has the modulus of continuity $\gamma:\Rr^+\to \Rr^+$ then so does $f(\cdot, t)$ for all $t\in [0, T]$. That is,
\bq\label{max:moduli}
|f(x, t)-f(y, t)|\le \gamma(|x-y|)\quad\forall (x, t)\in \T\times [0, T].
\eq
In particular, for $\gamma(z)=z\| f^\eps(\cdot, 0)\|_{\Lip(\T)}$ we have
\bq\label{max:slope}
 \| f(\cdot, t)\|_{\Lip(\T)}\le \| f(\cdot, 0)\|_{\Lip(\T)}\quad\forall t\in [0, T].
\eq
If $f_1$ and $f_2$ are smooth solutions in $C([0, T]; C^2(\T))$ of \eqref{Muskat:r} with $\eps\ge 0$, then
\bq
\label{max:amp}
\| f_1(\cdot, t)-f_2(\cdot, t)\|_{L^\infty(\T)}\le \| f_1(\cdot, 0)-f_2(\cdot, 0)\|_{L^\infty(\T)}\quad\forall t\in [0, T].
\eq
\end{coro}
\begin{proof}
The modulus of continuity $\gamma$ of $f(\cdot, 0)$ is equivalent to $f(x+y, 0)\le f(x, 0)+\gamma(|y|)$ for all $x,\, y\in \T$. For every fixed $y$, the function $f(x+y, t)+\gamma(|y|)$ is a smooth solution of \eqref{Muskat:r} with initial data $f(x+y, 0)+\gamma(|y|)$. The comparison principle in Proposition \ref{comparison:Muskatr} implies that
$f(x, t)\le f(x+y, t)+\gamma(|y|)$ for all $x,\, y\in \T$ and $t\in [0, T]$. Therefore, $|f(x, t)-f(x+y, t)|\le \gamma(|y|)$, proving \eqref{max:moduli}.

We turn to prove \eqref{max:amp}. We have  $f_1(\cdot, 0)\le f_2(\cdot, 0)+\| f_1(\cdot, 0)-f_2(\cdot, 0)\|_{L^\infty(\T)}$ and $f_2(\cdot, t)+\| f_1(\cdot, 0)-f_2(\cdot, 0)\|_{L^\infty(\T)}$ is a smooth solution of \eqref{Muskat:r}. By the comparison principle in Proposition \ref{comparison:Muskatr},
\[
f_1(x, t)\le f_2(x, t)+\| f_1(\cdot, 0)-f_2(\cdot, 0)\|_{L^\infty(\T)}\quad \forall (x, t)\in \T\times [0, T].
\]
Thus we can interchange the role of $f_1$ and $f_2$ to obtain \eqref{max:amp}.
\end{proof}
\subsection{Global well-posedness of smooth solutions}
By virtue of Proposition \ref{prop:DNcontour}, in the class of Lipschitz solutions, \eqref{Muskat:r} is equivalent to the system
\begin{align}
\begin{split}\label{Muskat:rGE}
\p_tf^\eps(x)&=\frac{1}{4\pi}p.v.\int_{-\pi}^{\pi}\frac{\sin x'+\sinh(f^\eps(x)-f^\eps(x-x'))\p_xf^\eps(x)}{\cosh(f^\eps(x)-f^\eps(x-x'))-\cos x'}\theta^\eps(x-x')dx'  +\eps \p_x^2 f^\eps(x),\\
\frac12\theta^\eps(x)&-\frac{1}{4\pi}\int_{-\pi}^\pi
\frac{\sinh(f^\eps(x)-f^\eps(x-x'))-\sin x'\p_x f^\eps(x)}{\cosh(f^\eps(x)-f^\eps(x-x'))-\cos x'}\theta^\eps(x-x')dx'=-\kappa \p_x f^\eps(x).
\end{split}
\end{align}
We prove that  \eqref{Muskat:rGE} is globally well-posed for any smooth initial data in Sobolev spaces.
\begin{prop}\label{GlobalEpsilonProblem}
Let $\eps>0$ and $s\ge 2$. For each $f^\eps_0\in H^s(\T)$,  there exists a unique solution $f^\eps\in C([0,T];H^s(\T))\cap L^2([0, T]; H^{s+1}(\T))$ a solution to \eqref{Muskat:r} for any time $T>0$ with initial data $f^\eps_0$. In particular, $f^\eps$ obeys the maximum principles
\bq\label{max:C1}
 \| f^\eps(t)\|_{L^\infty(\T)}\le \| f^\eps_0\|_{L^\infty(\T)},\quad  \| f^\eps(t)\|_{\Lip(\T)}\le \| f^\eps_0\|_{\Lip(\T)}\quad\forall t\ge 0.
\eq
\end{prop}
\begin{proof}
We first note that the dissipative term $\eps \p_x^2f^\eps$ makes the regularized Muskat problem \eqref{Muskat:r} semilinear. Then by the contraction mapping method, it was proved  in Section 4.1.5 \cite{NguyenPausader2019} that for any $s>\tdm$, \eqref{Muskat:r} has a unique solution $f^\eps\in C([0, T_\eps); H^s)\cap L^2_{\text{loc}}([0, T_\eps); H^{s+1})$  where $T_\eps$ is the maximal time. Moreover, if $T_\eps<\infty$ then
 \[
 \lim\sup_{t\to T^\eps}\| f^\eps(\cdot, t)\|_{H^s}=\infty.
 \]
 Therefore, to conclude that $f^\eps$ is global in time, it suffices to prove that the $H^s$ norm of $f^\eps$ is bounded on $[0, T_\eps)$. To this end, we shall prove that for any $T<T_\eps$, the $C([0, T]; H^s)$ norm of $f^\eps$ can be controlled  by $\| f^\eps_0\|_{H^s}$ and the $L^\infty([0, T]; W^{1, \infty})$ norm of $f^\eps$, where the latter is uniformly bounded in time owing to the maximum principles \eqref{max:slope} and  \eqref{max:amp}.

For notational simplicity we shall write $f^{\eps}=f$. In what follows, $\cF:\Rr^+\to \Rr^+$ denotes  continuous nondecreasing functions that may change from line to line but only depend on $\ka$.

{\bf $L^2$ estimate}. Using $(G(f)f, f)\ge 0$ and $(\p_x^2f, f)\le 0$, we find at once that
\bq\label{max:L2}
\|f(t)\|_{L^2}\le \| f_0\|_{L^2}\quad\forall t>0.
\eq
{\bf $\dot H^1$ estimate}. Multiplying  \eqref{Muskat:r} by $-\p_x^2f$, then integrating by parts and using the estimate \eqref{velocitybound}, we obtain
$$
\begin{aligned}
\frac12\frac{d}{dt}\|\p_xf\|_{L^2}^2&\leq \ka\|\p_x^2f\|_{L^2}\|G(f)f\|_{L^2}-\eps\|\p_x^2f\|_{L^2}^2\\
&\leq \cF(\|f\|_{W^{1,\infty}})\| \p_x^2f\|_{L^2}\| \p_xf\|_{L^2}-\eps\|\p_x^2f\|_{L^2}^2\\
&\le \eps^{-1}\cF(\|f\|_{W^{1,\infty}})\|\p_xf\|_{L^2}^2-\frac{\eps}{2}\|\p_x^2f\|_{L^2}^2.
\end{aligned}
$$
By a Gr\"onwall argument we find that
\bq\label{H1ev}
\|\p_xf\|_{L^2}^2(t)+\eps\int_0^t\|\p_x^2f\|_{L^2}^2(t')dt'\leq \|\p_xf\|_{L^2}^2(0)\exp\Big(\eps^{-1}\cF(\|f\|_{W^{1,\infty}})t\Big).
\eq
{\bf $\dot H^2$ estimate}. We differentiate \eqref{Muskat:r} in $x$, then multiply by $-\p_x^3f$ and integrate by parts to obtain
\[
 \frac12\frac{d}{dt}\|\p_x^2f\|_{L^2}^2=\ka (\p_x[G(f)f], \p_x^3f)_{L^2, L^2}-\eps\|\p_x^3f\|_{L^2}^2,
\]
where by Cauchy-Schwarz and Young's inequalities
\[
\ka (\p_xG(f)f, \p_x^3 f)_{L^2, L^2}\le \frac{\eps}{4}\| \p_x^3f\|_{L^2}^2+ \frac{\ka^2}{\eps} \| \p_x[G(f)f]\|_{L^2}^2.
\]
We claim that for any $\nu\in (0, 1)$,
\bq\label{key:H2est}
 \| \p_x[G(f)f]\|_{L^2}\le \nu \| \p_x^3f\|_{L^2}+\frac{1}{\nu}\cF(\| f\|_{W^{1, \infty}})(\|\p_x^2f\|_{L^2}+\| \p_x^2f\|_{L^2}^2).
\eq
Taking \eqref{key:H2est} for granted, we  choose  $\nu=\frac{\eps}{4\ka}$ so that the above estimates yield
\bq
 \frac12\frac{d}{dt}\|\p_x^2f\|_{L^2}^2\le \eps^{-3}\cF(\| f\|_{W^{1, \infty}})(1+\| \p_x^2f\|^2_{L^2})\|\p_x^2f\|_{L^2}^2-\frac{\eps}{2}\|\p_x^3f\|_{L^2}^2.
 \eq
 Then applying Gr\"onwall's lemma  and  \eqref{H1ev} to bound
 \[
 \int_0^t(1+\|\p_x^2f\|_{L^2}^2(t'))dt'\le t+\eps^{-1}\|\p_xf\|_{L^2}^2(0)\exp\Big(\eps^{-1}\cF(\|f\|_{W^{1,\infty}})t\Big),
 \]
we arrive at the $\dot H^2$ estimate
\begin{align}
\begin{split}\label{H2ev}
\|\p_x^2f\|_{L^2}^2(t)\!+\!\eps\!\int_0^t\!\!\|\p_x^3f\|_{L^2}^2(s)ds\le C(\| f_0\|_{H^2}, \| f\|_{W^{1, \infty}}, t, \eps)\end{split}.
\end{align}
The proof of \eqref{key:H2est} requires a careful decomposition of the singular integral representation \eqref{DN:contour12}-\eqref{tt:dg} of $G(f)f$ and is postponed to Lemma \ref{lemm:keyH2} below.

{\bf Higher regularity}. In this step we prove that $H^s$ regularity with $s>2$ can be propagated. To this end, we use  Sobolev estimates for the Dirichlet-Neumann operator with the control norm in $L^\infty_t H^{\tdm+}_x$ provided by \eqref{H2ev}. Indeed, fixing $s_0\in (\tdm, 2)$ and applying \eqref{est:DN} with $g=f$ and $\sigma=s$, we have
\bq
\| G(f)f\|_{H^{s-1}}\le \cF(\|f\|_{H^{s_0}})\| f\|_{H^s}.
\eq
Then an $H^s$ energy estimate for \eqref{Muskat:r} gives
\bq
\begin{aligned}
\mez\frac{d}{dt}\| f\|_{H^s}^2&\le \ka (G(f)f, f)_{H^{s-1}, H^{s+1}}-\eps \| \p_xf\|_{H^s}^2\\
& \le  \ka\cF(\|f(t)\|_{H^{s_0}})\| f(t)\|_{H^s}\| f(t)\|_{H^{s+1}}-\eps \| f(t)\|_{H^{s+1}}^2+\eps \|f(t)\|_{L^2}^2\\
&\le \frac{ \ka^2}{2\eps}\cF^2(\|f(t)\|_{H^{s_0}})\| f(t)\|^2_{H^s}-\frac{\eps}{2} \| f(t)\|_{H^{s+1}}^2+\eps \|f_0\|_{L^2}^2,
\end{aligned}
\eq
where in the last inequality we used Young's inequality and the maximum principle for the $L^2$ norm of $f$. By Gr\"onwall's lemma, we obtain
\bq
\| f(t)\|_{H^s}^2+\eps \| f\|_{L^2([0, t]; H^{s+1})}^2\le (1+\eps t)\| f_0\|_{H^s}^2\exp\Big(t\frac{\ka^2}{\eps}\cF^2\big(\|f\|_{L^\infty([0, t]; H^{s_0})}\big)\Big).
\eq
Then using the $\dot H^2$ estimate \eqref{H2ev} together with \eqref{max:L2}, \eqref{max:amp} and \eqref{max:slope}, we arrive at
\bq
\|  f\|_{L^\infty([0, T]; H^s)}\le C(\| f_0\|_{H^s}, T, \eps)\quad\forall T<T_\eps.
\eq
This concludes the proof of global regularity for \eqref{Muskat:r}.
\end{proof}
We now prove the claim \eqref{key:H2est}.
\begin{lemm}\label{lemm:keyH2}
There exists $\cF:\Rr^+\to \Rr^+$ nondecreasing such that for any $\nu \in (0, 1)$,
\bq
 \| \p_x[G(f^\eps)f^\eps]\|_{L^2}\le \nu \| \p_x^3f^\eps\|_{L^2(\T)}+\frac{1}{\nu}\cF(\| f^\eps\|_{W^{1, \infty}(\T)})(\|\p_x^2f^\eps\|_{L^2}+\| \p_x^2f^\eps\|_{L^2(\T)}^2).
\eq
\end{lemm}
\begin{proof}
We shall write $f^\eps=f$, $\tt^\eps=\tt$, and   $\delta_{x'}f(x)=f(x)-f(x-x')$. Recall that $G(f)f$ can be written as on the right-hand side of the first equation in \eqref{Muskat:rGE} with $\tt=(\mez I-K^*)^{-1}(-\p_xf)$. Moreover, by virtue of Proposition \ref{prop:potential},
\bq
\| \tt\|_{L^2}\le C(1+\| \p_xf\|_{L^\infty})^\frac{5}{2}\| \p_xf\|_{L^2}\le \sqrt{2\pi}C(1+\| \p_xf\|_{L^\infty})^\frac{7}{2}.
\eq
We decompose $4\pi \p_x[G(f)f](x)=\sum_{j=1}^4G_j(x)$ where
\begin{equation*}
G_1(x)=\p_x f(x)\int\frac{\cosh(\delta_{x'}f(x))\delta_{x'}(\p_xf)(x)}{\cosh(\delta_{x'}f(x))-\cos x'}\theta(x-x')dx',
\end{equation*}
\begin{equation*}
  G_2(x)=\p_x^2 f(x)\int\frac{\sinh(\delta_{x'}f(x))}{\cosh(\delta_{x'}f(x))-\cos x'}\theta(x-x')dx',
\end{equation*}
\begin{equation*}
G_3(x)=-\int\frac{[\sin x'+\sinh(\delta_{x'}f(x))\p_x f(x)]\sinh(\delta_{x'}f(x))\delta_{x'}(\p_xf)(x)}{(\cosh(\delta_{x'}f(x))-\cos x')^2}\theta(x-x')dx',
\end{equation*}
and
\begin{equation*}
 G_4(x)=\int\frac{\sin x'+\sinh(\delta_{x'}f(x))\p_x f(x)}{\cosh(\delta_{x'}f(x))-\cos x'}\p_x\theta(x-x')dx'dx.
\end{equation*}
Here and throughout this proof we write $\int=\int_{-\pi}^\pi$.

{\bf Control of $G_1$}. We split  $G_1=\sum_{j=1}^5G_{1,j}$ where, for $\delta<1$ to be chosen,
$$
G_{1,1}(x)=\p_x f(x)\int\frac{(\cosh(\delta_{x'}f(x))-\cos x'+\cos x'-1)\delta_{x'}(\p_xf)(x)}{\cosh(\delta_{x'}f(x))-\cos x'}\theta(x-x')dx',
$$
$$
G_{1,2}(x)=\p_x f(x)\int_{|x'|>\delta}\frac{\delta_{x'}(\p_xf)(x)}{\cosh(\delta_{x'}f(x))-\cos x'}\theta(x-x')dx',
$$
$$
G_{1,3}(x)=\p_x f(x)\int_{|x'|<\delta}\frac{\delta_{x'}(\p_xf)(x)-x'\p_x^2f(x)}{\cosh(\delta_{x'}f(x))-\cos x'}\theta(x-x')dx',
$$
$$
G_{1,4}(x)=\p_x f(x)\p_x^2f(x)\int_{|x'|<\delta}\Big(\frac{1}{\cosh(\delta_{x'}f(x))-\cos x'}-\frac{2}{(1+(\p_xf(x))^2)(x')^2}\Big)x'\theta(x-x')dx',
$$
and
$$
G_{1,5}(x)=\frac{2\p_x f(x)\p_x^2f(x)}{1+(\p_xf(x))^2}\int_{|x'|<\delta}\frac{\theta(x-x')}{x'}dx'.
$$
We shall use frequently the facts that
\bq\label{trigidentity}
\cosh(\delta_{x'}f(x))-\cos x'=2\sinh^2\big(\frac{\delta_{x'}f(x)}{2}\big)+2\sin^2\big(\frac{x'}{2}\big)
\eq
and
\bq\label{ineq:sine}
 \frac{|x'|}{2}\le \frac{\pi}2\big|\sin \frac{x'}{2}\big|\quad\forall x'\in [-\pi, \pi].
\eq
Here and in what follows, $C$ denotes a universal constant that may change from line to line.  Then the integral kernel in $G_{1,1}$ is bounded by $C\| \p_xf\|_{L^\infty}$, giving
$$
\|G_{1,1}\|_{L^2}\leq C\|\p_xf\|_{L^\infty}^2\|\int |\tt(\cdot -x')|dx'\|_{L^2}\le 2\pi C\|\p_xf\|_{L^\infty}^2\|\theta\|_{L^2}\leq \cF(\|f\|_{W^{1,\infty}})\| \p_xf\|_{L^2}.
$$
Using Young's inequality for convolutions, we bound
\begin{align*}
\|G_{1,2}\|_{L^2}&\leq C\|\p_xf\|_{L^\infty}\| \p_xf\|_{C^\mez}\Big\|\int_{|x'|>\delta}\frac{|\theta(\cdot-x')|}{|x'|^\tdm}dx'\Big\|_{L^2}\\
&\leq \cF(\|f\|_{W^{1,\infty}})\| \p_x^2f\|_{L^2} \frac{\|\theta\|_{L^2}}{\delta^\mez}\leq  \cF(\|f\|_{W^{1,\infty}})\frac{\| \p_x^2f\|_{L^2}}{\delta^\mez}.
\end{align*}
Next in view of the inequalities
\[
|\delta_{x'}(\p_xf)(x)-x'\p_x^2f(x)|\le |x'|^\tdm\| \p_x^2f\|_{C^\mez(\T)}\le C |x'|^\tdm\| \p_x^2f\|_{\dot{H}^1(\T)}\le C |x'|^\tdm\| \p_x^3f\|_{L^2(\T)},
\]
we deduce that
\[
\|G_{1,3}\|_{L^2}\leq C\|\p_xf\|_{L^{\infty}}\|\p_x^3f\|_{L^2}\Big\|\int_{|x'|<\delta}\frac{|\theta(\cdot-x')|}{|x'|^\frac12}dx'\Big\|_{L^2}\leq \cF(\|f\|_{W^{1,\infty}})\|\p_x^3f\|_{L^2}\delta^{\frac12}.
\]
To control $G_{1,4}$, we use Taylor's expansion to have 
\begin{align*}
\|G_{1,4}\|_{L^2}&\leq C(1+\|\p_xf\|_{L^{\infty}})^2\|\p_x^2f\|_{L^\infty}(\|\p_x^2f\|_{L^\infty}+1)\Big\|\int_{|x'|<\delta}|\theta(\cdot-x')|dx'\Big\|_{L^2}\\
&\leq \cF(\|f\|_{W^{1,\infty}})(\|\p_x^2f\|_{L^\infty}^2\delta+\| \p_xf\|_{L^2}\delta),
\end{align*}
which in conjunction with Gagliardo-Nirenberg's inequality implies
$$
\|G_{1,4}\|_{L^2}\leq \cF(\|f\|_{W^{1,\infty}})(\|\p_x^2f\|_{L^2}\|\p_x^3f\|_{L^2}\delta+\| \p_xf\|_{L^2}).
$$
The integral in $G_{1,5}$ has a Calder\'on-Zygmund type kernel, whence
$$
\|G_{1,5}\|_{L^2}\leq C\|\p_x^2f\|_{L^\infty}\|\theta\|_{L^2}\leq \cF(\|f\|_{W^{1,\infty}})\|\p_x^2f\|_{L^2}^{\frac12}\|\p_x^3f\|_{L^2}^{\frac12}.
$$
Gathering the above estimates yields
\[
\| G_1\|_{L^2}\le \cF(\|f\|_{W^{1,\infty}})\big(\|\p_x^2f\|_{L^2}\|\p_x^3f\|_{L^2}\delta+\|\p_x^3f\|_{L^2}\delta^{\frac12}+\| \p^2_xf\|_{L^2}\delta^{-\mez}+\|\p_x^2f\|_{L^2}^{\frac12}\|\p_x^3f\|_{L^2}^{\frac12}\big),
\]
where we bounded the $\| \p_xf\|_{L^2}$ term in the estimates for $G_{1,1}$ and $G_{1, 4}$ by $C\|\p_x^2f\|_{L^2}^{\frac12}\|\p_x^3f\|_{L^2}^{\frac12}$.

For any $\nu\in (0, 1)$, we  choose
\[
\delta=\frac{\nu^2}{2^{11}\cF(\|f\|_{W^{1,\infty}})(1+\|\p_x^2f\|_{L^2})}
\]
with $\cF$ sufficiently large so that
\bq\label{G1}
\|G_1\|_{L^2}\leq  2^{-10}\nu\|\p_x^3f\|_{L^{2}}+\frac{1}{\nu}\cF(\|f\|_{W^{1,\infty}})\|\p_x^2f\|_{L^2}
(1+\|\p_x^2f\|_{L^2})^\mez.
\eq
{\bf Control of $G_2$}. We decompose $G_2=\sum_{j=1}^4G_{2,j}$ where
$$
G_{2,1}(x)=\p_x^2 f(x)\int\frac{\sinh(\delta_{x'}f(x))-\delta_{x'}f(x)}{\cosh(\delta_{x'}f(x))-\cos x'}\theta(x-x')dx',
$$
$$
G_{2,2}(x)=\p_x^2 f(x)\int\frac{\delta_{x'}f(x)-x'\p_xf(x)}{\cosh(\delta_{x'}f(x))-\cos x'}\theta(x-x')dx',
$$
$$
G_{2,3}(x)=\p_x^2 f(x)\p_xf(x)\int\Big(\frac{1}{\cosh(\delta_{x'}f(x))-\cos x'}-\frac{2}{(1+(\p_xf(x))^2)(x')^2}\Big)x'\theta(x-x')dx',
$$
and
$$
G_{2,4}(x)=\frac{2\p_x^2 f(x)\p_xf(x)}{1+(\p_xf(x))^2}\int\frac{\theta(x-x')}{x'}dx'.
$$
The integral kernel in $G_{2,1}$ is bounded by an absolute constant so that
$$
\|G_{2,1}\|_{L^2}\leq  \cF(\|f\|_{W^{1,\infty}})\|\p_x^2f\|_{L^2}\Big\|\int|\theta(\cdot-x')|dx'\Big\|_{L^\infty}\le \cF(\|f\|_{W^{1,\infty}})\|\p_x^2f\|_{L^2}.
$$
The identity
\bq\label{taylor2order}
x'\p_xf(x)-\delta_{x'}f(x)=(x')^2\int_0^1s\p_x^2f(x+(s-1)x')ds
\eq
allows us to get
\begin{align*}
\begin{split}
\|G_{2,2}\|_{L^2}&\leq \|\p_x^2f\|_{L^2}\Big\|\int_0^1s\int|\p_x^2f(\cdot+(s-1)x')||\theta(\cdot-x')| dx'ds\Big\|_{L^\infty}\\
&\leq \|\p_x^2f\|_{L^2}\int_0^1\frac{sds}{\sqrt{1-s}}\|\p_x^2f\|_{L^2}\|\theta\|_{L^2}\leq \cF(\|f\|_{W^{1,\infty}})\|\p_x^2f\|^2_{L^2},
\end{split}
\end{align*}
where we applied Cauchy-Schwarz's inequality for the integral in $x'$.

Using the inequality
\[
|\delta_{x'}f(x)-x'\p_xf(x)|\le |x'|^\tdm\| \p_xf\|_{C^\mez}\le C|x'|^\tdm\| \p^2_xf\|_{L^2},
\]
we can bound $G_{2,3}$ as
\begin{align*}
\begin{split}
\|G_{2,3}\|_{L^2}\leq \cF(\|f\|_{W^{1,\infty}})\|\p_x^2f\|^2_{L^2}\| \int |\tt(\cdot-x')|dx'\|_{L^2}\le  \cF(\|f\|_{W^{1,\infty}})\|\p_x^2f\|^2_{L^2}.
\end{split}
\end{align*}
Finally,  $G_{2,4}$ obeys the same bound as $G_{1, 5}$:
$$
\|G_{2,4}\|_{L^2}\leq  \|\p_x^2f\|_{L^\infty}\|\theta\|_{L^2}\leq \cF(\|f\|_{W^{1,\infty}})\|\p_x^2f\|_{L^2}^{\frac12}\|\p_x^3f\|_{L^2}^{\frac12}.
$$
Gathering all the above  estimates leads to
\[
\| G_2\|_{L^2}\le \cF(\|f\|_{W^{1,\infty}})\big(\|\p_x^2f\|_{L^2}^{\frac12}\|\p_x^3f\|_{L^2}^{\frac12}
+\|\p_x^2f\|^2_{L^2}+\|\p_x^2f\|_{L^2} \big)
\]
and thus, for any $\nu\in (0, 1)$, we have
\bq\label{G2}
\|G_2\|_{L^2}\leq 2^{-10}\nu\|\p_x^3f\|_{L^{2}}+\frac{1}{\nu}\cF(\|f\|_{W^{1,\infty}})(\|\p_x^2f\|_{L^2}+\|\p_x^2f\|_{L^2}^2).
\eq
{\bf Control of $G_3$}. We split $G_{3}=\sum_{j=1}^7G_{3,j}$
where
\begin{align*}
G_{3,1}(x)=&\int\frac{(\sinh(\delta_{x'}f(x))-\delta_{x'}f(x))\p_x f(x)\sinh(\delta_{x'}f(x))\delta_{x'}(\p_xf)(x)}{(\cosh(\delta_{x'}f(x))-\cos x')^2}\theta(x-x')dx'\\
&+\int\frac{\delta_{x'}f(x)\p_x f(x)(\sinh(\delta_{x'}f(x))-\delta_{x'}f(x))\delta_{x'}(\p_xf)(x)}{(\cosh(\delta_{x'}f(x))-\cos x')^2}\theta(x-x')dx',
\end{align*}
\begin{equation*}
G_{3,2}(x)=\int_{|x'|>\delta}\frac{(\sin x'+\delta_{x'}f(x)\p_x f(x))\delta_{x'}f(x)\delta_{x'}(\p_xf)(x)}{(\cosh(\delta_{x'}f(x))-\cos x')^2}\theta(x-x')dx',
\end{equation*}
\begin{align*}
G_{3,3}(x)=\int_{|x'|<\delta}\frac{(\sin x'+\delta_{x'}f(x)\p_x f(x))(\delta_{x'}f(x)-\p_xf(x)x')\delta_{x'}(\p_xf)(x)}{(\cosh(\delta_{x'}f(x))-\cos x')^2}\theta(x-x')dx',
\end{align*}
\begin{equation*}
G_{3,4}(x)=\p_xf(x)\int_{|x'|<\delta}\frac{(\delta_{x'}f(x)-\p_xf(x)x')x'\delta_{x'}(\p_xf)(x)}{(\cosh(\delta_{x'}f(x))-\cos x')^2}\theta(x-x')dx',
\end{equation*}
\begin{equation*}
G_{3,5}(x)=\p_xf(x)\int_{|x'|<\delta}\frac{(\sin x'+(\p_x f(x))^2x')x'(\delta_{x'}(\p_xf)(x)-\p_x^2f(x)x')}{(\cosh(\delta_{x'}f(x))-\cos x')^2}\theta(x-x')dx',
\end{equation*}
\begin{align*}
G_{3,6}(x)=&\p_xf(x)\p_x^2f(x)\int_{|x'|<\delta}(\sin x'+(\p_x f(x))^2x')(x')^2A(x,x')\theta(x-x')dx'\\
\mbox{with}\quad&A(x,x')=\Big(\frac{1}{(\cosh(\delta_{x'}f(x))-\cos x')^2}-\frac{4}{(1+(\p_xf(x))^2)^2(x')^4}\Big),
\end{align*}
and
\begin{align*}
G_{3,7}(x)=&\frac{4\p_xf(x)\p_x^2f(x)}{(1+(\p_xf(x))^2)^2}\int_{|x'|<\delta}\frac{\sin x'+(\p_x f(x))^2x'}{(x')^2}\theta(x-x')dx'.
\end{align*}
The integral kernel in $G_{3,1}$ is  smooth enough so that
$$
\|G_{3,1}\|_{L^2}\leq C\|\p_xf\|_{L^\infty}^2\|\theta\|_{L^2}\leq \cF(\|f\|_{W^{1,\infty}})\| \p_xf\|_{L^2}.
$$
By brutal force we have
$$
\|G_{3,2}\|_{L^2}\leq \cF(\|f\|_{W^{1,\infty}})\| \p_xf\|_{C^\mez}\Big\|\int_{|x'|>\delta}\frac{|\theta(\cdot-x')|}{|x'|^\tdm}dx'\Big\|_{L^2}\leq \cF(\|f\|_{W^{1,\infty}})\frac{\| \p^2_xf\|_{L^2}}{\delta^\mez}.
$$
For the next two terms, Gagliardo-Nirenberg's inequality gives
\begin{align*}
\begin{split}
\|G_{3,3}\|_{L^2}+\|G_{3,4}\|_{L^2}\leq& \cF(\|f\|_{W^{1,\infty}})
\|\p_x^2f\|_{L^{\infty}}^2\Big\|\int_{|x'|<\delta}|\theta(\cdot-x')|dx'\Big\|_{L^2}\\
\leq&
\cF(\|f\|_{W^{1,\infty}})\|\p_x^2f\|_{L^2}\|\p_x^3f\|_{L^2}\delta.
\end{split}
\end{align*}
The term $G_{3,5}$ is bounded by
$$
\|G_{3,5}\|_{L^2}\leq \cF(\|f\|_{W^{1,\infty}})\|\p_x^2f\|_{\dot{C}^\frac12}\Big\|\int_{|x'|<\delta}\frac{|\theta(\cdot-x')|}{|x'|^\frac12}dx'\Big\|_{L^2}\leq \cF(\|f\|_{W^{1,\infty}})\|\p_x^3f\|_{L^2}\delta^{\frac12}.
$$
As for $G_{1,4}$, we have
\begin{align*}
\begin{split}
\|G_{3,6}\|_{L^2}\leq& \cF(\|f\|_{W^{1,\infty}})\|\p_x^2f\|_{L^\infty}^2\Big\|\int_{|x'|<\delta}|\theta(\cdot-x')|dx'\Big\|_{L^2}\\
\leq& \cF(\|f\|_{W^{1,\infty}})\|\p_x^2f\|_{L^2}\|\p_x^3f\|_{L^2}\delta.
\end{split}
\end{align*}
Finally, the kernels in $G_{3,7}$ are of Calder\'on-Zygmund type in such a way that
$$
\|G_{3,7}\|_{L^2}\leq \cF(\|f\|_{W^{1,\infty}})\|\p_x^2f\|_{L^2}^{\frac12}\|\p_x^3f\|_{L^2}^{\frac12}.
$$
Gathering the above estimates we find that
$$
\|G_3\|_{L^2}\leq \cF(\|f\|_{W^{1,\infty}})\big(\|\p_x^2f\|_{L^2}^{\frac12}\|\p_x^3f\|_{L^2}^{\frac12}+\|\p_x^2f\|_{L^2}\|\p_x^3f\|_{L^2}\delta+\|\p_x^3f\|_{L^2}\delta^{\frac12}+\| \p^2_xf\|_{L^2}\delta^{-\mez}\big).
$$
For any $\nu\in (0, 1)$, using Young's inequality for the first term and choosing
\[
\delta^\mez=\frac{\nu}{2^{11}\cF(\|f\|_{W^{1,\infty}})(1+\|\p_x^2f\|_{L^2})},
\]
we obtain
\bq\label{G3}
\|G_3\|_{L^2}\leq 2^{-10}\nu\|\p_x^3f\|_{L^{2}}+\frac{1}{\nu}\cF(\|f\|_{W^{1,\infty}})\|\p_x^2f\|_{L^2}.
\eq
{\bf Control of $G_4$}.  We first note that
$$G_4(x)=4\pi\nabla (S[f](\p_x\theta))(x,f(x))\cdot(1,\p_x f(x)).
$$
In view of the bound \eqref{bound:SLP} for tangential derivative of $S[f]$, we have
\begin{equation*}
\|G_{4}\|_{L^2}\leq \cF(\|f\|_{W^{1,\infty}})\|(\frac{1}{2}I-K^*)(\p_x\theta)\|_{L^2}.
\end{equation*}
To proceed, we differentiate the second equation in \eqref{Muskat:rGE} with respect to $x$ to have
\begin{equation}\label{pxtheta}
(\frac{1}{2}I-K^*)(\p_x\theta)(x)=-\kappa\p_x^2f(x)+G_{4,1}+G_{4,2}
\end{equation}
where
\begin{equation*}
G_{4,1}(x)=\frac1{4\pi}\int\frac{\cosh(\delta_{x'}f(x))\delta_{x'}(\p_xf)(x)-\sin x'\p_x^2 f(x)}{\cosh(\delta_{x'}f(x))-\cos x'}\theta(x-x')dx',
\end{equation*}
and
\begin{equation}\label{G42}
G_{4,2}(x)=\frac{-1}{4\pi}\int\frac{(\sinh(\delta_{x'}f(x))-\sin x'\p_x f(x))\sinh(\delta_{x'}f(x))\delta_{x'}(\p_xf)(x)}{(\cosh(\delta_{x'}f(x))-\cos x')^2}\theta(x-x')dx'.
\end{equation}
Therefore,
\begin{equation}\label{G41G42}
\|G_{4}\|_{L^2}\leq \cF(\|f\|_{W^{1,\infty}})(\|\p_x^2f\|_{L^2}+\|G_{4,1}\|_{L^2}+\|G_{4,2}\|_{L^2})
\end{equation}
and it remains to bound $G_{4,1}$ and $G_{4,2}$. A further splitting gives $G_{4,1}=G_{4,1}^1+G_{4,1}^2+G_{4,1}^3+G_{4,1}^4$ where
$$
G_{4,1}^1(x)=\frac1{4\pi}\int\frac{(\cosh(\delta_{x'}f(x))-\cos x'+\cos x'-1)\delta_{x'}(\p_xf)(x)}{\cosh(\delta_{x'}f(x))-\cos x'}\theta(x-x')dx',
$$
$$
G_{4,1}^2(x)=\frac{\p_x^2 f(x)}{4\pi}\int\frac{x'-\sin x'}{\cosh(\delta_{x'}f(x))-\cos x'}\theta(x-x')dx',
$$
\begin{equation*}
G_{4,1}^3(x)=\frac1{4\pi}\int_{|x'|>\delta}\frac{\delta_{x'}(\p_xf)(x)- x'\p_x^2 f(x)}{\cosh(\delta_{x'}f(x))-\cos x'}\theta(x-x')dx',
\end{equation*}
and
\begin{equation*}
G_{4,1}^4(x)=\frac1{4\pi}\int_{|x'|<\delta}\frac{\delta_{x'}(\p_xf)(x)- x'\p_x^2 f(x)}{\cosh(\delta_{x'}f(x))-\cos x'}\theta(x-x')dx'.
\end{equation*}
Since the kernel integral in $G_{4,1}^1$ is bounded by $C\| \p_xf\|_{L^\infty}$, we have
\begin{equation*}
\|G_{4,1}^1\|_{L^2}\leq C\| \p_xf\|_{L^\infty}\| \tt\|_{L^2}\le\cF(\|f\|_{W^{1,\infty}})\| \p_xf\|_{L^2}.
\end{equation*}
Similarly, the kernel integral in $G_{4, 2}^2$ is bounded by an absolute constant, yielding
\begin{equation*}
\|G_{4,1}^2\|_{L^2}\leq  C\|\p_x^2f\|_{L^2}\|\theta\|_{L^2}\leq   \cF(\|f\|_{W^{1,\infty}})\|\p_x^2f\|_{L^2}.
\end{equation*}
Using convolution properties, we bound
\begin{align*}
\begin{split}
\|G_{4,1}^3\|_{L^2}&\leq C\|\p_xf\|_{\dot{C}^\frac12}\Big\|\int_{|x'|>\delta}\frac{|\theta(\cdot-x')|}{|x'|^{\tdm}}dx'\Big\|_{L^2}
+C\|\p_x^2f\|_{L^2}\Big\|\int_{|x'|>\delta}\frac{|\theta(\cdot-x')|}{|x'|}dx'\Big\|_{L^\infty}  \\
&\leq C\frac{\|\p_xf\|_{\dot{C}^\frac12}\|\theta\|_{L^2}}{\delta^{\frac12}}+C\frac{\|\p_x^2f\|_{L^2}\|\theta\|_{L^2}}{\delta^{\frac12}}\\
&\leq \cF(\|f\|_{W^{1,\infty}})\frac{\|\p_x^2f\|_{L^2}}{\delta^{\frac12}}.
\end{split}
\end{align*}
 Similarly, $G_{4,1}^4$ can be controlled as
 \begin{align*}
\begin{split}
\|G_{4,1}^4\|_{L^2}\leq& \|\p_x^2f\|_{\dot{C}^\frac12}\Big\|\int_{|x'|<\delta}\frac{|\theta(\cdot-x')|}{|x'|^{\mez}}dx'\Big\|_{L^2}
\leq \cF(\|f\|_{W^{1,\infty}})\|\p_x^3f\|_{L^2}\delta^{\frac12}.
\end{split}
\end{align*}
We have proved that
\[
\cF(\|f\|_{W^{1,\infty}})\|G_{4,1}\|_{L^2}\leq  \cF(\|f\|_{W^{1,\infty}})\|\p_x^3f\|_{L^2}\delta^{\frac12}+\cF(\|f\|_{W^{1,\infty}})\|\p_x^2f\|_{L^2}\big(\delta^{-\frac12}+1\big).
\]
For any $\nu\in(0, 1)$, choosing
\[
\delta^\mez=\frac{\nu}{2^{10}\cF(\|f\|_{W^{1,\infty}})(1+\|\p_x^2f\|_{L^2})},
\]
we obtain
\bq\label{G41}
\cF(\|f\|_{W^{1,\infty}})\|G_{4,1}\|_{L^2}\leq  2^{-10}\nu\|\p_x^3f\|_{L^2}+\frac{1}{\nu}\cF(\|f\|_{W^{1,\infty}})\big(\|\p_x^2f\|_{L^2}+\|\p_x^2f\|_{L^2}^2\big).
\eq
As for $G_{4, 2}$,  we decompose  $G_{4,2}=\sum_{j=1}^4G_{4,2}^j$ where
\begin{equation*}
G_{4,2}^1(x)=\frac{-1}{4\pi}\int\frac{(\sinh(\delta_{x'}f(x))-\delta_{x'}f(x))\sinh(\delta_{x'}f(x))\delta_{x'}(\p_xf)(x)}{(\cosh(\delta_{x'}f(x))-\cos x')^2}\theta(x-x')dx',
\end{equation*}
\begin{equation*}
G_{4,2}^2(x)=\frac{-\p_x f(x)}{4\pi}\int\frac{(x'-\sin x')\sinh(\delta_{x'}f(x))\delta_{x'}(\p_xf)(x)}{(\cosh(\delta_{x'}f(x))-\cos x')^2}\theta(x-x')dx',
\end{equation*}
\begin{equation*}
G_{4,2}^3(x)=\frac{-1}{4\pi}\int_{|x'|>\delta}\frac{(\delta_{x'}f(x)-x'\p_x f(x))\sinh(\delta_{x'}f(x))\delta_{x'}(\p_xf)(x)}{(\cosh(\delta_{x'}f(x))-\cos x')^2}\theta(x-x')dx',
\end{equation*}
and
\begin{equation*}
G_{4,2}^4(x)=\frac{-1}{4\pi}\int_{|x'|<\delta}\frac{(\delta_{x'}f(x)-x'\p_x f(x))\sinh(\delta_{x'}f(x))\delta_{x'}(\p_xf)(x)}{(\cosh(\delta_{x'}f(x))-\cos x')^2}\theta(x-x')dx'.
\end{equation*}
Since the integral kernel in $G_{4,2}^1$ is bounded by $ \cF(\|f\|_{W^{1,\infty}})$, it follows that
\begin{equation*}\label{G421and2}
\|G_{4,2}^1\|_{L^2}+\|G_{4,2}^2\|_{L^2}\leq     \cF(\|f\|_{W^{1,\infty}})\| \tt\|_{L^2}\le  \cF(\|f\|_{W^{1,\infty}})\| \p_xf\|_{L^2}.
\end{equation*}
Next we estimate $G_{4, 2}^3$ and $G_{4, 2}^4$ as
\begin{align*}
\begin{split}\label{G423}
\|G_{4,2}^3\|_{L^2}\leq& \cF(\|f\|_{W^{1,\infty}}) \Big\|\int_{|x'|>\delta}\frac{|\theta(\cdot-x')|}{|x'|^{2}}dx'\Big\|_{L^2} \leq \cF(\|f\|_{W^{1,\infty}})\frac{\|\p_xf\|_{L^2}}{\delta}
\end{split}
\end{align*}
and
\begin{align*}
\begin{split}
\|G_{4,2}^4\|_{L^2}\leq& \cF(\|f\|_{W^{1,\infty}})\|\p_x^2f\|_{L^{\infty}}^2 \Big\|\int_{|x'|<\delta}|\theta(\cdot-x')|dx'\Big\|_{L^2} \leq \cF(\|f\|_{W^{1,\infty}})\|\p_x^2f\|_{L^2}\|\p_x^3f\|_{L^{2}}\delta.
\end{split}
\end{align*}
We thus obtain
\[
\cF(\|f\|_{W^{1,\infty}})\|G_{4,2}\|_{L^2}\le \cF(\|f\|_{W^{1,\infty}})\|\p_x^2f\|_{L^2}\|\p_x^3f\|_{L^{2}}\delta+\cF(\|f\|_{W^{1,\infty}})\|\p_xf\|_{L^2}(\delta^{-1}+1).
\]
For any $\nu\in (0, 1)$, choosing
\[
\delta=\frac{\nu}{2^{10}\cF(\|f\|_{W^{1,\infty}})(1+\|\p_x^2f\|_{L^2})}
\]
gives
\bq\label{G42b}
\cF(\|f\|_{W^{1,\infty}})\|G_{4,2}\|_{L^2}\le 2^{-10}\nu\|\p_x^3f\|_{L^{2}}+\frac{1}{\nu}\cF(\|f\|_{W^{1,\infty}})\|\p_x^2f\|_{L^2}.
\eq
Then plugging \eqref{G41} and \eqref{G42b} into \eqref{G41G42} we obtain
\bq\label{G4}
\|G_{4}\|_{L^2}\le 2^{-9}\nu\|\p_x^3f\|_{L^{2}}+\frac{1}{\nu}\cF(\|f\|_{W^{1,\infty}})\|\p_x^2f\|_{L^2}.
\eq
Finally, combining \eqref{G1}, \eqref{G2}, \eqref{G3}  and \eqref{G4} we conclude that
\[
\|\p_x[G(f)f]\|_{L^2}\leq \nu\|\p_x^3f\|_{L^{2}}+\frac{1}{\nu}\cF(\|f\|_{W^{1,\infty}})(\|\p_x^2f\|_{L^2}+\|\p_x^2f\|_{L^2}^2),
\]
which finishes the proof.
\end{proof}
\section{Viscosity solutions and comparison principle}\label{section:viscosity}
By virtue of the comparison principle for the Dirichlet-Neumann operator given by Proposition \ref{prop:comparisonDN} and the fact that $G(f+C)(f+C)=G(f)f$ for any constant $C$, we have the following natural definition of viscosity solutions for the Muskat problem \eqref{Muskat:DN}.
\begin{defi}[{\bf Viscosity solutions}]\label{def:viscosity}
 A function $f: \T\times [0, T]$ is called a viscosity subsolution (resp. supersolution) of \eqref{Muskat:DN} on $(0, T)$  provided that

(i) $f$   is upper semicontinuous (resp. lower semicontinuous) on $\T\times [0, T]$, and \\
(ii) for every $\psi:\T\times (0, T)\to \Rr$ with $\p_t\psi\in  C(\T\times (0, T))$ and $\psi\in  C((0, T); C^{1, 1}(\T))$, if $f-\psi$ attains a global maximum (resp. minimum) over $ \T\times [t_0-r, t_0]$ at $(x_0, t_0)\in \T\times (0, T)$ for some $r>0$, then
\bq\label{ineq:defviscosity}
\p_t \psi(x_0, t_0)\le -\ka \big(G(\psi)\psi\big)(x_0, t_0)\quad(\text{resp. }\ge).
\eq
A  viscosity solution is both a viscosity subsolution and viscosity supersolution.
\end{defi}
The next proposition shows that if a viscosity solution is regular ($C^{1,1}$) at a point $(x_0, t_0)$ then it satisfies the equation classically at the same point.
\begin{prop}[{\bf Consistency}]\label{prop:roughtest}
Let $f$ be a viscosity subsolution of \eqref{Muskat:DN} on $(0, T)$. Assume that $f\in W^{1, \infty}(\T\times (0, T))$ and $f$ is $C^{1, 1}$ at $(x_0, t_0)\in \T\times (0, T)$. 
 Then $G(f)f$ is classically well-defined   at $(x_0, t_0)$ and
\bq\label{sub:roughtest}
\p_t f(x_0, t_0)\le -\ka \big(G(f)f\big)(x_0, t_0).
\eq
The corresponding statement  for viscosity supersolutions holds true.
\end{prop}
\begin{proof}
According to Proposition \ref{prop:comparisonDN}, $G(f)f$ is classically well-defined at $(x_0, t_0)$. We shall write $X=(x, t)$ and $X_0=(x_0, t_0)$. Since $f$ is $C^{1,1}$ at $X_0$, for some $r_0\in (0, \frac{1}{10}\min\{\pi, t_0, T-t_0\})$ and $C>0$, we have
\bq\label{parabola:psibar}
f(X)\le f(X_0)+\na_Xf(X_0)\cdot (X-X_0)+\frac{C}{2}|X-X_0|^2:=\overline{\psi}(X),\quad |X-X_0|<r_0.
\eq
 Let $\psi$ be the tangent parabola with double opening:
 \bq\label{parabola:psi}
 \psi(X)=f(X_0)+\na_Xf(X_0)\cdot (X-X_0)+C|X-X_0|^2.
 \eq
We can find a family of functions $\psi_r \in C^\infty(\T\times \Rr)$, $r\in (0, r_0)$ satisfying
 \bq\label{properties:psir}
\begin{aligned}
&\psi_r\le \psi\quad\text{in } B_{r_0}(X_0),\\
& \psi_r\ge f\quad\text{on } \T\times [0, T],\\
&\forall \delta\in (0, r_0^2),\quad\psi_r\to f\quad\text{in}\,\, C(\T\times [\delta, T-\delta]),\\
&\forall \delta\in (0, r_0^2),\quad \sup_{r\in (0, \sqrt{\delta})}\| \psi_r\|_{W^{1, \infty}(\T\times [\delta, T-\delta])}<\infty.
\end{aligned}
\eq
The construction of $\psi_r$ is postponed to Appendix \ref{appendix:psir}. The properties in \eqref{properties:psir} imply that $\psi_r(t_0)$'s are uniformly $C^{1, 1}$ at $x_0$, i.e. there exists
  $M_0>0$ independent of $r\in (0, r_0)$ such that
 \bq\label{psir:uniC11}
 |\psi_r(x_0+x, t_0)+\psi_r(x_0-x, t_0)-2\psi_r(x_0, t_0)|\le M|x|^2\quad\forall |x|<r_0.
 \eq
 Since  each $\psi_r$ is a valid test function for the viscosity subsolution $f$, we have
\bq\label{sub:testpsir}
\p_t \psi_r(X_0)\le -\ka \big(G(\psi_r)\psi_r\big)(X_0).
\eq
At the maximum point $X_0$ of $f-\psi_r$, we have  $\p_t f=\p_t \psi_r$. Thus \eqref{sub:roughtest} will be a consequence of \eqref{sub:testpsir} and
\bq\label{conver:DNpsir0}
\lim_{r\to 0} \big(G(\psi_r)\psi_r\big)(X_0)=\big(G(f)f)(X_0).
\eq
We shall skip the time variable in the remainder of the proof because $t=t_0$ is fixed in \eqref{conver:DNpsir0}. The proof of  \eqref{conver:DNpsir0} proceeds in two steps.

{\bf Step 1.} Let $\phi_r$ (reps. $\phi$) be the harmonic extension of $\psi_r$ to $\Omega_r\equiv \Omega_{\psi_r}$ (resp. $\Omega_f$).
We claim that
\bq\label{conv:phir}
\lim_{r\to 0} \phi_r(x, y)=\phi(x, y)\quad\forall (x, y)\in \Omega_f,
\eq
where we note that $\Omega_f\subset \Omega_r$ for all $r$.  From \eqref{inf:phij}, \eqref{sup:phij} and the uniform boundedness of $\| \psi_r\|_{W^{1, \infty}(\T)}$, we have
\bq\label{uni:phir}
\| \phi_r\|_{L^\infty(\Omega_r)}\le \| \psi_r\|_{L^\infty(\T)}\le M.
\eq
Fix a sequence of bounded subsets $\Omega^j\Subset \Omega_f$ satisfying $\Omega^j\subset \Omega^{j+1}$ and $\cup_{j\ge 1}\Omega^j=\Omega_f$. Let $r_n$ be an arbitrary sequence of positive numbers converging to $0$. We relabel $\phi_n=\phi_{r_n}$ and $\Omega_n\equiv \Omega_{r_n}$. The uniform bound \eqref{uni:phir} implies that every subsequence of $\phi_n$ has a subsequence converging weakly-* in $L^\infty(\Omega_f)$. For notational simplicity, we write $\phi_n\to \phi_\infty$ weakly-* in $L^\infty(\Omega_f)$. We have the gradient estimate for the harmonic functions $\phi_j$
\[
\| \nabla_{x, y}\phi_n\|_{L^\infty(\Omega^j)}\le C_j\| \phi_n\|_{L^\infty(\Omega_f)}\le C_jM.
\]
By virtue of the Arzel\`a-Ascoli theorem, for every $j$ there exists a subsequence of $\phi_n$ converging to $\phi_\infty$ in $C(\overline{\Omega^j})$. Using the Cantor diagonal argument, we can find a subsequence $\phi_{n_k}\to \phi_\infty$ in $C(\overline{\Omega^j})$ for all $j$. Thus \eqref{conv:phir} will follow once we can prove that
\bq\label{phiinfty}
\phi_\infty=\phi.
\eq
 This can be achieved by using  the variational characterization of $\phi_{n_k}$ and $\phi$. Indeed, we recall from \eqref{def:vari} and \eqref{form:vari} that $\phi_n=u_n+\underline{\psi_n}$, where $\underline{\psi_n}\in \dot H^1(\Omega_n)$ with $\underline{\psi_n}(x, \psi_n(x))=\psi_n(x)$ (in the trace sense)
and $u_n\in \dot H^1_0(\Omega_n)$ satisfies
\bq\label{variform:phin}
\int_{\Omega_n}\na_{x, y}u_{n}\cdot \nabla_{x, y}\varphi dxdy=-\int_{\Omega_n}\na_{x, y}\underline{\psi_n}\cdot \nabla_{x, y}\varphi dxdy\quad\forall \varphi\in \dot H^1_0(\Omega_n).
\eq
In particular, the uniform bounds
\begin{align*}
&\| u_n\|_{\dot H^1(\Omega_n)}\le \| \underline{\psi_n}\|_{\dot H^1(\Omega_n)}\le C(1+\| \psi_n\|_{\Lip(\T^d)})\| \psi_n\|_{\dot H^\mez(\T^d)}\le \wt C,\\
&\| \phi_n\|_{\dot H^1(\Omega_n)}\le 2\| \underline{\psi_n}\|_{\dot H^1(\Omega_n)}\le 2\wt C
\end{align*}
hold. Since $\Omega_f\subset \Omega_n$, upon passing to a subsequence of $n_k$, we have
\bq\label{weakconv:phij}
u_{n_k}\wc u_\infty\quad\text{in } \dot H^1(\Omega_f),\quad \phi_{n_k}\wc \phi_\infty\quad\text{in } \dot H^1(\Omega_f).
\eq
As for the convergence of $\underline{\psi_n}$, we recall from Theorem \ref{theo:lifting} that we can choose $\underline{\psi_n}(x, y)=\psi_{\sharp}^n(x, y-\psi_n(x))$, where $\psi_{\sharp}^n(x, z)=e^{z|D_x|}\psi_n(x)$ for $(x, z)\in \T\times \Rr_-$. Since $\psi_n\to f$ in $C^\alpha(\T)$ for all $\alpha\in (0, 1)$,   $\psi_{\sharp}^n \to f_{\sharp}:=e^{z|D_x|}f(x)$ in $L^\infty(\T\times \Rr_-)$. Consequently, $\underline{\psi_n}(x, y)\to \underline{f}(x, y):=f_\sharp(x, y-f(x))$ in $L^\infty(\Omega_f)$. Then for all $\tt\in C^\infty_c(\overline{\Omega_f})$,  using integration by parts and the dominated convergence theorem gives
\begin{align*}
\int_{\Omega_f}\na_{x, y}\underline{\psi_n}\cdot\na_{x, y}\tt dxdy&=\int_{\T}\psi_n(x)\p_N\tt(x, f(x))dx-\int_{\Omega_f}\underline{\psi_n}\Delta_{x, y}\tt dxdy\\
&\to \int_{\T}f(x)\p_N\tt(x, f(x))dx-\int_{\Omega_f}\underline{f}\Delta_{x, y}\tt dxdy=\int_{\Omega_f}\na_{x, y}\underline{f}\cdot\na_{x, y}\tt dxdy,
\end{align*}
where $N=(-\p_x f, 1)$. As $ C^\infty_c(\overline{\Omega_f})$ is dense in $\dot H^1(\Omega_f)$, it follows that
\bq\label{wc:upsin}
\underline{\psi_n}\wc \underline{f}\quad \dot H^1(\Omega_f).
\eq
Recalling that $\phi_{n_k}=u_{n_k}+\underline{\psi_{n_k}}$, we deduce from  \eqref{weakconv:phij} and \eqref{wc:upsin} that
\bq\label{phiinfty:2}
\phi_\infty=u_\infty+\underline{f}.
\eq
On the other hand, since $\dot H^1_0(\Omega_f)\subset \dot H^1_0(\Omega_n)$, \eqref{variform:phin} implies
\[
\int_{\Omega_f}\na_{x, y}u_{n}\cdot \nabla_{x, y}\varphi dxdy=-\int_{\Omega_f}\na_{x, y}\underline{\psi_n}\cdot \nabla_{x, y}\varphi dxdy\quad\forall \varphi \in \dot H^1_0(\Omega_f).
\]
By virtue of \eqref{weakconv:phij} and \eqref{wc:upsin}, letting $n=n_k\to \infty$ in the preceding equality we obtain
\[
\int_{\Omega_f}\na_{x, y}u_\infty\cdot \nabla_{x, y}\varphi dxdy=-\int_{\Omega_f}\na_{x, y}\underline{f}\cdot \nabla_{x, y}\varphi dxdy\quad\forall \varphi \in \dot H^1_0(\Omega_f).
\]
This together with \eqref{phiinfty:2} proves that $\phi_\infty$ is the harmonic extension of $f$ to $\Omega_f$ in the variational sense. By the uniqueness of  variational solutions we conclude that $\phi_\infty=\phi$, proving \eqref{phiinfty}.

{\bf Step 2.} Recall that $f\in W^{1, \infty}(\T)$ is $C^{1, 1}$ at $x_0$. In addition, $\psi_r\in C^\infty(\T)$ is tangent to $f$ from above at $x_0$, so that $N_{r}(x_0)=N(x_0)$ and there exists an interior disk $B$ tangent to  $\Omega_f$ at $z_0=(x_0, f(x_0))$. Clearly, $B$ is also tangent to all $\Omega_r$. Assume without loss of generality that $B=B_1(0)$.  In addition, in view of the uniform $C^{1, 1}$ property \eqref{psir:uniC11}, Corollary \ref{coro:pointwiseelliptic} implies that $\phi_r$'s are uniformly $C^{1, \alpha}$ at $(x_0, f(x_0))$. In particular, there exists  $M>0$ such that for all $r <r_0$, the trace $g_r=\phi_r\vert_B$ satisfies
\bq\label{uni:gr}
|g_r(x_0+x)+g_r(x_0-x)-2g_r(x_0)|\le M|x|^{1+\alpha}\quad\forall |x|<r_0.
\eq
Let $g=\phi\vert_B$. Proposition \ref{prop:disk} then gives
\begin{align*}
 &\big(G(\psi_r)\psi_r\big)(x_0)=-\frac{|N(x_0)|}{8\pi}\int_{-\pi}^{\pi}\frac{g_r(x_0+x)+g_r(x_0-x)-2g_r(x_0)}{\sin^2(\frac{x}{2})}dx,\\
 &\big(G(\psi)\psi\big)(x_0)=-\frac{|N(x_0)|}{8\pi}\int_{-\pi}^{\pi}\frac{g(x_0+x)+g(x_0-x)-2g(x_0)}{\sin^2(\frac{x}{2})}dx.
\end{align*}
The convergence \eqref{conv:phir} implies that $g_r(x)\to g(x)$ for all $x\in \overline{B}$. The uniform bounds \eqref{uni:phir} and \eqref{uni:gr} then allow us to apply the dominated convergence theorem to conclude the claim \eqref{conver:DNpsir0}.
\end{proof}
Using Proposition \ref{prop:roughtest}, we prove that viscosity solutions obey the comparison principle. An immediate consequence is the uniqueness of viscosity solutions.
\begin{theo}[{\bf Comparison principle}]\label{theo:comparison:viscosity}
Assume that $f,~g:\T\times [0, T]\to \Rr$ are respectively a bounded viscosity subsolution  and supersolution of \eqref{Muskat:DN} on $(0, T)$. If $f(x, 0)\le g(x, 0)$ for all $x\in \T$, then $f(x, t)\le g(x, t)$ for all $(x, t)\in \T\times [0, T]$.
\end{theo}
\begin{proof}
For  classical sub and supersolutions, one can follow the proof of Proposition \ref{comparison:Muskatr}. Here the source difficulty comes from the low regularity of $f$ and $g$. We employ the regularization  using sup- and inf-convolutions. For small $\delta>0$, let $f^\delta$ and $g_\eps$ be respectively the  sup-convolution and inf-convolution of $f$ and  $g$:
\begin{align*}
&f^\delta(x, t)=\sup_{(y, s)\in \T\times [0, T]} f(y, t)-\frac{1}{2\delta}\big(|x-y|^2+|t-s|^2\big),\\
&g_\delta(x, t)=\inf_{(y, s)\in \T\times [0, T]} g(y, t)+\frac{1}{2\delta}\big(|x-y|^2+|t-s|^2\big)
\end{align*}
for $(x, t)\in \T\times [0, T]$. Since $f$ (reps. $g$) is upper (resp. lower) semicontinuous, the  supremum (reps. infimum) in the definition of $f^\delta$ (reps. $g_\delta$)  is in fact a maximum (reps. minimum).  Clearly, $f^\delta$ and $g_\delta$ are $2\pi$-periodic in $x$. We record the following standard properties of sup- and inf-convolution  (see e.g. \cite{CrandallIshiiLion}).
\begin{itemize}
\item[(i)] $f^\delta,\, g_\delta\in \Lip(\T\times [0, T])$.
\item[(ii)] $f^\delta$ (reps. $g_\delta)$ is semiconvex  (reps. semiconcave) in the senses that each point has a tangent paraboloid from below (reps. above) with opening $\delta^{-1}$.
\item[(iii)] The half-relaxed limits
\bq\label{Gamma:conv}
\limsup_{\substack{\delta\to 0\\ (y, s)\to (x, t)}}f^\delta(y, s)=f(x, t),\quad
\liminf_{\substack{\delta\to 0\\(y, s)\to (x, t)}}g_\delta(y, s)=g(x, t)
\eq
hold for all $(x, t)\in \T\times [0, T]$.
\end{itemize}
The boundedness of $f$ and $g$ is only used in (i).

Next we prove that each $f^\delta$ (reps. $g_\delta$) is a viscosity subsolution (reps. supersolution) of \eqref{Muskat:DN}. This follows from the translation invariance of \eqref{Muskat:DN}. Indeed, let $\psi$ be a test function as in Definition \ref{def:viscosity} such that $f^\delta-\psi$ has a global maximum over $\T\times [t_0-r, t_0]$ at  $(x_0, t_0)$ for some $r>0$. Let $(y_0, s_0)\in \T\times [0, T]$ be the point where the maximum in the definition of $f^\delta(x_0, t_0)$ is attained.  We have that
\begin{align*}
&f^\delta(x_0, t_0)= f(y_0, s_0)-\frac{1}{2\delta}\big(|x_0-y_0|^2+|t_0-s_0|^2\big),\\
& f^\delta(x, t)\ge f(x-x_0+y_0, t-t_0+s_0)-\frac{1}{2\delta}\big(|x_0-y_0|^2+|t_0-s_0|^2\big).
\end{align*}
Therefore, the function
\[
\wt \psi(x, t)=\psi(x+x_0-y_0, t+t_0-s_0)+\frac{1}{2\delta}\big(|x_0-y_0|^2+|t_0-s_0|^2\big)
\]
is a valid test function for which  $f-\wt\psi$ has a global maximum over $\T\times [s_0-r, s_0]$ at $(y_0, s_0)$. Therefore,  $\p_t \wt\psi(y_0, s_0)\le -\ka \big(G(\wt\psi)\wt\psi\big)(y_0, s_0)$, and hence $\p_t \psi(x_0, t_0)\le -\ka \big(G(\psi)\psi\big)(x_0, t_0)$. Thus each $f^\delta$ is viscosity subsolution of \eqref{Muskat:DN}. By an analogous argument, each $g_\delta$ is viscosity supersolution of \eqref{Muskat:DN}.

We claim  that for very $\eps>0$, there exists  $\delta(\eps)>0$ such that for all $\delta\le \delta(\eps)$,
\bq\label{comparision:reg}
f^\delta(x, t)\le g_\delta(x, t)+\eps\quad\forall (x, t)\in \T\times [0, T].
\eq
Taking this for granted, the half-relaxed limits in \eqref{Gamma:conv} yield $f(x, t)\le g(x, t)$ for all $(x, t)\in \T\times [0, T]$, proving the comparison principle.  To prove \eqref{comparision:reg}, we assume by contradiction that for some $\eps_0>0$, there exists a sequence $\delta_n\to 0$ such that for all $n$,
 \[
M_n:=\max_{\T\times [0, T]}(f^{\delta_n}-g_{\delta_n})=(f^{\delta_n}-g_{\delta_n})(x_n, t_n)>\eps_0.
\]
Because
\[
\lim_{n\to \infty}(f^{\delta_n}(x, 0)-g_{\delta_n}(x, 0))\le f(x, 0)-g(x, 0)\le 0,
\]
we have $f^{\delta_n}(x, 0)-g_{\delta_n}(x, 0)<\frac{\eps_0}{2}$  for all  $n\ge n_0$. This implies in particular that $t_n>0$ for all $n\ge n_0$.  Choose $\eta>0$ sufficiently small so that
 \[
 (f^{\delta_{n_0}}-g_{\delta_{n_0}}-\eta t)(x_{n_0}, t_{n_0})=M_{n_0}-\eta t_{n_0}>\frac{2\eps_0}{3},
 \]
 and hence
\[
M_*:=\max_{\T\times [0, T]}(f^{\delta_{n_0}}-g_{\delta_{n_0}}-\eta t)>\frac{2\eps_0}{3}.
\]
 Moreover, $M_*$  is attained at some point $(x_*, t_*)$ with $t_*>0$ since $\max_{\T}(f^\delta-g_\delta-\eta t)\vert_{t=0}\le \frac{\eps_0}{2}$. Consequently,
 \[
 \max_{\T\times [0, T]}(f^{\delta_{n_0}}-g_{\delta_{n_0}}-\eta t-M_*)=0
 \]
 and is attained at $(x_*, t_*)$. In what follows we shall write $\delta_{n_0}=\delta$ to alleviate the notation. Then the smooth function $\eta t+M_*$ touches $f^\delta-g_\delta$ from above at $(x_*, t_*)$, so that $f^\delta-g_\delta$ has a tangent paraboloid from above at $(x_*, t_*)$. On the other hand, both $f^\delta$ and $-g_\delta$ have a tangent paraboloid from below at $(x_*, t_*)$, hence they are $C^{1, 1}$ at $(x_*, t_*)$. Then we can apply Proposition \ref{prop:roughtest}  to have that $f^\delta$ and $g_\delta$ are classical sub and supersolutions  at $(x_*, t_*)$:
 \[
 \p_t f^\delta (x_*, t_*)\le -\ka \big(G(f^\delta)f^\delta\big)(x_*, t_*),\quad \p_t g_\delta (x_*, t_*)\ge -\ka \big(G(g_\delta)g_\delta\big)(x_*, t_*).
 \]
 By virtue of Proposition \ref{prop:comparisonDN},
 \[
 G(f^\delta)f^\delta(x_*, t_*)\ge G(g_\delta+\eta t+M_*)(g_\delta+\eta t+M_*)(x_*, t_*).
 \]
In addition, since $t_*>0$ we have
\[
\p_t(g_\delta+\eta t+M_*)(x_*, t_*)\le \p_t f^\delta(x_*, t_*).
\]
 It follows that
\begin{align*}
\p_t(g_\delta+\eta t+M_*)\le \p_t f^\delta&\le  -\ka G(f^\delta)f^\delta\\
&\le  -\ka G(g_\delta+\eta t+M_*)(g_\delta+\eta t+M_*)\\
&= -\ka G(g_\delta)g_\delta\le \p_t g_\delta
\end{align*}
at $(x_*, t_*)$. This leads to the contradiction $\eta\le 0$.
 \end{proof}
By virtue of Theorem \ref{theo:comparison:viscosity}, we have
 \begin{coro}
Proposition \ref{prop:maxslop} is valid for viscosity solutions.
 \end{coro}
\section{Proof of Theorem \ref{theo:main}}
 Let $f_0$ be an arbitrary initial surface in $W^{1, \infty}(\T)$. We construct a global solution to the contour dynamics formulation \eqref{reform:f}-\eqref{reform:tt} by regularizing initial data and the Muskat problem. Precisely, for $\eps\in (0, 1)$, let $f_0^\eps=\Gamma_\eps*f_0(x)$ be the mollification of $f_0$ where $\Gamma_\eps$ is an approximation of the identity. Since $f_0^\eps\in H^s(\T)$ for any $s>0$,   Proposition \ref{GlobalEpsilonProblem} yields the unique global solution $f^\eps$ to the regularized Muskat equation \eqref{Muskat:r} with initial data $f_0^\eps$. Moreover,  $f^\eps$ is smooth and in view of \eqref{max:C1},
 \bq\label{ub:f}
 \begin{aligned}
& \| f^\eps(t)\|_{L^\infty(\T)}\le \| f^\eps_0\|_{L^\infty(\T)}\le \| f_0\|_{L^\infty(\T)},\\
& \| f^\eps(t)\|_{\Lip(\T)}\le \| f^\eps_0\|_{\Lip(\T)}\le \| f_0\|_{\Lip(\T)}.
 \end{aligned}
 \eq
Setting $\tt^\eps=(\mez I-K^*[f^\eps])(-\ka f^\eps)$, we deduce from Proposition \ref{prop:potential} and \eqref{ub:f} that
\bq\label{ub:tt}
\| \tt^\eps(t)\|_{L^2(\T)}\le C(1+\| f(t)\|_{\Lip(\T)})^\frac{5}{2}\| \p_xf(t)\|_{L^2(\T)}\le C(1+\|f_0\|_{\Lip(\T)})^\frac{7}{2}\quad\forall t\ge 0.
\eq
Let $\eps_n\to 0$ and relabel $f_n=f^{\eps_n}$ and $\tt_n=\tt^{\eps_n}$. From the uniform bounds \eqref{ub:f} and \eqref{ub:tt}, we obtain the weak* convergences (upon extracting subsequences)
\bq\label{weak*}
f_n \overset{\ast}{\rightharpoonup} f~\text{in}~L^\infty([0, \infty); W^{1, \infty}(\T)),\quad \tt_n \overset{\ast}{\rightharpoonup} \tt~\text{in}~L^\infty([0, \infty); L^2(\T)).
\eq
In particular, \eqref{ub:f} implies that
\bq
 \| f(t)\|_{L^\infty(\T)}\le  \| f_0\|_{L^\infty(\T)},\quad  \| f(t)\|_{\Lip(\T)}\le  \| f_0\|_{\Lip(\T)}.
\eq
We now prove that $f$ and $\tt$ satisfy \eqref{reform:f}-\eqref{reform:tt} for all $t>0$. Fix an arbitrary time $T>0$.

{\bf Step 1: Strong convergence in $C( \T\times [0, T])$}. Combining the $L^2$ estimate \eqref{velocitybound} for $G(f)g$ and the uniform Lipschitz bound \eqref{ub:f}, we have
\[
\| G(f_n)f_n\|_{L^\infty([0, T]; L^2(\T))}\le C(1+\| f_0\|_{\Lip(\T)})^2\| f_0\|_{\Lip(\T)}
\]
and
\[
 \| \p^2_xf_n\|_{L^\infty([0, T]; H^{-1})}\le  \| \p_xf_n\|_{L^\infty([0, T]; L^2)}\le C \| f_n\|_{L^\infty([0, T]; \Lip(\T))}\le C\| f_0\|_{\Lip(\T)}.
 \]
 It then follows from equation \eqref{Muskat:r} that $\p_t f^\eps$ is uniformly bounded in $L^\infty([0, T]; H^{-1}(\T))$. We have the continuous embedding $W^{1, \infty}(\T)\subset C(\T)\subset H^{-1}(\T)$ where the first one is compact by the Azel\`a-Ascoli theorem. Thus the Aubin-Lions lemma \cite{Lions1969} implies that
 \bq\label{strong:cv}
 f_n\to f\quad\text{in}~C(\T\times [0, T]).
 \eq
 In particular, $\lim_{t\to 0^+}f(\cdot ,t)=f_0(\cdot)$.

 {\bf Step 2: $f$ is the unique viscosity solution}. It suffices to prove that $f$ is a viscosity subsolution since its uniqueness then follows at once from Theorem \ref{theo:comparison:viscosity}. Indeed, assume that for  $\psi:\T\times (0, T)\to \Rr$ with $\p_t\psi\in  C(\T\times (0, T))$ and $\psi\in  C((0, T); C^{1, 1}(\T))$, $f-\psi$ attains a global maximum over $ \T\times [t_0-r, t_0]$ at $(x_0, t_0)\in \T\times (0, T)$ for some $r>0$. Setting $\overline{\psi}(x, t)=\psi(x, t)+|t-t_0|^2$ we have that $f-\overline\psi$ attains a strict global maximum over $\T\times [t_0-r, t_0]$ at $(x_0, t_0)$.  By the uniform convergence \eqref{strong:cv}, there exists for each sufficiently large $n$ a point $(x_n ,t_n)\in \T\times [t_0-r, t_0]$ such that $f_n-\overline\psi$ attains a  global maximum at $(x_n, t_n)$ and $(x_n, t_n)\to (x_0, t_0)$. Set
 \[
 \wt \psi_n=\overline\psi+M_n,\quad M_n:=\max_{\T\times [t_0-r, t_0]}(f_n-\overline\psi)
 \]
  so that $f_n-\wt\psi_n$ attains a zero  global maximum over $\T\times [t_0-r, t_0]$ at $(x_n, t_n)$. It follows that  \bq\label{proofviscosity:1}
   \p_t\wt\psi_n(x_n, t_n)\le \p_tf_n(x_n, t_n),\quad   \p_x\wt\psi_n(x_n, t_n)= \p_xf_n(x_n, t_n).
\eq
The comparison principle in Proposition \ref{prop:comparisonDN}  gives
  \bq\label{proofviscosity:2}
  \big(G(f_n)f_n\big)(x_n, t_n)\ge \big(G(\wt \psi_n)\wt\psi_n\big)(x_n, t_n).
  \eq
  For any function $g\in C^{1, 1}(\T)$, we denote the generalized second order derivative of $g$ at $x$ by
  \[
 \p^{2, *}g(x)=\{a: \exists  x_j\to x~\text{with}~g~\text{twice differentiable at}~x_j~\text{and}~\p_x^2 g(x_j)\to a\}.
  \]
Since $g\in C^{1, 1}(\T)$, $\p^2g(x)$ exists for almost every $x\in \T$ and $|\p^2g(x)|\le \| \p g\|_{\Lip(\T)}$. Consequently,  $ \p^{2, *}g(x)$ is nonempty for all $x\in \T$ and $|a|\le \| \p g\|_{\text{Lip}}$ for all $a\in  \p^{2, *}g(x)$. If in addition $g$ is twice differentiable near $x$ and $\p^2g$ is continuous at $x$  then $\p^{2, *}g(x)=\{\p^2g(x)\}$. According to  the second order optimality condition proved in \cite{OptimalityCondition}, if $g\in C^{1, 1}(\T)$ has a local maximum at $x$ then there exists $a\in  \p^{2, *}g(x)$ such that $a\le 0$. Applying this with $g(x)=f_n(x, t_n)-\wt \psi_n(x, t_n)$ we find that
\[
\p_x^2f_n(x_n, t_n)\le a_n\quad\text{for some}~a_n\in \p_x^{2, *}\wt\psi_n(x_n, t_n).
\]
Combining this with \eqref{proofviscosity:1} and \eqref{proofviscosity:2} yields
  \[
  \begin{aligned}
  \p_t\wt\psi_n(x_n, t_n)&\le \p_tf_n(x_n, t_n)=-\ka \big(G(f_n)f_n\big)(x_n, t_n)+\frac{1}{n}\p_x^2f_n(x_n, t_n)\\
  &\le -\ka \big(G(\wt\psi_n)\wt\psi_n\big)(x_n, t_n)+\frac{1}{n}a_n.
  \end{aligned}
  \]
  Consequently,
  \[
  \p_t\overline{\psi}(x_n, t_n)\le -\ka \big(G(\overline\psi)\overline\psi\big)(x_n, t_n)+\frac{1}{n}a_n.
   \]
  Since
  \[
 \sup_{n\in \Nn}|a_n|\le  \sup_{n\in \Nn} \| \p_x\wt\psi_n\|_{L^\infty([t_0-r, t_0]; \text{Lip}(\T))}=\| \p_x\psi\|_{L^\infty([t_0-r, t_0]; \Lip(\T))}
  \]
  and  $\p_t\overline{\psi}\in  C(\T\times (0, T))$,  letting $n\to \infty$ in the preceding inequality yields
  \[
   \p_t\overline{\psi}(x_0, t_0)\le -\ka \lim_{n\to \infty} \big(G(\psi)\psi\big)(x_n, t_n),
   \]
   where we used the fact that $G(\overline\psi)\overline\psi=G(\psi)\psi$. We claim that
   \bq\label{lim:Gn}
  \lim_{n\to \infty} \big(G(\psi)\psi\big)(x_n, t_n)= \big(G(\psi)\psi\big)(x_0, t_0).
   \eq
   Taking this for granted and noticing that $\p_t \overline{\psi}(x_0, t_0)=\p_t \psi(x_0, t_0)$  we obtain
  \[
  \p_t\psi(x_0, t_0)\le -\ka \big(G(\psi)\psi\big)(x_0, t_0),
  \]
whence $f$ is a viscosity subsolution. Analogously, $f$ is also a viscosity supersolution. To prove the claim \eqref{lim:Gn}  we write
\begin{align*}
\big(G(\psi)\psi\big)(x_n, t_n)-\big(G(\psi)\psi\big)(x_0, t_0)&=\big(G(\psi(t_n))\psi(t_n)\big)(x_n)-\big(G(\psi(t_0))\psi(t_0)\big)(x_n)\\
&\quad+\big(G(\psi(t_0))\psi(t_0)\big)(x_n)-\big(G(\psi(t_0))\psi(t_0)\big)(x_0)\\
&=N^n_1+N^n_2.
\end{align*}
  For $\psi(t_0)\in C^{1, 1}(\T)$, $G(\psi(t_0))\psi(t_0)\in C^\alpha(\T)$ for all $\alpha\in (0, 1)$ and thus $N^n_2\to 0$. As for $N^n_1$ we apply the continuity and contraction estimates in Propositions \ref{prop:estDN} and \ref{theo:contraDN} to have
  \begin{align*}
  |N_1^n|&\le \|G(\psi(t_n))\psi(t_n)-G(\psi(t_0))\psi(t_0) \|_{C(\T)}\\
  &\le\|G(\psi(t_n))\psi(t_n)-G(\psi(t_0))\psi(t_0) \|_{H^1(\T)}\\
   & \le\|G(\psi(t_n))[\psi(t_n)-\psi(t_0)]\|_{H^1(\T)}+\| G(\psi(t_n))\psi(t_0)-G(\psi(t_0))\psi(t_0) \|_{H^1(\T)}\\
  &\le \cF\big(\| (\psi(t_n), \psi(t_0))\|_{H^2(\T)}\big)\| {\psi}(t_n)-{\psi}(t_0)\|_{H^2(\T)}\to 0.
  \end{align*}
since $\psi\in C((0, T); C^{1, 1}(\T))$. This finishes the proof of \eqref{lim:Gn}.

   We shall prove in the following steps that $f$ and $\tt$ solve \eqref{reform:f}-\eqref{reform:tt} by passing to the limit in the equivalent  formulation \eqref{Muskat:rGE} of \eqref{Muskat:r}.

{\bf Step 3: Equation \eqref{reform:tt}}.  For any $\vp\in C_c^\infty(\T\times [0, T))$, the second equation in \eqref{Muskat:rGE} together with \eqref{def:K*} gives
\begin{equation}\label{thetaepsweak}
 \begin{aligned}
\mez\int_0^T\!\!\int_\T\vp(x,t)\tt_n(x,t)dxdt-I_1^n
=-\kappa\int_0^T\!\!\int_\T\vp(x,t)\p_xf_n(x,t)dxdt,
\end{aligned}
\end{equation} 	
where
$$
I_1^n=\int_0^T\!\!\int_\T\int_{\T}\p_x\vp(x,t)\frac{1}{2\pi}\arctan\Big(\frac{\tanh(\frac{f_n(x,t)-f_n(x',t)}{2})}{\tan(\frac{x-x'}{2})}\Big)\tt_n(x',t)dx'dxdt.
$$
The first and the last integrals in \eqref{thetaepsweak} converge to their corresponding limit in view of the weak* convergences in \eqref{weak*}. We show now that
$$
I^n_1\to I_1=\int_0^T\!\!\int_\T\int_{\T}\p_x\vp(x,t)\frac{1}{2\pi}\arctan\Big(\frac{\tanh(\frac{f(x,t)-f(x',t)}{2})}{\tan(\frac{x-x'}{2})}\Big)\tt_n(x',t)dx'dxdt\quad \mbox{as}\quad n\to\infty.
$$
We split $I_1^n-I_{1}=D_1^n+D_2^n$ where
\begin{align*}
\begin{split}
D_1^n=&\int_0^T\!\!\int_\T\tt_n(x',t)\int_{\T\cap\{|x|>\delta\}}\p_x\vp(x\!+\!x',t)\frac{1}{2\pi}\arctan\Big(\frac{\tanh(\frac{f_n(x+x',t)-f_n(x',t)}{2})}{\tan(\frac{x}{2})}\Big)dxdx'dt\\
&-\int_0^T\!\!\int_\T\theta(x',t)\int_{\T\cap\{|x|>\delta\}}\p_x\vp(x\!+\!x',t)\frac{1}{2\pi}\arctan\Big(\frac{\tanh(\frac{f(x+x',t)-f(x',t)}{2})}{\tan(\frac{x}{2})}\Big)dxdx'dt
\end{split}
\end{align*}
and
\begin{align*}
\begin{split}	
	D_2^n=&\int_0^T\!\!\int_\T\tt_n(x',t)\int_{\T\cap\{|x|<\delta\}}\p_x\vp(x\!+\!x',t)\frac{1}{2\pi}\arctan\Big(\frac{\tanh(\frac{f_n(x+x',t)-f_n(x',t)}{2})}{\tan(\frac{x}{2})}\Big)dxdx'dt\\
	&-\int_0^T\!\!\int_\T\theta(x',t)\int_{\T\cap\{|x|<\delta\}}\p_x\vp(x\!+\!x',t)\frac{1}{2\pi}\arctan\Big(\frac{\tanh(\frac{f(x+x',t)-f(x',t)}{2})}{\tan(\frac{x}{2})}\Big)dxdx'dt.
\end{split}
\end{align*}
From the  uniform bound \eqref{ub:tt} and the obvious inequality $|\arctan(\cdot)|\le \frac{\pi}{2}$, it is readily seen that
$$|D_2^n|\leq \delta C\|\p_x\vp\|_{L^\infty} (1+\| f_0\|_{\Lip(\T)})^{\frac72}.$$
The strong convergence \eqref{strong:cv} of $f_n$ combined with the weak* convergence of $\theta_n$ implies that for any $\delta\in (0, 1)$,
$$
\lim_{n\to \infty}|D_1^{n}|=0.
$$
Thus by sending $n\to \infty$ and subsequently $\delta\to 0$, we conclude that $I_1^n\to I_1$. Consequently,
\begin{align*}
\mez \int_0^T\int_\T \tt(x, t) \vp(x, t) dxdt&-p.v.\frac{1}{2\pi}\int_0^T\int_\T\p_{x}\vp(x, t) \int_{\T}\arctan\Big(\frac{\tanh(\frac{f(x, t)-f(x', t)}{2})}{\tan(\frac{x-x'}{2})}\Big)\tt(x', t)dx'\\
&\qquad=-\ka \int_0^T\int_\T\vp(x, t)\p_xf(x, t)dxdt.
\end{align*}
For $f\in W^{1, \infty}(\T)$ and $\tt\in L^2(\T)$, we have  $K^*[f](\tt)\in L^2(\T)$ and hence integrating by parts in the second integral yields \eqref{reform:tt} in $L^\infty_t L^2_x$.

{\bf Step 4: Equation \eqref{reform:f}}.  We proceed analogously for the $f_n$ equation in \eqref{Muskat:rGE}. We multiply \eqref{Muskat:rGE} by $\vp\in C^\infty(\T\times [0, T))$ and use  \eqref{DN:contour1} to integrate by parts, leading to
\begin{align}
\begin{split}\label{fepsweak}
\int_0^T\!\!\!\int_\T\p_t\vp(x,t)f_n(x,t)dxdt+\int_\T\vp(x,0)(\Gamma_{\eps_n}&*f_0)(x)dx=I_2^{n}\\
&-\eps_n\int_0^T\!\!\!\int_\T\p_x\vp(x,t)\p_xf_n(x,t)dxdt,
\end{split}
\end{align}
where
$$
I_2^{n}=\frac{1}{4\pi}\int_0^T\!\!\!\int_\T
\p_x\vp(x,t)\int_\T\ln\Big(\cosh(f_n(x,t)\!-\!f_n(x',t))\!-\!\cos(x\!-\!x')\Big)\tt_n(x',t)dx'dxdt.
$$
By \eqref{weak*}, the first integral in \eqref{fepsweak} converges to the same integral with $f$ in place of $f_n$. Since $\Gamma_{\eps_n}*f_0\to f_0$ in $C(\T)$, the second integral in \eqref{fepsweak} converges to $\int_\T\vp(x,0) f_0(x)dx$. The last integral tends to zero owing to the uniform Lipschitz bound  \eqref{ub:f}. Thus it remains to analyze the nonlinear term $I_2^{n}$. Using  identity \eqref{trigidentity}, Fubini theorem and the change of variable in $x\mapsto x+x'$, we rewrite $I_2^n$ as
$$
I_2^{n}=\frac{1}{4\pi}\int_0^T\!\!\!\int_\T \tt_{n}(x',t)
\int_\T\p_x\vp(x\!+\!x',t)\ln\Big(2\sinh^2\Big(\frac{f_n(x\!+\!x',t)\!-\!f_n(x',t)}{2}\Big)\!+2\!\sin^2\big(\frac{x}2\big)\Big)dxdx'dt.
$$
Then we  split $I_2^{n}=J^{n}_1+J^{n}_2$ where
$$
J^{n}_1=\frac{1}{4\pi}\int_0^T\!\!\!\int_\T\tt_{n}(x',t)
\int_\T\p_x\vp(x\!+\!x',t)\ln\Big(1+\frac{\sinh^2\big(\frac{f_n(x+x',t)-f_n(x',t)}{2}\big)}{\sin^2(\frac{x}{2})}\Big)dxdx'dt
$$
and
\begin{align*}
J^{n}_2&=\frac{1}{4\pi}\int_0^T\!\!\!\int_\T\tt_{n}(x',t)
\int_\T\p_x\vp(x\!+\!x',t)\ln\big(2\sin^2(\frac{x}{2})\big)dxdx'dt\\
&=-\frac{1}{4\pi}\int_0^T\!\!\!\int_\T\tt_{n}(x',t)
\int_\T\vp(x\!+\!x',t)\cot (\frac{x}{2})dxdx'dt\\
&=\mez\int_0^T\!\!\!\int_\T
\tt_{n}(x',t)H(\vp)(x',t)dx'dt.
\end{align*}
Here $H$ denotes the Hilbert transform  on $\T$. The weak* convergence of $\tt_n$ in \eqref{weak*} implies the convergence for $J^n_2$. The estimate
$$
\left|\frac{\sinh\big(\frac{f_n(x+x',t)-f_n(x',t)}{2}\big)}{\sin(\frac{x}{2})}\right|\leq \frac{\pi}{2}\|f'_0\|_{L^\infty}\cosh(\|f_0\|_{L^\infty})
$$
allows to bound the $\ln$ function in $J_1^{n}$. Then the convergence for $J_1^n$ follows along the same lines as for $I_1^{n}$. Therefore, taking $n\to\infty$ in \eqref{fepsweak} yields
\begin{align*}
\begin{split}
\int_0^T\!\!\!\int_\T\p_t\vp(x,t)&f(x,t)dxdt+\int_\T\vp(x,0)f_0(x)dx=\\
&\frac{1}{4\pi}\int_0^T\!\!\!\int_\T
\p_x\vp(x,t)\int_\T\ln\Big(\cosh(f(x,t)\!-\!f(x',t))\!-\!\cos(x\!-\!x')\Big)\tt(x',t)dx'dxdt.
\end{split}
\end{align*}
Finally, the regularity of $f$ and $\theta$ together with integration by parts provide the evolution equation \eqref{reform:f}  for $f$  in the $L^\infty_tL^2_x$ sense. The proof of Theorem \ref{theo:main} is complete.
\appendix
\section{Traces for homogeneous Sobolev spaces}
Let $f\in \Lip(\T^d)$ and recall the notation $\Omega_f=\{(x, y)\in \T^d\times \Rr: y<f(x)\}$. Let $\dot H^1(\Omega_f)$ denote the homogeneous Sobolev space defined by \eqref{def:dotH1}. We prove the trace and lifting theorems for functions in $\dot H^1(\Omega_f)$.
\begin{theo}\label{theo:trace}
There exists a unique bounded linear trace operator $\Tr: \dot H^1(\Omega_f)\to \dot H^\mez(\T^d)$ such that the following hold.
\begin{itemize}
\item[(1)] $\Tr(u)(x)=u(x, f(x))$ for all $u\in \dot H^1(\Omega_f)\cap C(\overline{\Omega_f})$.
\item[(2)] For some constant $C=C(d)$,
\bq\label{trace:ineq}
\| \Tr(u)\|_{\dot H^\mez(\T^d)}\le C(1+\| f\|_{\Lip(\T^d)})\| u\|_{\dot H^1(\Omega_f)}.
\eq
\end{itemize}
\end{theo}
\begin{proof}
By the density of $C^\infty_c(\overline{\Omega_f})$ in $\dot H^1(\Omega_f)$, it suffices to prove the estimate \eqref{trace:ineq} for $u\in C^\infty_c(\overline{\Omega_f})$.

{\it Step 1:} $f=0$, i.e. $\Omega_f=\Omega_0=\T^d\times \Rr_-$. We extend $u$ to $\T^d\times \Rr$ by even reflection. That is, $u(x, y)=u(x, -y)$ for $y>0$. Clearly, $\| u\|_{\dot H^1(\T^d\times \Rr)}\le 2\| u\|_{\dot H^1(\T^d\times \Rr_-)}$. We have
$\Tr(u)(x)=u(x, 0)$ and
\[
\wh{\Tr(u)}(k)=\frac{1}{2\pi}\int_{\Rr}\wh{u}(k, \xi)d\xi,\quad \wh{u}(k, \xi)=\int_{\T^d\times \Rr}u(x, y)e^{-ixk-iy\xi}dydx.
\]
By the Cauchy-Schwartz inequality, for $k\in \Zz^d\setminus\{0\}$,
\[
\left|\int_{\Rr}\wh{u}(k, \xi)d\xi\right|^2\le \int_{\Rr}(k^2+\xi^2)|\wh{u}(k, \xi)|^2d\xi\int_{\Rr}\frac{1}{k^2+\xi^2}d\xi\le \frac{\pi}{|k|}\int_{\Rr}(k^2+\xi^2)|\wh{u}(k, \xi)|^2d\xi,
\]
so
\begin{align*}
\| \Tr(u)\|^2_{\dot H^\mez(\T^d)}&= \sum_{k\in \Zz^d\setminus\{0\}}|k||\wh{\Tr(u)}(k)|^2\le\frac{1}{4\pi}  \sum_{k\in \Zz^d}\int_{\Rr}(k^2+\xi^2)|\wh{u}(k, \xi)|^2d\xi\\
&=C(d)\| \na_{x, y}u\|_{L^2(\T^d\times \Rr)}^2\le 2C(d)\| u\|_{\dot H^1(\T^d\times \Rr_-)}^2.
\end{align*}
{\it Step 2:} $f\in \Lip(\T^d)$.  Setting $v(x, z)=u(x, z+f(x))$ for $(x, z)\in \T^d\times \Rr_-$, we have $\Tr(u)(x)=v(x, 0)$ and $v\in \dot H^1(\T^d\times \Rr_-)$ with
\[
\| v\|_{\dot H^1(\T^d\times \Rr_-)}\le (1+\| f\|_{\Lip(\T^d)})\| u\|_{\dot H^1(\Omega_f)}.
\]
Using Step 1 we deduce that
\[
\| \Tr(u)\|_{\dot H^\mez(\T^d)}=\| v(\cdot, 0)\|_{\dot H^\mez(\T^d)}\le C(d)\| v\|_{H^1(\T^d\times \Rr_-)}\le C(d)(1+\| f\|_{\Lip(\T^d)})\| u\|_{\dot H^1(\Omega_f)},
\]
thereby finishing the proof.
\end{proof}
\begin{theo}\label{theo:lifting}
For every $g\in \dot H^\mez(\T^d)$, there exists $u\in \dot H^1(\Omega_f)$ such that $\Tr(u)=g$ and for some universal constant $C=C(d)$,
\bq
\| u\|_{\dot H^1(\Omega_f)}\le C(1+\| f\|_{\Lip(\T^d)})\| g\|_{\dot H^\mez(\T^d)}.
\eq
Moreover, we can choose $u(x, y)=u_\sharp(x, y-f(x))$, where $u_\sharp(x, Y)=e^{Y|D_x|}g(x)$ for $(x, Y)\in \T^d\times \Rr_-$.
\end{theo}
\begin{proof}
Clearly, with the given function $u$ we have $u(x, f(x))=u_\sharp (x, 0)=g(x)$.
Since $\p_{x_j} u_\sharp(x, Y)=e^{Y|D_x|}\p_{x_j}g(x)$ and $\p_Y u_\sharp(x, Y)=e^{Y|D_x|}|D_x|g(x)$, using the Plancherel Theorem for the $x$ variable  one can easily prove that
\[
\sum_{j=1}^d\| \p_{x_j} u_\sharp\|^2_{L^2(\T^d\times \Rr_-)}= C_1(d)\| g\|_{\dot H^\mez(\T^d)}^2,\quad \| \p_Y u_\sharp\|_{L^2(\T^d\times \Rr_-)}=C_2(d)\| g\|_{\dot H^\mez(\T^d)}.
\]
 Consequently,
 \[
 \| u\|_{\dot H^1(\Omega_f)}\le  2(1+\| f\|_{\Lip(\T^d)})\| u_\sharp\|_{\dot H^1(\T^d\times \Rr_-)}\le C(d)(1+\| f\|_{\Lip(\T^d)})\| g\|_{\dot H^\mez(\T^d)}
 \]
 which finishes the proof.
\end{proof}
\section{The Dirichlet-Neumann operator for the disk}
\begin{prop}\label{prop:disk}
Let $g\in C(\T)$ be  $C^{1, \alpha}$ at $x\in \T$. Let $u\in C(\ol{B_1})$ be the unique solution of the Dirichlet problem
\[
\Delta u=0\quad\text{in~} B_1(0),\quad u\vert_{\p B_1(0)}=g,
\]
where  $e^{ix}\equiv (\cos x, \sin x)$.  Then the normal derivative
\[
\p_n u(e^{ix}):=\lim_{r\to 1^-} e^{ix}\cdot \na u(re^{ix})
\]
is well-defined and
\bq\label{normal:disk}
\p_n u(e^{ix})=-\frac{1}{8\pi}\int_{-\pi}^{\pi}\frac{g(x+x')+g(x-x')-2g(x)}{\sin^2(\frac{x'}{2})}dx'.
\eq
\end{prop}
\begin{proof}
Since $g$ is $C^{1, \alpha}$ at $x$, there exist $M>0$ and $\gamma>0$ such that
\bq\label{def:Holder}
|g(x+x')-g(x)-\p g(x)x'|\le M|x'|^{1+\alpha}\quad\forall |x'|<\gamma.
\eq
Since $g\in C(\T)$, $u$ is given by the Poisson formula
 \[
 u(z)=\int_{\p B_1(0)}K(z, z')g(z')dS(z'),\quad
 K(z, z')=\frac{1-|z|^2}{2\pi}\frac{1}{|z-z'|^2}
 \]
 for $z\in B_1(0)$.  Consequently,
 \bq\label{normal:diskb}
 z\cdot\na u(z)=\int_{\p B_1(0)}z\cdot\na_zK(z, z')g(z')dS(z'),
 \eq
 where
 \bq\label{dK}
 z\cdot\na_zK(z, z')=-\frac{|z|^2}{\pi |z-z'|^2}-(1-|z|^2)\frac{z\cdot(z-z')}{\pi |z-z'|^4}.
 \eq
When $g\equiv 1$ we have $u\equiv 1$ and thus \eqref{normal:diskb} yields
 \bq\label{intdK}
 \int_{\p B_1(0)}z\cdot\na_zK(z, z')dS(z')=0,\quad z\in B_1(0).
 \eq
Denoting $z_*=e^{ix}$ and using \eqref{intdK}, we write
  \bq\label{znabu}
  \begin{aligned}
   z\cdot\na u(z)&=\int_{\p B_1(0)}z\cdot\na_zK(z, z')[g(z')-g(z_*)]dS(z')\\
   &=\frac{|z|^2}{\pi}\int_{\p B_1(0)}\frac{1}{|z-z'|^2}[g(z_*)-g(z')]dS(z')\\
   &\qquad+\frac{1-|z|^2}{\pi}\int_{\p B_1(0)}\frac{z\cdot(z-z')}{|z-z'|^4}[g(z_*)-g(z')]dS(z')\\
   &=J_1+J_2,\quad z\in B_1(0).
   \end{aligned}
   \eq
 {\bf Step 1.} In this step we shall prove that
   \bq\label{claim:J2}
   \lim_{r\to 1^{-}}J_2=0.
   \eq
  Indeed,  we first use the equality $r^2+1-2r\cos x=(r-\cos x)^2+\sin^2 x$ to have
  \[
  \begin{aligned}
  J_2&=\frac{1-r^2}{\pi}\int_{-\pi}^{\pi}\frac{r^2-r\cos(x-x')}{(r^2+1-2r\cos(x-x'))^2}[g(x)-g(x')]dx'\\
  &= \frac{(1-r^2)r}{2\pi}\int_{-\pi}^{\pi}\frac{r-1+2\sin^2\frac{x'}{2}}{[(r-\cos x')^2+\sin^2 x']^2}[2g(x)-g(x+x')-g(x-x')]dx'.
  \end{aligned}
  \]
In the integral defining $J_2$, it suffices to consider $x'$ near  $0$  because away from $0$ the integrand is uniformly bounded as $r\to 1$.  For  sufficiently small $a\in (0, \gamma)$, let  $J_2^0$ be the contribution of $|x'|<a$. In view of \eqref{def:Holder}, we have
\begin{align*}
  |J_2^0|&\le  \frac{(1-r^2)r}{2\pi}\int_{-a}^a\frac{(1-r)M|x'|^{1+\alpha}}{[(r-\cos x')^2+\sin^2 x']^2}dx'\\
  &\quad + \frac{(1-r^2)r}{2\pi}\int_{-a}^a\frac{2(\sin^2\frac{x'}{2})M|x'|^{1+\alpha}}{[(r-\cos x')^2+\sin^2 x']^2}dx':={\rm I}+{\rm II}.
\end{align*}
  Using the inequalities
  \[
  |\sin x'|\le |x'|\quad\forall x'\quad\text{and}\quad |\sin x'|\ge c|x'|\quad\forall x'\in (-a, a)
  \]
 we estimate
  \[
  |{\rm II}|\le \frac{(1-r^2)r}{2\pi}\int_{-a}^a\frac{2(\sin^2\frac{x'}{2})M|x'|^{1+\alpha}}{\sin^4 x'}dx'\le \frac{(1-r^2)rM}{4\pi c^4}\int_{-a}^a\frac{1}{|x'|^{1-\alpha}}dx'\le \frac{(1-r^2)rMa^\alpha}{2\pi\alpha c^4}.
  \]
  Hence, ${\rm II}\to 0$ as $r\to 1^{{\color{blue}-}}$.  On the other hand,
  \bq\label{est:J2:I}
  \begin{aligned}
  |{\rm I}|&\le  \frac{(1-r^2)r}{2\pi} \int_{-a}^a\frac{(1-r)M|x'|^{1+\alpha}}{[(r-\cos x')^2+\sin^2 x']\sin^2 x'}dx'\\
  &\le   \frac{(1-r^2)r}{2\pi}\frac{M}{c^2}\int_{-a}^a\frac{1-r}{[(r-\cos x')^2+\sin^2 x']|x'|^{1-\alpha}}dx'.
  \end{aligned}
  \eq
  Writing $(r-\cos x')^2+\sin^2 x'=(1-r)^2+4r\sin^2 \frac{x'}{2}$, we deduce that
  \[
  |{\rm I}|\le\frac{(1-r^2)r}{2\pi}  \frac{M}{c^2}\int_{-a}^a\frac{1-r}{[(1-r)^2+rc^2|x'|^2]|x'|^{1-\alpha}}dx'\le C(1-r)^\alpha\frac{M(1+r)r}{2c^2(\sqrt{r}c)^\alpha}.
   \]
   This concludes the proof of \eqref{claim:J2}.

   {\bf Step 2.} Combining \eqref{znabu} and \eqref{claim:J2} yields
   \[
   \begin{aligned}
   \p_nu(e^{ix})= \lim_{r\to 1^{-}}z\cdot\na u(z)&=\lim_{r\to 1^{-}} J_1(z)= \lim_{r\to 1^{-}}\frac{1}{\pi}\int_{\p B_1(0)}\frac{1}{|z-z'|^2}[g(z)-g(z')]dS(z')\\
    &= \lim_{r\to 1^{-}}\frac{1}{2\pi}\int_{-\pi}^{\pi}\frac{2g(x)-g(x+x')-g(x-x')}{(r-\cos x')^2+\sin^2 x'}dx'.
   \end{aligned}
\]
In order to obtain \eqref{normal:disk} we compute
\[
\begin{aligned}
L:&=\int_{-\pi}^{\pi}\frac{2g(x)-g(x+x')-g(x-x')}{(r-\cos x')^2+\sin^2 x'}dx'-\frac{1}{4}\int_0^{2\pi}\frac{2g(x)-g(x+x')-g(x-x')}{\sin^2 \frac{x'}{2}}dx'\\
&=\frac{1-r}{4}\int_{-\pi}^{\pi}\frac{2g(x)-g(x+x')-g(x-x')}{[(r-\cos x')^2+\sin^2 x']\sin^2 \frac{x'}{2}}(1+r-2\cos x')dx'.
\end{aligned}
\]
 As before it suffices to consider  $|x'|<a$ for small $a$. Writing $1+r-2\cos x'=(r-1)+4\sin^2(\frac{x'}{2})$ we split
\[
\begin{aligned}
L&=(1-r)\int_{-a}^a\frac{2g(x)-g(x+x')-g(x-x')}{(r-\cos x')^2+\sin^2 x'}dx'\\
&\quad+\frac{1-r}{4}\int_{-a}^a\frac{2g(x)-g(x+x')-g(x-x')}{[(r-\cos x')^2+\sin^2 x']\sin^2 \frac{x'}{2}}(r-1)dx':=L_1+L_2.
\end{aligned}
\]
Clearly,
\[
\begin{aligned}
|L_1|&\le  (1-r)M\int_{-a}^a\frac{|x'|^{1+\alpha}}{\sin^2 x'}dx'\le (1-r)\frac{2Ma^\alpha}{c^2\alpha}.
\end{aligned}
\]
On the other hand, $L_2$ can be treated as in \eqref{est:J2:I}. The proof is complete.
\end{proof}
 \section{Construction of \texorpdfstring{$\psi_r$}{}}\label{appendix:psir}
 Let  $f\in W^{1, \infty}(\T\times (0, T))$ be $C^{1, 1}$ at $X_0=(x_0, t_0)\in \T\times (0, T)$. Recall the definitions  \eqref{parabola:psibar} and \eqref{parabola:psi} of the parabolas $\overline{\psi}$ and $\psi$ tangent to the graph $\{y=f(x, t)\}$ from above at $(x_0, f(x_0, t_0))$.   We construct a family of functions $\psi_r\in  C^\infty(\T\times \Rr)$, $r\in (0, r_0)$  satisfying the properties in \eqref{properties:psir}. Assume without loss of generality that $f(X_0)=0$ so that
 \[
 \overline{\psi}(X)=\na_Xf(X_0)(X-X_0)+\frac{C}{2}|X-X_0|^2,\quad \psi(X)=\na_Xf(X_0)(X-X_0)+C|X-X_0|^2
 \]
and  $f(X)\le \overline{\psi}(X)$ for $|X-X_0|<r_0<\frac{1}{10}\min\{\pi, t_0, T-t_0\}$. We extend $f$ to $0$ outside $(0, T)$ and set $M=\| f\|_{\Lip(\T\times (0, T))}$.

Let $\Gamma_r$ be an approximation of the identity in $\Rr^2$ as $r\to 0$ and denote $h^{(r)}=h*\Gamma_r$ for any $h:\Rr^2\to \Rr$. For any $U\subset \Rr^2$, 
$$
\| h^{(r)}-h\|_{L^\infty(U)}\le r\|h\|_{\Lip(U+B_r(0))}.
$$

For $r\in (0,  r_0)$, let $g_r=f^{(r^2)}+(M+2)r^2$, so that $g_r-f\ge 2r^2$. By increasing $C$ if necessary, we can assume $C>4(M+2)$. Then for $|X-X_0|\in [r, r_0)$, we have
\bq\label{gr:psi}\begin{aligned}
\psi(x)-g_r(x)&\ge \na_Xf(X_0)(X-X_0)+C|X-X_0|^2-\big[f(X)+Mr^2+(M+2)r^2\big]\\
&\ge \overline{\psi}(X)-f(X)+\frac{C}{2}r^2-2(M+1)r^2\ge 2r^2.
\end{aligned}
\eq
Let $F(a, b)=\min\{a, b\}$. Note that $\|F\|_{\Lip(\Rr^2)}=1$ and $F^{(\delta)}(a, b)=a$ if $a<b-\delta$. In addition, $F^{(\delta)}(a, b)\le a^{(\delta)}=a$ and likewise $F^{(\delta)}(a, b)\le b$.

Define $\psi_r=F^{(r^2)}(g_r, \psi)$ in $B_r(X_0)$ and $\psi_r=g_r$ in $\Rr^2\setminus B_r(X_0)$. In fact, \eqref{gr:psi} shows that  $\psi(X)-g_r(X)\ge r^2$ in an open neighborhood of  $\p B_r(X_0)$, so that $\psi_r=g_r$ in the same neighborhood. It follows that $\psi_r$ is a smooth function on $\Rr^2$.  We have $\psi_r\le \psi$ in $B_r(X_0)$ and in view of \eqref{gr:psi}, $\psi_r(X)\le \psi(X)$ for $|X-X_0|\in [r, r_0)$. Consequently, $\psi_r\le \psi$ in $B_{r_0}(X_0)$. In $B_r(X_0)$, when $\psi_r(X)<\psi(X)$, we must have $\psi(X)\ge g_r(X)-r^2$, whence
\begin{align*}
&\psi_r(X)\ge \min\{g_r(X), \psi(X)\}-r^2\ge g_r(X)- 2r^2\ge f(X).
\end{align*}
Therefore, $\psi_r\ge f$ in $B_r(X_0)$.

Let $\delta \in (0, r_0^2)$ be arbitrarily small. Clearly,
\[
\lim_{r\to 0} f^{(r^2)}= f\quad\text{in}\quad C(\T\times [\delta, T-\delta]),\quad \sup_{r^2\in (0, \delta)}\| f^{(r^2)}\|_{\Lip(\T\times [\delta, T-\delta])}\le M.
\]
This together with the definition of $\psi_r$ yields
\[
\lim_{r\to 0} \psi_r= f\quad\text{in}\quad C(I\times [\delta, T-\delta]),\quad \sup_{r^2\in (0, \delta)}\| \psi_r\|_{\Lip(I\times [\delta, T-\delta])}<\infty,
\]
where $I=[x_0-\pi, x_0+\pi]$ and $X_0=(x_0, t_0)$.

We have proved that $\psi_r$ satisfies all the properties in \eqref{properties:psir} but with $\T$ replaced by $I$. To finish the proof, we note that $x_0$ is the midpoint of $I$ and $\psi_r=g_r$ for $|X-X_0|>r$. Therefore, since $g_r\in C^\infty(\T\times \Rr)$ and $I$ has length $2\pi$, we can periodize $\psi_r$ as
\[
\psi_r(x+k2\pi, t)=\psi_r(x, t)\quad\forall (x, t)\in I\times \Rr,\, \forall k\in \Zz
\]
and obtain $\psi_r\in C^\infty(\T\times \Rr)$.

\vspace{.1in}
\noindent{\bf{Acknowledgment.}} 
 H. Dong was partially supported by the Simons Foundation, grant no. 709545. F. Gancedo was partially supported by the ERC
through the Starting Grant project H2020-EU.1.1.-639227 and by the grant EUR2020-112271 (Spain).
 H. Q. Nguyen was partially supported by NSF grant DMS-190777. We would like to thank Nestor Guillen for insightful discussions on viscosity solutions for nonlocal equations.

\bibliographystyle{plain}
\bibliographystyle{amsplain}
\bibliography{references}

\begin{thebibliography}{10}

\bibitem{Alazard2014}
T.~Alazard, N.~Burq, and C.~Zuily.
\newblock On the {C}auchy problem for gravity water waves.
\newblock {\em Invent. Math.}, 198(1):71--163, 2014.

\bibitem{AlazardLazard2019}
Thomas Alazard and Omar Lazar.
\newblock Paralinearization of the {M}uskat equation and application to the
  {C}auchy problem.
\newblock {\em Arch. Ration. Mech. Anal.}, 237(2):545--583, 2020.

\bibitem{AlazardOneFluid2019}
Thomas Alazard, Nicolas Meunier, and Didier Smets.
\newblock Lyapunov functions, identities and the {C}auchy problem for the
  {H}ele-{S}haw equation.
\newblock {\em Comm. Math. Phys.}, 377(2):1421--1459, 2020.

\bibitem{AlazardHung32020}
Thomas Alazard and Quoc-Hung Nguyen.
\newblock Endpoint {S}obolev theory for the {M}uskat equation.
\newblock {\em Preprint arXiv:2010.06915}, 2020.

\bibitem{AlazardHung22020}
Thomas Alazard and Quoc-Hung Nguyen.
\newblock On the {C}auchy problem for the {M}uskat equation. ii: Critical
  initial data.
\newblock {\em Preprint arXiv:2009.08442}, 2020.

\bibitem{Ambrose2007}
David~M. Ambrose.
\newblock Well-posedness of two-phase {D}arcy flow in 3{D}.
\newblock {\em Quart. Appl. Math.}, 65(1):189--203, 2007.

\bibitem{Ambrose2014}
David~M. Ambrose.
\newblock The zero surface tension limit of two-dimensional interfacial {D}arcy
  flow.
\newblock {\em {J}. {M}ath. {F}luid {M}ech.}, 16(1):105--143, 2014.

\bibitem{ACS1996}
I.~Athanasopoulos, L.~Caffarelli, and S.~Salsa.
\newblock Regularity of the free boundary in parabolic phase-transition
  problems.
\newblock {\em Acta Math.}, 176(2):245--282, 1996.

\bibitem{CaffarelliSalsa2005}
Luis Caffarelli and Sandro Salsa.
\newblock {\em A geometric approach to free boundary problems}, volume~68 of
  {\em Graduate Studies in Mathematics}.
\newblock American Mathematical Society, Providence, RI, 2005.

\bibitem{CaffarelliSilvestre2009}
Luis Caffarelli and Luis Silvestre.
\newblock Regularity theory for fully nonlinear integro-differential equations.
\newblock {\em Comm. Pure Appl. Math.}, 62(5):597--638, 2009.

\bibitem{Caffarelli3}
Luis~A. Caffarelli.
\newblock A {H}arnack inequality approach to the regularity of free boundaries.
  {III}. {E}xistence theory, compactness, and dependence on {$X$}.
\newblock {\em Ann. Scuola Norm. Sup. Pisa Cl. Sci. (4)}, 15(4):583--602
  (1989), 1988.

\bibitem{Cameron2019}
Stephen Cameron.
\newblock Global well-posedness for the two-dimensional {M}uskat problem with
  slope less than 1.
\newblock {\em Anal. PDE}, 12(4):997--1022, 2019.

\bibitem{Cameron2020}
Stephen Cameron.
\newblock Global wellposedness for the 3{D} {M}uskat problem with medium size
  slope.
\newblock {\em Preprint arXiv:2002.00508}, 2020.

\bibitem{CFM2019}
\'{A}. Castro, D.~Faraco, and F.~Mengual.
\newblock Degraded mixing solutions for the {M}uskat problem.
\newblock {\em Calc. Var. Partial Differential Equations}, 58(2):Art. 58, 29,
  2019.

\bibitem{CFM2021}
\'{A}. Castro, D.~Faraco, and F.~Mengual.
\newblock Localized mixing zone for {M}uskat bubbles and turned interfaces.
\newblock {\em Preprint arXiv:2102.07451}, 2021.

\bibitem{CCFG2013}
\'{A}ngel Castro, Diego C\'{o}rdoba, Charles Fefferman, and Francisco Gancedo.
\newblock Breakdown of smoothness for the {M}uskat problem.
\newblock {\em Arch. Ration. Mech. Anal.}, 208(3):805--909, 2013.

\bibitem{CCFG2016}
Angel Castro, Diego C\'{o}rdoba, Charles Fefferman, and Francisco Gancedo.
\newblock Splash singularities for the one-phase {M}uskat problem in stable
  regimes.
\newblock {\em Arch. Ration. Mech. Anal.}, 222(1):213--243, 2016.

\bibitem{CCFGL-F2012}
\'{A}ngel Castro, Diego C\'{o}rdoba, Charles Fefferman, Francisco Gancedo, and
  Mar\'{i}a L\'{o}pez-Fern\'{a}ndez.
\newblock Rayleigh-{T}aylor breakdown for the {M}uskat problem with
  applications to water waves.
\newblock {\em Ann. of Math. (2)}, 175(2):909--948, 2012.

\bibitem{Guillen2019}
H\'{e}ctor~A. Chang-Lara, Nestor Guillen, and Russell~W. Schwab.
\newblock Some free boundary problems recast as nonlocal parabolic equations.
\newblock {\em Nonlinear Anal.}, 189:11538, 60, 2019.

\bibitem{Chen1993}
Xinfu Chen.
\newblock The {H}ele-{S}haw problem and area-preserving curve-shortening
  motions.
\newblock {\em Arch. Rational Mech. Anal.}, 123(2):117--151, 1993.

\bibitem{CG-BS2016}
C.~H.~Arthur Cheng, Rafael Granero-Belinch\'{o}n, and Steve Shkoller.
\newblock Well-posedness of the {M}uskat problem with {$H^2$} initial data.
\newblock {\em Adv. Math.}, 286:32--104, 2016.

\bibitem{ChoiJerisonKim2007}
Sunhi Choi, David Jerison, and Inwon Kim.
\newblock Regularity for the one-phase {H}ele-{S}haw problem from a {L}ipschitz
  initial surface.
\newblock {\em Amer. J. Math.}, 129(2):527--582, 2007.

\bibitem{ChoiJerisonKim2009}
Sunhi Choi, David Jerison, and Inwon Kim.
\newblock Local regularization of the one-phase {H}ele-{S}haw flow.
\newblock {\em Indiana Univ. Math. J.}, 58(6):2765--2804, 2009.

\bibitem{ConstantinPugh1993}
P.~Constantin and M.~Pugh.
\newblock Global solutions for small data to the {H}ele-{S}haw problem.
\newblock {\em Nonlinearity}, 6(3):393--415, 1993.

\bibitem{CCGR-PS2016}
Peter Constantin, Diego C\'{o}rdoba, Francisco Gancedo, Luis
  Rodr\'{i}guez-Piazza, and Robert~M. Strain.
\newblock On the {M}uskat problem: global in time results in 2{D} and 3{D}.
\newblock {\em Amer. J. Math.}, 138(6):1455--1494, 2016.

\bibitem{CCGS2013}
Peter Constantin, Diego C\'{o}rdoba, Francisco Gancedo, and Robert~M. Strain.
\newblock On the global existence for the {M}uskat problem.
\newblock {\em J. Eur. Math. Soc. (JEMS)}, 15(1):201--227, 2013.

\bibitem{CGSV2017}
Peter Constantin, Francisco Gancedo, Roman Shvydkoy, and Vlad Vicol.
\newblock Global regularity for 2{D} {M}uskat equations with finite slope.
\newblock {\em Ann. Inst. H. Poincar\'{e} Anal. Non Lin\'{e}aire},
  34(4):1041--1074, 2017.

\bibitem{CCG2011}
Antonio C\'{o}rdoba, Diego C\'{o}rdoba, and Francisco Gancedo.
\newblock Interface evolution: the {H}ele-{S}haw and {M}uskat problems.
\newblock {\em Ann. of Math. (2)}, 173(1):477--542, 2011.

\bibitem{CCG2013}
Antonio C\'{o}rdoba, Diego C\'{o}rdoba, and Francisco Gancedo.
\newblock Porous media: the {M}uskat problem in three dimensions.
\newblock {\em Anal. PDE}, 6(2):447--497, 2013.

\bibitem{CordobaGancedo2007}
Diego C\'{o}rdoba and Francisco Gancedo.
\newblock Contour dynamics of incompressible 3-{D} fluids in a porous medium
  with different densities.
\newblock {\em Comm. Math. Phys.}, 273(2):445--471, 2007.

\bibitem{CordobaGancedo2010}
Diego C\'{o}rdoba and Francisco Gancedo.
\newblock Absence of squirt singularities for the multi-phase {M}uskat problem.
\newblock {\em Comm. Math. Phys.}, 299(2):561--575, 2010.

\bibitem{CG-SZ2017}
Diego C\'{o}rdoba, Javier G\'{o}mez-Serrano, and Andrej Zlato\v{s}.
\newblock A note on stability shifting for the {M}uskat problem, {II}: {F}rom
  stable to unstable and back to stable.
\newblock {\em Anal. PDE}, 10(2):367--378, 2017.

\bibitem{CordobaLazar2018}
Diego {C}ordoba and {O}mar {L}azar.
\newblock Global well-posedness for the 2{D} stable {M}uskat problem in
  ${H}^{3/2}$.
\newblock {\em Preprint arXiv:1803.07528}, 2018.

\bibitem{CordobaPernas-Castano2017}
Diego C\'{o}rdoba and Tania Pernas-Casta\~{n}o.
\newblock Non-splat singularity for the one-phase {M}uskat problem.
\newblock {\em Trans. Amer. Math. Soc.}, 369(1):711--754, 2017.

\bibitem{CrandallIshiiLion}
Michael~G. Crandall, Hitoshi Ishii, and Pierre-Louis Lions.
\newblock User's guide to viscosity solutions of second order partial
  differential equations.
\newblock {\em Bull. Amer. Math. Soc. (N.S.)}, 27(1):1--67, 1992.

\bibitem{CrandallLions}
Michael~G. Crandall and Pierre-Louis Lions.
\newblock Viscosity solutions of {H}amilton-{J}acobi equations.
\newblock {\em Trans. Amer. Math. Soc.}, 277(1):1--42, 1983.

\bibitem{Darcy1856}
Henry {D}arcy.
\newblock Les {F}ontaines {P}ubliques de la {V}ille de {D}ijon.
\newblock {\em Dalmont, {P}aris}, 1856.

\bibitem{PoyferreNguyen2017}
Thibault de~Poyferr\'{e} and Quang-Huy Nguyen.
\newblock A paradifferential reduction for the gravity-capillary waves system
  at low regularity and applications.
\newblock {\em Bull. Soc. Math. France}, 145(4):643--710, 2017.

\bibitem{Lin2017}
Fan Deng, Zhen Lei, and Fanghua Lin.
\newblock On the two-dimensional {M}uskat problem with monotone large initial
  data.
\newblock {\em Comm. Pure Appl. Math.}, 70(6):1115--1145, 2017.

\bibitem{DongEscauriazaKim2018}
Hongjie Dong, Luis Escauriaza, and Seick Kim.
\newblock On {$C^1$}, {$C^2$}, and weak type-{$(1,1)$} estimates for linear
  elliptic operators: part {II}.
\newblock {\em Math. Ann.}, 370(1-2):447--489, 2018.

\bibitem{DongXiong2015}
Hongjie Dong and Jingang Xiong.
\newblock Boundary gradient estimates for parabolic and elliptic systems from
  linear laminates.
\newblock {\em Int. Math. Res. Not. IMRN}, (17):7734--7756, 2015.

\bibitem{DuchonRobert1984}
Jean Duchon and Raoul Robert.
\newblock \'{E}volution d'une interface par capillarit\'{e} et diffusion de
  volume. {I}. {E}xistence locale en temps.
\newblock {\em Ann. Inst. H. Poincar\'{e} Anal. Non Lin\'{e}aire},
  1(5):361--378, 1984.

\bibitem{Elliott1981}
C.~M. Elliott and V.~Janovsk\'{y}.
\newblock A variational inequality approach to {H}ele-{S}haw flow with a moving
  boundary.
\newblock {\em Proc. Roy. Soc. Edinburgh Sect. A}, 88(1-2):93--107, 1981.

\bibitem{EscherMatioc2011}
Joachim Escher and Bogdan-Vasile Matioc.
\newblock On the parabolicity of the {M}uskat problem: well-posedness,
  fingering, and stability results.
\newblock {\em Z. Anal. Anwend.}, 30(2):193--218, 2011.

\bibitem{EscherSimonett1997}
Joachim Escher and Gieri Simonett.
\newblock Classical solutions for {H}ele-{S}haw models with surface tension.
\newblock {\em Adv. Differential Equations}, 2(4):619--642, 1997.

\bibitem{Fabes1978}
E.~B. Fabes, M.~Jodeit, Jr., and N.~M. Rivi\`ere.
\newblock Potential techniques for boundary value problems on
  {$C^{1}$}-domains.
\newblock {\em Acta Math.}, 141(3-4):165--186, 1978.

\bibitem{FlynnNguyen2020}
Patrick~T. Flynn and Huy~Q. Nguyen.
\newblock The vanishing surface tension limit of the {M}uskat problem.
\newblock {\em Preprint arXiv:2001.10473}, 2020.

\bibitem{Forster2018}
Clemens F\"{o}rster and L\'{a}szl\'{o} Sz\'{e}kelyhidi, Jr.
\newblock Piecewise constant subsolutions for the {M}uskat problem.
\newblock {\em Comm. Math. Phys.}, 363(3):1051--1080, 2018.

\bibitem{GG-JPS2019}
F.~Gancedo, E.~Garc\'{i}a-Ju\'{a}rez, N.~Patel, and R.~M. Strain.
\newblock On the {M}uskat problem with viscosity jump: {G}lobal in time
  results.
\newblock {\em Adv. Math.}, 345:552--597, 2019.

\bibitem{GG-JPS2019Bubble}
Francisco Gancedo, Eduardo Garc\'{i}a-Ju\'{a}rez, Neel Patel, and Robert
  Strain.
\newblock Global {R}egularity for {G}ravity {U}nstable {M}uskat {B}ubbles.
\newblock {\em Preprint arXiv:1902.02318}, 2019.

\bibitem{GG-BS2020}
Francisco Gancedo, Rafael Granero-Belinch\'{o}n, and Stefano Scrobogna.
\newblock Surface tension stabilization of the {R}ayleigh-{T}aylor instability
  for a fluid layer in a porous medium.
\newblock {\em Ann. Inst. H. Poincar\'{e} Anal. Non Lin\'{e}aire},
  37(6):1299--1343, 2020.

\bibitem{GancedoLazar2020}
Francisco Gancedo and Omar Lazar.
\newblock Global well-posedness for the 3{D} {M}uskat problem in the critical
  {S}obolev space.
\newblock {\em Preprint arXiv:2006.01787}, 2020.

\bibitem{GancedoStrain2014}
Francisco Gancedo and Robert~M. Strain.
\newblock Absence of splash singularities for surface quasi-geostrophic sharp
  fronts and the {M}uskat problem.
\newblock {\em Proc. Natl. Acad. Sci. USA}, 111(2):635--639, 2014.

\bibitem{Giaquinta1983}
Mariano Giaquinta.
\newblock {\em Multiple integrals in the calculus of variations and nonlinear
  elliptic systems}, volume 105 of {\em Annals of Mathematics Studies}.
\newblock Princeton University Press, Princeton, NJ, 1983.

\bibitem{OptimalityCondition}
Jean-Baptiste Hiriart-Urruty, Jean-Jacques Strodiot, and V.~Hien Nguyen.
\newblock Generalized {H}essian matrix and second-order optimality conditions
  for problems with {$C^{1,1}$} data.
\newblock {\em Appl. Math. Optim.}, 11(1):43--56, 1984.

\bibitem{JerisonKenig1995}
David Jerison and Carlos~E. Kenig.
\newblock The inhomogeneous {D}irichlet problem in {L}ipschitz domains.
\newblock {\em J. Funct. Anal.}, 130(1):161--219, 1995.

\bibitem{Kamenomostskaja1961}
S.~L. Kamenomostskaja.
\newblock On {S}tefan's problem.
\newblock {\em Mat. Sb. (N.S.)}, 53 (95):489--514, 1961.

\bibitem{Kim2003}
Inwon~C. Kim.
\newblock Uniqueness and existence results on the {H}ele-{S}haw and the
  {S}tefan problems.
\newblock {\em Arch. Ration. Mech. Anal.}, 168(4):299--328, 2003.

\bibitem{Kim2011}
Inwon~C. Kim and Norbert Po\v{z}\'{a}r.
\newblock Viscosity solutions for the two-phase {S}tefan problem.
\newblock {\em Comm. Partial Differential Equations}, 36(1):42--66, 2011.

\bibitem{Lions1969}
J.-L. Lions.
\newblock {\em Quelques m\'{e}thodes de r\'{e}solution des probl\`emes aux
  limites non lin\'{e}aires}.
\newblock Dunod; Gauthier-Villars, Paris, 1969.

\bibitem{Matioc2019}
Bogdan-Vasile Matioc.
\newblock The {M}uskat problem in two dimensions: equivalence of formulations,
  well-posedness, and regularity results.
\newblock {\em Anal. PDE}, 12(2):281--332, 2019.

\bibitem{Mengual2020}
Francisco Mengual.
\newblock H-principle for the 2{D} incompressible porous media equation with
  viscosity jump.
\newblock {\em Preprint arXiv:2004.03307}, 2020.

\bibitem{Muskat1937book}
M.~Muskat.
\newblock {\em The Flow of Homogeneous Fluids Through Porous Media}.
\newblock New York: McGraw-Hill, 1937.

\bibitem{HQNguyen2019}
Huy~Q. Nguyen.
\newblock On well-posedness of the {M}uskat problem with surface tension.
\newblock {\em Adv. Math.}, 374:107344, 35, 2020.

\bibitem{NguyenPausader2019}
Huy~Q. Nguyen and Beno\^{\i}t Pausader.
\newblock A paradifferential approach for well-posedness of the {M}uskat
  problem.
\newblock {\em Arch. Ration. Mech. Anal.}, 237(1):35--100, 2020.

\bibitem{SCH2004}
Michael Siegel, Russel~E. Caflisch, and Sam Howison.
\newblock Global existence, singular solutions, and ill-posedness for the
  {M}uskat problem.
\newblock {\em Comm. Pure Appl. Math.}, 57(10):1374--1411, 2004.

\bibitem{Silvestre2011}
Luis Silvestre.
\newblock On the differentiability of the solution to the {H}amilton-{J}acobi
  equation with critical fractional diffusion.
\newblock {\em Adv. Math.}, 226(2):2020--2039, 2011.

\bibitem{Verchota1984}
Gregory Verchota.
\newblock Layer potentials and regularity for the {D}irichlet problem for
  {L}aplace's equation in {L}ipschitz domains.
\newblock {\em J. Funct. Anal.}, 59(3):572--611, 1984.

\end{thebibliography}

\end{document}